\documentclass[ejsv2,noshowframe,preprint]{imsart}

\input{tex_setup/packages}
\startlocaldefs
\numberwithin{equation}{section}

\theoremstyle{plain}

\newaliascnt{claim}{axiom}

\aliascntresetthe{claim}
\newtheorem{theorem}{Theorem}[section]
\newaliascnt{lemma}{theorem}
\newtheorem{lemma}[lemma]{Lemma}
\aliascntresetthe{lemma}

\theoremstyle{definition}
\newaliascnt{definition}{theorem}
\newtheorem{definition}[definition]{Definition}
\aliascntresetthe{definition}

\newtheorem{remark}{Remark}
\newtheorem{asm}{Assumption}

\theoremstyle{remark}

\newcommand{\Crefpair}[3]{\crefname{#1}{#2}{#3}\Crefname{#1}{#2}{#3}}

\Crefpair{theorem}{Theorem}{Theorems}
\Crefpair{lemma}{Lemma}{Lemmas}
\Crefpair{definition}{Definition}{Definitions}
\Crefpair{remark}{Remark}{Remarks}
\Crefpair{asm}{Assumption}{Assumptions}
\Crefpair{axiom}{Axiom}{Axioms}
\Crefpair{claim}{Claim}{Claims}
\Crefpair{case}{Case}{Cases}
\Crefpair{section}{Section}{Sections}
\Crefpair{subsection}{Section}{Sections}
\Crefpair{subsubsection}{Section}{Sections}
\Crefpair{equation}{Equation}{Equations}

\creflabelformat{equation}{#2(#1)#3}
\crefrangeformat{equation}{Equations~#3(#1)#4--#5(#2)#6}
\Crefrangeformat{equation}{Equations~#3(#1)#4--#5(#2)#6}
\crefmultiformat{equation}
  {Equations~#2(#1)#3}
  { and~#2(#1)#3}
  {, #2(#1)#3}
  {, and~#2(#1)#3}
\Crefmultiformat{equation}
  {Equations~#2(#1)#3}
  { and~#2(#1)#3}
  {, #2(#1)#3}
  {, and~#2(#1)#3}

\crefformat{enumi}{#2(#1)#3}
\Crefformat{enumi}{#2(#1)#3}
\crefrangeformat{enumi}{#3(#1)#4--#5(#2)#6}
\Crefrangeformat{enumi}{#3(#1)#4--#5(#2)#6}
\crefmultiformat{enumi}
  {#2(#1)#3}
  { and~#2(#1)#3}
  {, #2(#1)#3}
  {, and~#2(#1)#3}
\Crefmultiformat{enumi}
  {#2(#1)#3}
  { and~#2(#1)#3}
  {, #2(#1)#3}
  {, and~#2(#1)#3}

\renewcommand{\hat}{\widehat}

\newcommand{\E}{\mathbb{E}}
\renewcommand{\P}{\mathbb{P}}
\DeclareMathOperator{\Var}{Var}
\DeclareMathOperator{\Cov}{Cov}
\DeclareMathOperator{\rk}{rk}
\newcommand{\given}{\;\middle|\;}
\newcommand{\pto}{\xrightarrow{\P}}

\newcommand{\wto}{\rightsquigarrow}
\newcommand{\defeq}{\coloneqq}

\newcommand{\dd}{\;\mathrm{d}}

\newcommand{\R}{\mathbb{R}}
\newcommand{\N}{\mathbb{N}}

\newcommand{\calI}{\mathcal{I}}

\newcommand{\calO}{\mathcal{O}}
\newcommand{\calX}{\mathcal{X}}
\newcommand{\calY}{\mathcal{Y}}

\DeclarePairedDelimiterX{\abs}[1]{\lvert}{\rvert}{\mathopen{}\thinspace #1 \thinspace\mathclose{}}
\DeclarePairedDelimiterX{\norm}[1]{\lVert}{\rVert}{\mathopen{}\thinspace #1 \thinspace\mathclose{}}
\DeclarePairedDelimiterX{\set}[1]{\lbrace}{\rbrace}{\mathopen{}\thinspace #1 \thinspace\mathclose{}}
\DeclarePairedDelimiterX{\ip}[1]{\langle}{\rangle}{\mathopen{}\thinspace #1 \thinspace\mathclose{}}
\DeclarePairedDelimiterX{\floor}[1]{\lfloor}{\rfloor}{\mathopen{}\thinspace #1 \thinspace\mathclose{}}
\DeclarePairedDelimiterX{\ceil}[1]{\lceil}{\rceil}{\mathopen{}\thinspace #1 \thinspace\mathclose{}}

\DeclarePairedDelimiterXPP{\Ex}[1]{\mathbb{E}}{\lbrack}{\rbrack}{}{%
  \renewcommand\given{\nonscript\:\delimsize\vert\nonscript\:\mathopen{}}%
  \mathopen{}#1\mathclose{}%
}
\DeclarePairedDelimiterXPP{\Exi}[2]{\mathbb{E}_{#1}}{\lbrack}{\rbrack}{}{%
  \renewcommand\given{\nonscript\:\delimsize\vert\nonscript\:\mathopen{}}%
  \mathopen{}#2\mathclose{}%
}
\DeclarePairedDelimiterXPP{\Pb}[1]{\P}{(}{)}{}{%
  \renewcommand\given{\nonscript\:\delimsize\vert\nonscript\:\mathopen{}}%
  \mathopen{}#1\mathclose{}%
}
\DeclarePairedDelimiterXPP{\Pbi}[2]{\P_{#1}}{(}{)}{}{%
  \renewcommand\given{\nonscript\:\delimsize\vert\nonscript\:\mathopen{}}%
  \mathopen{}#2\mathclose{}%
}
\DeclarePairedDelimiterXPP{\Varb}[1]{\Var}{(}{)}{}{%
  \mathopen{}#1\mathclose{}%
}
\DeclarePairedDelimiterXPP{\Varbi}[2]{\Var_{#1}}{(}{)}{}{%
  \mathopen{}#2\mathclose{}%
}
\DeclarePairedDelimiterXPP{\Covp}[1]{\Cov}{(}{)}{}{%
  \mathopen{}#1\mathclose{}%
}
\DeclarePairedDelimiterXPP{\Covpi}[2]{\Cov_{#1}}{(}{)}{}{%
  \mathopen{}#2\mathclose{}%
}

\DeclarePairedDelimiterXPP{\1}[1]{\mathds{1}}{\lbrace}{\rbrace}{}{%
  \mathopen{}\thinspace #1 \thinspace\mathclose{}%
}

\endlocaldefs

\begin{document}
\begin{frontmatter}
    \title{Jackknife Variance Estimation for H\'ajek-Dominated Generalized U-Statistics}
    \runtitle{Jackknifing Generalized U-Statistics}

    \begin{aug}
        \author{\fnms{Jakob R.}~\snm{Juergens}\ead[label=e1]{jrjuergens@wisc.edu}\orcid{0009-0003-6792-2472}}
        \address{Department of Economics,
            University of Wisconsin -- Madison\printead[presep={,\ }]{e1}}

    \end{aug}

    \begin{abstract}
        Valid uncertainty quantification for subsampling-based and randomized estimators often depends on variance estimators whose behavior is much less understood than that of the underlying point estimator.
        We prove ratio-consistency of the jackknife variance estimator, and certain delete-\(d\) variants, for a broad class of generalized U-statistics whose variance is asymptotically dominated by their H\'ajek projection and whose normalized first-projection squares satisfy a row-wise \(L^r\) weak law, with the classical fixed-order case recovered as a special instance.
        This projection-dominance plus square-LLN structure unifies and generalizes several criteria from the existing literature, clarifies when the simple nonparametric jackknife is theoretically justified in the generalized setting, and yields consistent variance estimation for the two-scale distributional nearest-neighbor regression estimator under substantially weaker conditions than previously required.
    \end{abstract}

    \begin{keyword}[class=MSC]
        \kwd[Primary ]{62E20}
        \kwd{62F40}
        \kwd[; secondary ]{62G08}
        \kwd{62G05}
    \end{keyword}

    \begin{keyword}
        \kwd{Jackknife}
        \kwd{Generalized U-Statistics}
        \kwd{U-Statistics}
        \kwd{Variance Estimation}
        \kwd{Hoeffding Decomposition}
        \kwd{H\'ajek Projection}
    \end{keyword}

\end{frontmatter}
\section{Introduction}

Variance estimation is often the bottleneck in turning modern subsampling-based estimators into usable inference.
Many estimators of current interest in statistics and econometrics, including random-subsample learners and localized nonparametric estimators, admit generalized U-statistic representations.
Their asymptotic behavior is increasingly well understood, but variance estimation for this class remains much less developed outside the classical fixed-order setting.
The jackknife is attractive because it is simple, generic, and computationally convenient, but theoretical guarantees in these generalized settings remain scarce.

We establish conditions under which the jackknife yields ratio-consistent variance estimates for generalized U-statistics.
The central message is that jackknife consistency is governed by two first-projection requirements rather than by the full combinatorial expansion of the statistic.
The statistic's variance must be asymptotically dominated by its first-order (H\'ajek) projection, and the normalized squares of that projection must obey a mild row-wise weak law.
The dominance condition keeps the higher-order Hoeffding terms asymptotically negligible.
The square weak law then ensures that the diagonal averaged-square term appearing in the delete-$d$ calculation tracks the same variance scale.
For the delete-1 jackknife, this second requirement is just the familiar law of large numbers for squared first-order influence terms; it becomes a separate condition here because the generalized kernel itself changes with $n$.
Together, the two requirements give a unified treatment of the ordinary jackknife and its delete-$d$ variants.
They extend classical jackknife results for U-statistics (e.g.,~\cite{arvesen_jackknifing_1969}) to the generalized setting without modifying the base estimator.
They also place the plain jackknife on the same conceptual footing as infinitesimal-jackknife-style procedures in generalized settings, beginning with Jaeckel's original Bell Laboratories memorandum~\cite{jaeckel_infinitesimal_1972} and later developments such as~\cite{efron_jackknife-after-bootstrap_1992}, while using milder conditions than many existing alternatives.

As an application, we revisit the Two-Scale Distributional Nearest-Neighbor (TDNN) regression estimator of~\cite{demirkaya_optimal_2024}.
Under mild regularity conditions, this estimator satisfies asymptotic H\'ajek dominance whenever the kernel order grows more slowly than the sample size, substantially weakening the original assumptions used to justify jackknife-based variance estimates.

\subsection*{Related Literature}
Our starting point is the classical theory of U-statistics initiated by~\cite{hoeffding_class_1948}; for general background, see~\cite{lee_u-statistics_2019}.
For variance estimation, fixed-order benchmarks include~\cite{arvesen_jackknifing_1969},~\cite{efron_jackknife_1981}, and~\cite{shao_general_1989} for jackknife methods, and~\cite{arcones_bootstrap_1992} for the bootstrap.
In a closely related fixed-order setting,~\cite{schucany_small_1989} study small-sample variance estimators for U-statistics and find the ordinary jackknife to compare favorably with direct term-by-term corrections.
These papers are the baseline against which our generalized results should be read.

A second strand studies U-statistics whose order grows with the sample size and related subsampling-based statistics with the same combinatorial structure.
An early reference is~\cite{frees_infinite_1989}, which develops large-sample properties for infinite-order U-statistics.
In the modern ensemble literature,~\cite{mentch_quantifying_2016} show that predictions from subsampled tree ensembles can be analyzed through a U-statistic representation, while~\cite{zhou_v-statistics_2021} develop a complementary V-statistic viewpoint for with-replacement ensembles.
Relatedly,~\cite{peng_rates_2022} use generalized U-statistics to obtain rates of convergence for random forests, and~\cite{demirkaya_optimal_2024} develop pointwise inference for the two-scale distributional nearest-neighbor estimator that motivates our application.

A third strand studies variance estimation when the kernel order is not small relative to the sample size.
For general U-statistics,~\cite{wang_variance_2014} propose an unbiased variance estimator together with a partition-resampling implementation, and~\cite{wang_pseudo-kernel_2017} extend that line with a pseudo-kernel construction aimed at large kernel sizes and cross-validation problems.
For ensemble predictors,~\cite{sexton_standard_2009} study jackknife and bootstrap standard errors for bagging and random forests,~\cite{wager_confidence_2014} develop jackknife and infinitesimal-jackknife variance estimators for bagged learners and random forests, and~\cite{wang_quantifying_2022} propose an unbiased variance estimator for subsampling-based ensembles under a U-statistic formulation.

Recent work also emphasizes that leading-term approximations can fail when the kernel order is large.
In that spirit,~\cite{xu_variance_2024} study variance estimation for random forests through what they call peak-region dominance and establish ratio consistency for their unbiased estimator,~\cite{peng_bias_2026} analyze the bias-consistency of the infinitesimal jackknife and extend the discussion to subsampling-based U-statistics, and~\cite{wang_new_2025} argue for looking beyond the leading term in the exact variance expansion through a dominant-region viewpoint.

Our contribution answers a different question.
Rather than designing a new estimator for a particular generalized U-statistic, we ask when the plain delete-$d$ jackknife is ratio-consistent.
The answer is a projection-level one: asymptotic H\'ajek dominance plus a row-wise weak law for squared first-order projections is enough to make the ordinary jackknife and its delete-$d$ variants work in both complete and incomplete Bernoulli-sampled settings.
This isolates the first-projection structure behind jackknife consistency, clarifies when the unmodified jackknife is theoretically justified in generalized settings, and yields the TDNN application under the substantially weaker growth condition $s_2 = o(n)$.
At the application level, this perspective is close to nearest-neighbor and forest views of localization, including the adaptive-nearest-neighbor interpretation of forests in~\cite{lin_random_2006} and the localized-weight perspective of generalized random forests in~\cite{athey_generalized_2019}, and more broadly to localized conditional moment estimation with DNN weights as the localizer.

\subsection*{Notation}
We use the following notation.
Let $[n] = \{1, \dotsc, n\}$.
Given a finite index set $\calI \subset \N$, we introduce the following notational conventions.
\begin{equation}
	L_{s}(\calI) = \set*{(l_1, \dotsc, l_s) \in \calI^{s} \given l_{1} < l_{2} < \dotsc < l_{s}}
	\quad \text{and} \quad
	L_{n,s} = L_s\left([n]\right)
\end{equation}
Write $D_{[n]}=(Z_1,\dotsc,Z_n)$ for an i.i.d.\ sample from $F_Z$ with associated measure $\mu_Z$, so $D_{[n]} \sim \bigotimes_{i=1}^n \mu_Z$.
A realization is denoted by $d_{[n]}$.
For $\ell \in L_{n,s}$, $D_{[n],-\ell}$ is the data set with indices in $\ell$ removed, and $D_\ell$ is the sub-sample indexed by $\ell$.
For a single deleted observation we write $D_{[n],-i}$.
We use analogous notation for covariates, responses, and realizations, e.g. $X_{[c]}=(X_1,\dotsc,X_c)$, $Y_{[c]}=(Y_1,\dotsc,Y_c)$, $x_{[n]}$, and $y_{[n]}$.
If two index vectors $\ell$ and $\iota$ are disjoint, $\ell \cup \iota$ denotes their concatenation; for example, if $\ell=(8,2,5)$ and $\iota=(1,6)$, then $\ell \cup \iota=(8,2,5,1,6)$.
Finally, $\wto$ denotes weak convergence and $\pto$ convergence in probability.
We write $A \lesssim B$ to mean $A \leq C B$ for a universal constant $C$ that does not depend on $n$ or $s$, for all sufficiently large $n$.

\subsection*{Paper Organization}
Section~2 introduces the generalized U-statistic framework, states the H\'ajek-dominance, square-LLN, and sampling conditions, and proves ratio-consistency of ordinary and delete-$d$ jackknife variance estimators in complete and incomplete settings.
Section~3 applies the framework to TDNN estimation, showing that jackknife variance estimation remains valid when $s_2=o(n)$ and that studentized inference follows whenever TDNN asymptotic normality holds on the same variance scale.
The appendix, beginning with \cref{sec:appendix_delete_d}, contains the general delete-$d$ jackknife proofs, the single-scale DNN inputs, and the TDNN extension.

\section{Generalized U-Statistics}

We use the generalized U-statistic framework of~\cite{peng_rates_2022}.
This framework unifies incomplete, randomized, and infinite-order U-statistics, covering random forests and a broad class of ensemble estimators.
The appeal of the framework is that many modern subsampling estimators can be studied through a common projection structure rather than through estimator-specific algebra.
Developing jackknife theory at this level therefore supports the use of such estimators in fields such as computer science and economics.

\begin{definition}[Generalized U-Statistic]\label{def:gen_ustat}\mbox{}\\*
    Suppose $D_{[n]} = \left(Z_1, \ldots, Z_n\right)$ is a data set consisting of i.i.d.\ observations from $F_Z$.
    Let $h$ denote a (possibly randomized) real-valued function utilizing $s$ of these observations that is permutation-symmetric in those $s$ arguments.
    A generalized U-statistic with kernel $h$ of order (rank) $s$ is any estimator of the form
    \begin{equation}\label{eq:genUStat}
        U_{n, s, N, \omega}\left(D_{[n]}\right)
        = \frac{1}{\hat{N}} \sum_{\ell \in L_{n,s}} \rho_{\ell} h\left(D_{\ell} ; \omega\right)
    \end{equation}
    where $\omega$ denotes i.i.d.\ randomness, independent of the original data.
    $\left(\rho_{\ell}\right)_{\ell \in L_{n,s}}$ denotes a collection of i.i.d.\ Bernoulli random variables determining which subsamples are selected and is independent of all other inputs to the U-statistic.
    Set $p \defeq \Pb*{\rho_{\ell}=1} = N /\binom{n}{s}$, and let the actual number of selected subsamples be $\hat{N}=\sum_{\ell \in L_{n,s}} \rho_{\ell}$, so $\Ex*{\hat{N}}=N$.
    The Bernoulli design therefore has \(0<N\leq\binom{n}{s}\), so \(0<p\leq1\); \(N\) is the expected selected count and need not be an integer.
    The ratio in \cref{eq:genUStat} is understood on the event $\hat N>0$; on the event $\hat N=0$, where the numerator is also zero, we set $U_{n,s,N,\omega}\left(D_{[n]}\right)=0$.
    This zero-count event has probability $\Pb*{\hat N=0}=(1-p)^{\binom{n}{s}}\leq \exp(-N)$, so the convention is asymptotically immaterial whenever \(N \to \infty\), as in the incomplete-statistic regimes considered below.
    When $N=\binom{n}{s}$, the estimator in \cref{eq:genUStat} is a complete generalized U-statistic and is denoted as $U_{n, s, \omega}$.
    When $N<\binom{n}{s}$, these estimators are incomplete generalized U-statistics.
\end{definition}

Here $N$ is the target number of subsamples, $\hat{N}$ is the realized number selected by the Bernoulli design, and $\rho_{\ell}$ records whether the kernel is evaluated on subset $\ell$.
The complete case sets every $\rho_{\ell}=1$, while the incomplete case keeps the same kernel but randomizes which subsamples are used.
Separating the kernel from the sampling scheme lets one variance-estimation argument cover both settings.

We later apply the general results to the two-scale distributional nearest-neighbor estimator of~\cite{demirkaya_optimal_2024}.
For a kernel \(h_s\), write $\theta_{s} = \Ex*{h_{s}\left(D_{[s]}\right)}$ for the nominal kernel mean.
In the complete case this is exactly the expectation of the statistic.
In the incomplete Bernoulli-sampled case, the zero-count convention in \cref{def:gen_ustat} changes the literal expectation to \(\theta_s\Pb*{\hat N>0}\), because the statistic is set to zero on \(\set*{\hat N=0}\).
The discrepancy is confined to an exponentially unlikely event, so \(\theta_s\) remains the natural centering scale.
When convenient, the complete-case proof may set \(\theta_s=0\) by exact centering.
The incomplete-case proof keeps the same centering interpretation but uses an explicit zero-count centering-transfer argument to account for the convention.
For some results, boundedness of $\theta_s$ in $s$ is the more important requirement, and this is benign in most applications.
For TDNN estimation, we keep the nonparametric regression function explicit because it is central to the application.
The auxiliary randomness can be viewed as $\omega=(W_\ell)_{\ell \in L_{n,s}}$, where each $W_\ell$ contains only the extra randomness injected into the kernel on subsample $\ell$.

Much of classical U-statistic theory rests on the celebrated Hoeffding decomposition of~\cite{hoeffding_class_1948}.
This technique decomposes a classical U-statistic into uncorrelated components of orders one through $s$, with each term capturing progressively higher-order interactions.
Equivalently, the order-$i$ component is the projection of the kernel onto the space of functions that depend on exactly $i$ arguments and are orthogonal to all lower-order components.
It is a powerful tool for studying variance and limiting distributions because the components are orthogonal and higher-order terms are often negligible.
\cite{peng_rates_2022} extend this decomposition to generalized U-statistics, with the classical result as a special case.

\begin{definition}[Generalized Hoeffding Decomposition]\label{def:gen_HDecomp}\mbox{}\\*
    Suppose $D_{[n]} = \left(Z_1, \ldots, Z_n\right)$ is a data set consisting of i.i.d.\ samples from $F_Z$ and $d_{[n]}$ is a fixed realization of the data set.
    Let $h_{s}\left(D_{[s]} ; \omega\right)$ be a (possibly randomized) real valued function that is permutation-symmetric in $D_{[s]}$.
    Let
    \begin{equation}
        h_{s|i}\left(d_{[i]}\right)
        = \Ex*{h_{s}\left(d_{[i]}, D_{[s-i]} ; \omega\right)} - \Ex*{h_s(D_{[s]}; \omega)}
    \end{equation}
    for $i=1, \ldots, s$ and let
    \begin{align}
        h_{s}^{(i)}\left(D_{[i]}\right)
         & = h_{s|i}\left(D_{[i]}\right) - \sum_{j=1}^{i-1} \sum_{\ell \in L_{i,j}} h_{s}^{(j)}\left(D_{\ell}\right)
         && \text{for } i=1, \ldots, s-1, \\
        h_{s}^{(s)}\left(D_{[s]}; \omega\right)
         & = h_{s}\left(D_{[s]} ; \omega\right)-\sum_{j=1}^{s-1} \sum_{\ell \in L_{s,j}} h_{s}^{(j)}\left(D_{\ell}\right),                           \\
        H_{s}^{i}\left(D_{[n]}\right)
         & = \binom{n}{i}^{-1} \sum_{\ell \in L_{n, i}} h_{s}^{(i)}\left(D_{\ell}\right), \quad \text{for } i=1, \ldots, s-1 \quad \text{and}         \\
        H_{s}^{s}\left(D_{[n]}; \omega\right)
         & = \binom{n}{s}^{-1} \sum_{\ell \in L_{n, s}} h_{s}^{(s)}\left(D_{\ell}; \omega\right).
    \end{align}
    The $H$-decomposition of a generalized complete U-statistic is expressed as
    \begin{equation}
        \begin{aligned}
            U_{n, s, \omega}\left(D_{[n]}\right)
             & = \sum_{i=1}^{s-1}\left[\binom{s}{i}\binom{n}{i}^{-1} \sum_{\ell \in L_{n,i}} h_s^{(i)}\left(D_{\ell}\right)\right]
            + \binom{n}{s}^{-1} \sum_{\ell \in L_{n,s}} h_s^{(s)}\left(D_{\ell}; \omega\right)                                     \\
             & = \sum_{i=1}^{s-1}\binom{s}{i} H_{s}^{i}\left(D_{[n]}\right)
            + H_{s}^{s}\left(D_{[n]}; \omega\right).
        \end{aligned}
    \end{equation}
\end{definition}

The lower-order projection kernels $h_s^{(i)}$ for $i < s$ carry no $\omega$ argument because conditional expectations integrate out the external randomization; $\omega$ only remains in the residual term $h_s^{(s)}$.

The decomposition in \cref{def:gen_HDecomp} identifies the key variance-estimation objects: the first-projection kernel $h_s^{(1)}$, the H\'ajek term $(s/n)\sum_{i=1}^n h_s^{(1)}(Z_i)$, and the higher-order remainder $\sum_{j=2}^s \binom{s}{j} H_s^j\left(D_{[n]}\right)$.
Operationally, $h_s^{(1)}(Z_i) = \Ex*{h_s(D_{[s]};\omega) \given Z_i} - \theta_s$ measures how observation $i$ shifts the kernel's conditional mean, while the higher-order terms collect interaction effects that are harder for resampling methods to track.
The argument asks when these interactions are small enough that jackknife perturbations behave as if applied to the linear H\'ajek term.

Continuing with the notation of~\cite{peng_rates_2022}, define the following variance terms.
\begin{align}
    \zeta_{s, \omega}^{c}
     & = \Covp*{h\left(D_{[c]}, D_{[s-c]}; \omega\right), h\left(D_{[c]}, D_{[s-c]}^{\prime}; \omega^{\prime}\right)}
    \quad \text{for} \quad c = 1, \dotsc, s-1 \label{eq:var_terms1}                                                            \\
    \zeta_{s}^{s}
     & = \Covp*{h\left(D_{[s]}; \omega\right), h\left(D_{[s]}; \omega\right)}
    =  \Varb*{h_{s}\left(D_{[s]}; \omega\right)}\label{eq:var_terms2}                                                 \\
    V_{s, \omega}^{c}
     & = \Varb*{h_{s}^{(c)}\left(D_{[c]}\right)}
    \quad \text{for} \quad c = 1, \dotsc, s-1                                                                                  \\
    V_{s}^{s}
     & = \Varb*{h_{s}^{(s)}\left(D_{[s]}; \omega\right)}
\end{align}
Here, variables with a prime (such as $D_{[s-c]}^{\prime}$ or $\omega^{\prime}$) denote random variables that follow the same distribution as their non-prime counterparts.
They are also independent of all other input variables, including each other and their non-prime counterparts.
The subscript $\omega$ for terms of order $1 \leq c < s$ indicates that expectations include $\omega$, so additional randomization appears only in the final residual terms of the expansions.
For \(c=1\), the covariance definition coincides with the variance of the first Hoeffding projection.
Indeed, the two kernel evaluations in \(\zeta_{s,\omega}^{1}\) are conditionally independent given the shared observation \(Z_1\), so
\begin{equation}
    \zeta_{s,\omega}^{1}
    =
    \Varb*{\Ex*{h_s(D_{[s]};\omega)\given Z_1}}
    =
    \Varb*{h_s^{(1)}(Z_1)}.
\end{equation}
Thus, \(\zeta_{s,\omega}^1\) is the first-projection variance scale and will be the key quantity governing the leading term in both the variance decomposition and the jackknife analysis below.
By contrast, $\zeta_s^s = \Varb*{h_s\left(D_{[s]}; \omega\right)}$ is the full-kernel variance scale.
Standard U-statistic results, extended to the generalized setting, give the following variance decomposition in terms of Hoeffding projection variances.
\begin{equation}\label{eq:Var_decomp}
    \begin{aligned}
        \Varb*{U_{n, s, \omega}\left(D_{[n]}\right)}
         & = \sum_{j = 1}^{s-1} \binom{s}{j}^2 \Varb*{H_{s}^{j}\left(D_{[n]}\right)}
        + \Varb*{H_{s}^{s}\left(D_{[n]}; \omega\right)}                                                    \\
         & = \sum_{j = 1}^{s - 1} \binom{s}{j}^2 \binom{n}{j}^{-1} \Varb*{h_{s}^{(j)}\left(D_{[j]}\right)}
        + \binom{n}{s}^{-1} \Varb*{h_{s}^{(s)}\left(D_{[s]}; \omega\right)}                                \\
         & = \sum_{j = 1}^{s-1} \binom{s}{j}^2 \binom{n}{j}^{-1} V_{s, \omega}^{j} + \binom{n}{s}^{-1} V_{s}^{s}    \\
    \end{aligned}
\end{equation}
For fixed $s$ and no additional randomization, \cref{eq:Var_decomp} recovers the classical U-statistic variance decomposition.

\section{Consistent Variance Estimation}

Theorems 1 and 2 of~\cite{peng_rates_2022} establish normality of generalized U-statistics under suitable H\'ajek-projection variance dominance conditions.
The same condition drives jackknife consistency here, so we state it separately.
The delete-\(d\) jackknife results below are stated for triangular-array rows with integer kernel order \(s=s_n\geq2\) eventually; the rank-one case is not treated separately here because the higher-order Hoeffding remainder is absent.
\begin{asm}[Asymptotic H\'ajek Dominance Condition]\label{asm:hajek_dominance}\mbox{}\\*
    Consider a potentially incomplete generalized U-statistic $U_{n, s, N, \omega}$ along rows with \(\zeta_{s,\omega}^{1}>0\) eventually.
    If
    \begin{equation}
        \frac{s}{n} \left(\frac{\zeta_{s}^{s}}{s \zeta_{s, \omega}^{1}} - 1\right) \longrightarrow 0,
    \end{equation}
    we say $U_{n, s, N, \omega}$ satisfies the asymptotic H\'ajek dominance condition.
\end{asm}
This condition places the statistic in a first-order variance regime.
As $n \rightarrow \infty$ and $s = s(n) \rightarrow \infty$, both the full-kernel variance and the first-projection variance may vanish, but their relative scale must leave the first-order term asymptotically decisive.
In that regime, deleting observations perturbs the statistic primarily through its linear H\'ajek term rather than through higher-order interactions.
\begin{lemma}[Dominance of H\'ajek Projection Variance]\label{lem:Hajek_Dominance}\mbox{}\\*
    Let $\mathrm{U}_{n, s, \omega}\left(D_{[n]}\right)$ be a complete generalized U-statistic.
    Let the kernel variance terms $\zeta_{s}^{s}$ and $\zeta_{s, \omega}^{1}$ be defined as in \cref{eq:var_terms1,eq:var_terms2}.
    Then under \cref{asm:hajek_dominance}, the H\'ajek projection term asymptotically dominates the variance of the generalized U-statistic
    \begin{equation}
        \frac{n}{s^2}\frac{\Varb*{\mathrm{U}_{n, s, \omega}\left(D_{[n]}\right)}}{\zeta_{s, \omega}^{1}}
        \longrightarrow 1.
    \end{equation}
\end{lemma}

Thus \cref{asm:hajek_dominance} says that individual-observation contributions, not higher-order interactions, explain the leading variance, with scale $(s^2/n)\zeta_{s,\omega}^1$.
That variance comparison controls the higher-order Hoeffding remainder, but the delete-$d$ proof also contains a diagonal averaged-square term built from the normalized first projection.
To make that term track the same variance scale, we impose the following row-wise weak-law condition.
\begin{asm}[Row-wise $L^r$ Square-LLN Condition]\label{asm:row_Lr}\mbox{}\\*
    Consider a potentially incomplete generalized U-statistic $U_{n, s, N, \omega}$ along rows with \(\zeta_{s,\omega}^{1}>0\) eventually.
    Write
    \begin{equation}
        W_{n,1}
        \defeq
        \frac{h_s^{(1)}(Z_1)^2}{\zeta_{s,\omega}^1}.
    \end{equation}
    If there exists $r \in (1,2]$ such that
    \begin{equation}
        \frac{\Ex*{W_{n,1}^r}}{n^{r-1}} \longrightarrow 0,
    \end{equation}
    we say $U_{n,s,N,\omega}$ satisfies the row-wise $L^r$ square-LLN condition.
\end{asm}
This is the standard $L^r$ sufficient condition for a weak law of row-wise i.i.d.\ triangular arrays; see, for example,~\cite{petrov_limit_2023}.
In Appendix~A we use the von Bahr--Esseen inequality~\cite{von_bahr_inequalities_1965} to record the short proof specialized to the normalized first-projection squares.
For the ordinary delete-1 jackknife with fixed kernel order, this condition is essentially automatic: the first projection is an ordinary i.i.d.\ influence term, so the usual law of large numbers controls the diagonal average.
It becomes visible here because $h_s^{(1)}$ and $W_{n,1}$ form a triangular array as the kernel order grows with $n$.
\begin{remark}
    A stronger but sometimes convenient sufficient condition is the standardized first-projection moment bound
    \begin{equation}
        \sup_n
        \frac{\Ex*{\abs*{h_s^{(1)}(Z_1)}^{2+\eta}}}
        {\left(\zeta_{s,\omega}^{1}\right)^{1+\eta/2}}
        < \infty
        \qquad \text{for some } \eta \in (0,2].
    \end{equation}
    Then \cref{asm:row_Lr} follows with $r = 1 + \eta/2$.
    This is the same scale of projection moment condition used by~\cite{peng_rates_2022} for Berry--Esseen analysis, but it is stronger than what the jackknife consistency proof itself requires.
\end{remark}
The assumptions now separate the two sources of difficulty in the complete case.
\cref{asm:hajek_dominance} suppresses the higher-order Hoeffding remainder, while \cref{asm:row_Lr} controls the diagonal average of squared first-order contributions.
Incomplete generalized U-statistics introduce a third layer because the jackknife replicates reuse a random subset of possible kernel evaluations.
For that setting we impose the following sampling condition from~\cite{peng_bias_2026}, which ensures that each deleted-sample statistic is still based on asymptotically enough sampled subsamples.
\begin{asm}[Asymptotically-Sufficient Sampling Condition]\label{asm:AS_Sampling}\mbox{}\\*
    Consider a potentially incomplete generalized U-statistic $U_{n, s, N, \omega}$ along rows with \(\zeta_{s,\omega}^{1}>0\) eventually.
    If
    \begin{equation}
        \frac{n}{N s \zeta_{s, \omega}^{1}} \longrightarrow 0,
    \end{equation}
    we say $U_{n, s, N, \omega}$ satisfies the asymptotically-sufficient sampling condition.
\end{asm}
This condition is specific to the incomplete case.
Each jackknife replicate reuses only the subsamples that were actually drawn, so the Bernoulli sampling layer must be rich enough that each observation still appears many times on average.
Otherwise, sampling noise from the subsample-selection scheme can dominate the jackknife comparison before the H\'ajek projection has a chance to govern the variance.

When \(\zeta_{s,\omega}^{1}\asymp1\), this condition is equivalent up to constants to \(n/(Ns)\to0\).
In that case the total number of observation appearances across selected subsamples grows faster than the sample size, or equivalently each observation appears a diverging number of times in expectation.
For example, this occurs when \(N = n^{1+\alpha}\) and \(s \rightarrow \infty\) with \(s = o(n)\) and \(\alpha > 0\).
Taken together, the three assumptions describe three separate requirements.
The statistic is asymptotically governed by its linear H\'ajek projection, the diagonal average of squared linear contributions obeys a law of large numbers, and Bernoulli thinning is dense enough to remain a second-order perturbation.
This is precisely the regime in which delete-$d$ jackknife perturbations should behave as if they were applied to a linear statistic.
We now justify jackknife variance estimation under this projection-dominance and square-LLN structure, beginning with the variance target
\begin{equation}\label{eq:genUStat_Var}
    \sigma^{2}_{n}
    = \Varb*{U_{n, s, N, \omega}\left(D_{[n]}\right)}.
\end{equation}
The subscript in \cref{eq:genUStat_Var} emphasizes dependence on $n$, and hence on $s$ and $N$.
We consider the following variance estimators.
\begin{definition}[Jackknife Variance Estimators]\label{def:jk_var_estimators}\mbox{}\\*
    We consider the \textit{jackknife variance estimator}
    \begin{equation}\label{eq:JK_Var_Est}
        \hat{\sigma}_{JK}^2\left(D_{[n]}; \omega\right)
        \defeq \frac{n-1}{n} \sum_{i = 1}^{n}\left(U_{n, s, N, \omega}\left(D_{[n], -i}\right) - U_{n, s, N, \omega}\left(D_{[n]}\right)\right)^2
    \end{equation}
    and the \textit{delete-$d$ jackknife variance estimator}
    \begin{equation}\label{eq:JKD_Var_Est}
        \hat{\sigma}_{JKD}^2\left(D_{[n]}; d, \omega\right)
        \defeq \frac{n-d}{d}\binom{n}{d}^{-1} \sum_{\ell \in L_{n,d}}\left(U_{n, s, N, \omega}\left(D_{[n], -\ell}\right)
        - U_{n, s, N, \omega}\left(D_{[n]}\right)\right)^2.
    \end{equation}
    For incomplete generalized U-statistics, the deleted-sample evaluations use the same realizations of \(\rho\) and \(\omega\) as the full-sample statistic.
    This keeps the jackknife comparison focused on deleting observations rather than mixing that effect with fresh subsampling or auxiliary-randomization noise.
    Specifically, for a delete set \(\ell \in L_{n,d}\), define
    \begin{equation}
        \calO_{s,0}(\ell)
        \defeq
        \set*{\iota \in L_{n,s}: \iota \cap \ell = \emptyset},
        \qquad
        \hat N_{\ell}^{\circ}
        \defeq
        \sum_{\iota \in \calO_{s,0}(\ell)} \rho_{\iota}.
    \end{equation}
    Then \(U_{n,s,N,\omega}(D_{[n],-\ell})\) is evaluated by retaining only the originally sampled subsamples that avoid \(\ell\):
    \begin{equation}
        U_{n,s,N,\omega}\left(D_{[n],-\ell}\right)
        =
        \frac{1}{\hat N_{\ell}^{\circ}}
        \sum_{\iota \in \calO_{s,0}(\ell)}
        \rho_{\iota} h\left(D_{\iota}; \omega\right),
        \qquad
        \text{when } \hat N_{\ell}^{\circ}>0.
    \end{equation}
    If \(\hat N_{\ell}^{\circ}=0\), we again set this deleted-sample evaluation to zero.
    For each fixed \(\ell\), this zero-count event has probability at most \(\exp(-N_{d}^{\circ})\), where \(N_{d}^{\circ}=p\binom{n-d}{s}\); when \(N\to\infty\) and \(sd=o(n)\), \(N_{d}^{\circ}\to\infty\) because \(N_{d}^{\circ}/N\to 1\).
    No Bernoulli subsamples are redrawn after deletion.
\end{definition}
For \(d=1\), \cref{eq:JKD_Var_Est} reduces to the ordinary jackknife estimator in \cref{eq:JK_Var_Est}.
We first consider complete generalized U-statistics and point out that analogous results on the consistency of the pseudo-infinitesimal jackknife estimator are derived in~\cite{peng_bias_2026}.

\begin{theorem}[Variance Estimation for Complete Generalized U-Statistics]\label{thm:JK_Consistency_compl}\mbox{}\\*
    Let $U_{n, s, \omega}$ be a complete generalized U-statistic satisfying \cref{asm:hajek_dominance,asm:row_Lr}.
    Suppose \(2\leq s=s_n\leq n\), \(1\leq d=d_n\), \(s+d\leq n\) eventually, \(s=o(n)\), \(sd=o(n)\), \(\zeta_{s,\omega}^{1}>0\) eventually, and \(\sigma_n^2>0\) eventually.
    Let $\hat{\sigma}_{JKD}^2\left(D_{[n]}; d, \omega\right)$ be the associated delete-$d$ jackknife variance estimator as defined in \cref{eq:JKD_Var_Est}. Then
    \begin{equation}
        \frac{\hat{\sigma}_{JKD}^2\left(D_{[n]}; d, \omega\right)}{\sigma_{n}^{2}} \pto 1.
    \end{equation}
    In particular, for $d = 1$,
    \begin{equation}
        \frac{\hat{\sigma}_{JK}^2\left(D_{[n]}; \omega\right)}{\sigma_{n}^{2}} \pto 1.
    \end{equation}
\end{theorem}
\cref{sec:appendix_delete_d_complete} contains the full proof.
\begin{remark}
    Although $sd = o(n)$ allows $d$ to grow, the most relevant practical regimes have fixed $d$.
    The result then allows one to combine delete-$d$ estimates for different fixed values of \(d\) to eliminate lower-order bias terms.
    Although such combinations are beyond this paper, they could improve finite-sample performance.
    In contrast, for a classical U-statistic with fixed $s$ this result justifies the use of delete-$d$ jackknife variance estimators in the regime $d = o(n)$.
\end{remark}
The extension to incomplete generalized U-statistics is practically important.
When no closed-form evaluation is available, evaluating the kernel on all subsets of size $s$ is often prohibitive, so one uses a Bernoulli sampling scheme as in \cref{def:gen_ustat}.
The following result covers that setting.

\begin{theorem}[Variance Estimation for Incomplete Generalized U-Statistics]\label{thm:JK_Consistency_incompl}\mbox{}\\*
    Let $U_{n, s, N, \omega}$ be a potentially incomplete generalized U-statistic satisfying \cref{asm:hajek_dominance,asm:row_Lr,asm:AS_Sampling}.
    Suppose \(2\leq s=s_n\leq n\), \(1\leq d=d_n\), \(s+d\leq n\) eventually, \(s=o(n)\), \(sd=o(n)\), \(\zeta_{s,\omega}^{1}>0\) eventually, and \(\sigma_n^2>0\) eventually.
    Furthermore, let $\zeta_{s}^{s}$ and $\theta_s=\Ex*{h_s(D_{[s]};\omega)}$ be bounded in $s$.
    Let $\hat{\sigma}_{JKD}^2\left(D_{[n]}; d, \omega\right)$ be the associated delete-$d$ jackknife variance estimator as defined in \cref{eq:JKD_Var_Est}. Then
    \begin{equation}
        \frac{\hat{\sigma}_{JKD}^2\left(D_{[n]}; d, \omega\right)}{\sigma_{n}^{2}} \pto 1.
    \end{equation}
    In particular, for $d = 1$,
    \begin{equation}
        \frac{\hat{\sigma}_{JK}^2\left(D_{[n]}; \omega\right)}{\sigma_{n}^{2}} \pto 1.
    \end{equation}
\end{theorem}
\cref{sec:appendix_delete_d_incomplete} contains the full proof.

\section{Application to Two-Scale Distributional Nearest Neighbor Estimator}

We illustrate the inferential payoff of the general theory with the two-scale distributional nearest-neighbor regression estimator of~\cite{demirkaya_optimal_2024}.
This estimator targets pointwise nonparametric regression at a fixed covariate value $x$, so valid uncertainty quantification requires a consistent variance estimator for a localized subsampling-based object.
The TDNN estimator is therefore a natural application of the generalized U-statistic results above.
We work under the following setup.

\begin{asm}[Nonparametric Regression DGP]\label{asm:npr_dgp}\mbox{}\\*
    The observed data consist of an i.i.d.\ sample taking the following form.
    \begin{equation}
        D_{[n]} = \{Z_{i} = (X_{i}, Y_{i})\}_{i = 1}^{n}
        \quad \text{from the model} \quad
        Y = \mu(X) + \varepsilon,
    \end{equation}
    where $Y \in \calY \subset \R$ is the response, $X \in \calX \subset \R^k$ is a feature vector of fixed dimension $k$ distributed according to a density function $f$ with associated probability measure $\varphi$ on $\calX$, and $\mu(x)$ is the unknown mean regression function.
    $\varepsilon$ is the unobservable model error on which we impose the following conditions.
    \begin{equation}
        \Ex*{\varepsilon \given X} = 0, \quad
        \Varb*{\varepsilon \given X = x} = \sigma_{\varepsilon}^2\left(x\right)
    \end{equation}
    Let the distribution induced by this model be denoted by $P$ and thus $Z_{i} = \left(X_{i}, Y_{i}\right) \overset{\text{iid}}{\sim} P$.
\end{asm}

The pointwise asymptotic-normality result of~\cite{demirkaya_optimal_2024} assumes homoskedastic errors independent of $X$ and centers the estimator by an explicit residual bias term.
For studentization, we combine that result with the heteroskedastic localization adjustment below and require the residual bias to be negligible on the TDNN standard-error scale.
\begin{remark}
    The heteroskedastic version replaces the uses of $\varepsilon \perp X$ in~\cite{demirkaya_optimal_2024} with localized conditional moments.
    The required changes occur only where independence is used to factor an error moment away from a DNN selector weight.
    For the cubic term below, we also use continuity at $x$ of $m_3(u) \coloneq \Ex*{\varepsilon^3 \given X=u}$.
    Let $q_s(X_1) = \Ex*{\kappa(x;Z_1,D_{[s]})\given X_1}$ be the DNN selector weight from \cref{lem:dnn_selector_profile}.
    \emph{(i)~Second-moment terms.}
    The upper-bound argument uses the second moment of the error weighted by the DNN selector.
    Under independence this appears as the factorization $\Ex*{\varepsilon_1^2\,q_s(X_1)} = \sigma_\varepsilon^2\,\Ex*{q_s(X_1)}$.
    Under \cref{asm:npr_dgp}, the corresponding expression is $\Ex*{\varepsilon_1^2\,q_s(X_1)} = \Ex*{\sigma_\varepsilon^2(X_1)\,q_s(X_1)}$.
    Since $s\,\Ex*{\sigma_\varepsilon^2(X_1)\,q_s(X_1)} \to \sigma_\varepsilon^2(x)$ by DNN localization and the variance-function continuity in \cref{asm:tdnn_variance_design}~\cref{asm:tdnn_design_variance}, the upper bound holds with the local value $\sigma_\varepsilon^2(x)$.
    \emph{(ii)~Variance lower bound.}
    The first-projection variance lower bound is also localized.
    The homoskedastic bound $\eta_1 \geq \sigma_\varepsilon^2/(2s-1)$ uses $\varepsilon \perp X$.
    Under \cref{asm:npr_dgp}, the same role is played by $\eta_1 \geq \Ex*{\sigma_\varepsilon^2(X_1)\,q_s^2(X_1)} \geq \underline{\sigma}_\varepsilon^2/(2s-1)$, exactly as established in \cref{lem:dnn_hajek_input}.
    \emph{(iii)~Fourth-moment expansion.}
    The only fourth-moment term that changes qualitatively is the cubic cross-term $\Ex*{\mu(X_1)\,m_3(X_1)\,q_s(X_1)}$ in the expansion of $\Ex*{Y_1^4\,q_s(X_1)}$ (Demirkaya Lemma~3), which no longer factors.
    Since $\mu$ is bounded on $\calX$, $\Ex*{Y^4}<\infty$ implies $\mu(\cdot)m_3(\cdot)\in L^1$, and continuity of $m_3$ at $x$ gives continuity of $\mu(\cdot)m_3(\cdot)$ at $x$; therefore \cref{lem:dem13} yields convergence to $\mu(x)\,\Ex*{\varepsilon^3\given X=x}$.
\end{remark}
With this localization convention in place, we separate the inputs used for jackknife variance consistency from the smoother inputs used only when invoking external TDNN asymptotic-normality results.

\begin{asm}[TDNN Variance-Design Conditions]\label{asm:tdnn_variance_design}\mbox{}\\*
    For the jackknife variance results, assume the following design and second-moment regularity conditions.
    \begin{enumerate}[(i)]
        \item \textbf{Compact Support:}
              The feature space $\calX = \operatorname{supp}(X)$ is a bounded, compact subset of $\R^k$.
              \label{asm:tdnn_design_compact}
        \item \textbf{Density Bounds:}
              The density $f(\cdot)$ is bounded away from 0 and $\infty$, i.e.,
              \begin{equation}
                  \forall u \in \calX:
                  \quad
                  0 < \underline{\mathfrak{f}} \leq f(u) \leq \overline{\mathfrak{f}} < \infty.
              \end{equation}
              \label{asm:tdnn_design_density}
        \item \textbf{Regression Regularity:}
              $\mu(\cdot)$ is continuous on $\calX$ and $\mu(\cdot) \in L^{2}\left(\calX\right)$ with respect to $\varphi$.
              \label{asm:tdnn_design_regression}
        \item \textbf{Variance Regularity:}
              $\sigma^2_{\varepsilon}: \calX \longrightarrow \R_{>0}$ is continuous on $\calX$ and square-integrable with respect to $\varphi$.
              Since $\calX$ is compact, we write
              \begin{equation}
                  0 < \underline{\sigma}_{\varepsilon}^{2}
                  \leq
                  \sigma^2_\varepsilon(u)
                  \leq
                  \overline{\sigma}_{\varepsilon}^{2}
                  < \infty,
                  \qquad u \in \calX.
              \end{equation}
              \label{asm:tdnn_design_variance}
    \end{enumerate}
\end{asm}

\begin{asm}[TDNN Smoothness Inputs for External CLT Results]\label{asm:tdnn_clt_smoothness}\mbox{}\\*
    When applying the external TDNN asymptotic-normality result of~\cite{demirkaya_optimal_2024}, assume in addition that $f(\cdot)$ and $\mu(\cdot)$ are four times continuously differentiable with bounded second, third, and fourth-order partial derivatives.
    Specifically, for all $u \in \calX$ and all $(i,j,l,m) \in [k]^4$,
    \begin{alignat*}{14}
        -\infty & \; < \;    & \underline{\mathfrak{f}}^{\prime}
                & \; \leq \; &                                   & \partial_{i,j} f(u), \;          &         & \partial_{i,j,m} f(u), \;   &  & \partial_{i,j,l,m} f(u)   &  & \; \leq \; & \overline{\mathfrak{f}}^{\prime} & \; < \; & \infty \\
        -\infty & \; < \;    & \underline{\mathfrak{m}}^{\prime}
                & \; \leq \; &                                   & \partial_{i,j} \mu(u), \;        &         & \partial_{i,j,m} \mu(u), \; &  & \partial_{i,j,l,m} \mu(u)
                &            & \; \leq \;                        & \overline{\mathfrak{m}}^{\prime} & \; < \; & \infty.
    \end{alignat*}
    For the heteroskedastic localization adjustment in the preceding remark, also assume the conditional third-moment function $m_3(u)\coloneq\Ex*{\varepsilon^3 \given X=u}$ exists in a neighborhood of $x$ and is continuous at $x$ whenever the corresponding fourth-moment expansion is used.
\end{asm}

\begin{remark}
    The bounded-derivatives condition in \cref{asm:tdnn_clt_smoothness} is used only by external TDNN asymptotic-normality inputs invoked through \cref{thm:Inference_TDNN}.
    \cref{thm:TDNN_Hajek_Dominance,thm:JK_Consistency_TDNN} require only \cref{asm:tdnn_variance_design}, together with the response moment condition below for the row-wise square-LLN.
    In particular, the global fourth-order differentiability could be replaced by a local condition near $x$ at the cost of a more involved bias argument, without affecting the jackknife consistency result.
\end{remark}
We now define the DNN and TDNN kernels under this setup.
The single-scale distributional nearest-neighbor (DNN) estimator, developed in the U-statistic-type setting of~\cite{steele_exact_2009} and~\cite{biau_layered_2010}, averages over subsamples after keeping the observation closest to the target point \(x\) within each subsample.
To describe that selection rule, order the sample by distance to the fixed feature vector of interest \(x\):
\begin{equation}\label{eq:ordering}
    \norm*{X_{(1)} - x}_2
    < \norm*{X_{(2)} - x}_2
    < \dotsc
    < \norm*{X_{(n)} - x}_2
\end{equation}
The ordering in \cref{eq:ordering} is well-defined because $X$ is continuously distributed; for notational simplicity, tied observations receive the same rank.
Let $\rk(x; X, D)$ denote the rank relative to a point of interest $x$ that would be assigned to an observation with covariate vector $X$ if it was added to a sample $D$.
The data-driven selector
\begin{equation}
    \kappa(x; Z_{i}, D_{\ell})
    = \1*{\rk(x; X_{i}, D_{\ell}) = 1}
\end{equation}
therefore records whether observation \(i\) is the nearest neighbor to \(x\) within the subsample \(D_\ell\).
With kernel $h_{s}(x; D_{\ell}) \defeq \sum_{i = 1}^{n} \1*{i \in \ell} \kappa(x; Z_{i}, D_{\ell}) Y_{i}$, averaging over all subsamples gives the DNN estimator its U-statistic representation:
\begin{equation}\label{eq:U_stat}
    \tilde{\mu}_{s}(x; D_{[n]})
    = \binom{n}{s}^{-1} \sum_{\ell \in L_{n,s}} h_{s}(x; D_{\ell})
\end{equation}
The TDNN estimator combines two such DNN averages at subsampling scales $1 \leq s_1 < s_2 \leq n$.
The two-scale combination eliminates the leading bias term, analogously to higher-order kernels in nonparametric kernel regression.
Write $\mathfrak{S}=(s_1,s_2)$ and use this subscript for quantities associated with the TDNN sequence.
The corresponding bias-correction weights are
\begin{equation}
    w_{1}^{*} = \frac{1}{1-(s_1/s_2)^{-2/k}}
    \quad\text{and}\quad
    w_2^{*} = 1 - w_{1}^{*}(s_1, s_2)
\end{equation}
With these weights,~\cite{demirkaya_optimal_2024} define the TDNN estimator as a weighted combination of the two DNN scales.
Because the smaller-scale DNN average can be embedded inside each \(s_2\)-subsample, the same estimator can be written as an order-\(s_2\) generalized U-statistic:
\begin{equation}
    \begin{aligned}
        \hat{\mu}_{\mathfrak{S}}\left(x; D_{[n]}\right)
         & = w_{1}^{*}\tilde{\mu}_{s_1}\left(x; D_{[n]}\right) + w_2^{*}\tilde{\mu}_{s_2}\left(x; D_{[n]}\right)
        = \binom{n}{s_2}^{-1} \sum_{\ell \in L_{n,s_2}} h_{\mathfrak{S}}(x; D_{\ell})
    \end{aligned}
\end{equation}
The corresponding order-$s_2$ TDNN kernel is
\begin{equation}
    \begin{aligned}
        h_{\mathfrak{S}}\left(x; D_{[s_2]}\right)
         & = w_{1}^{*}\left[\binom{s_2}{s_1}^{-1}\sum_{\ell \in L_{s_2, s_1}} h_{s_1}\left(x; D_{\ell}\right)\right] + w_{2}^{*} h_{s_2}\left(x; D_{[s_2]}\right) \\
         & = w_{1}^{*} \tilde{\mu}_{s_1}\left(x; D_{[s_2]}\right) + w_{2}^{*} h_{s_2}\left(x; D_{[s_2]}\right)
    \end{aligned}
\end{equation}
A bounded-ratio condition prevents the TDNN estimator from degenerating into an effectively single-scale DNN estimator.
\begin{asm}[Bounded Ratio of Kernel-Orders]\label{asm:kernel_order_ratio}\mbox{}\\*
    There is a constant $\mathfrak{c} \in (0,1/2)$ such that the ratio of kernel orders is bounded in the following way.
    \begin{equation}
        \forall n: \quad 0 < \mathfrak{c} \leq s_1 / s_2 \leq 1 - \mathfrak{c} < 1.
    \end{equation}
\end{asm}
If $s_1/s_2$ is too close to $1$, the two scales are nearly redundant; if it is too small, the bias-correction weights become overly uneven.

The inferential gain comes from applying the general jackknife theorem to this order-$s_2$ kernel.
In the original paper, consistency of the jackknife variance estimator is established under the strong condition $s_{2} = o(n^{1/3})$, which narrows the range of kernel orders for which jackknife-based standard errors are theoretically justified.
The projection argument developed here relaxes the variance-estimation requirement to $s_{2} = o(n)$, thereby enlarging the regime in which the TDNN estimator admits consistent jackknife variance estimation.
Studentized pointwise inference then follows in any subregime where a TDNN asymptotic-normality result is available on the same variance scale.
Here the variance target is localized at $x$, so we write
\begin{equation}\label{eq:TDNN_Var}
    \sigma^{2}_{n}(x)
    = \Varb*{\hat{\mu}_{\mathfrak{S}}\left(x; D_{[n]}\right)}
\end{equation}
We also include $x$ in the jackknife variance estimators to make the variance target in \cref{eq:TDNN_Var} explicit.

\begin{theorem}[The TDNN Estimator satisfies the Asymptotic H\'ajek Dominance Condition]\label{thm:TDNN_Hajek_Dominance}\mbox{}\\*
    Consider a data-generating process as outlined in \cref{asm:npr_dgp} and \cref{asm:tdnn_variance_design}.
    Let $0 < s_1 < s_2 = o(n)$ be such that \cref{asm:kernel_order_ratio} holds.
    Then the single-scale DNN and two-scale TDNN regression estimators satisfy the Asymptotic H\'ajek Dominance Condition.
\end{theorem}

\cref{sec:dnn_hajek_dominance} first establishes the required single-scale DNN bounds.
\cref{sec:tdnn_hajek_dominance_rewrite} then lifts those bounds to TDNN by decomposing the two-scale kernel into an embedded $s_1$-scale DNN average inside an $s_2$-sample plus the ordinary $s_2$-scale DNN kernel.
The key point is that the TDNN first projection is a controlled linear combination of two single-scale DNN first projections, and the selector coefficients do not cancel its first-order variance.

Thus \cref{thm:TDNN_Hajek_Dominance} supplies the variance-dominance half of the general jackknife theory for TDNN estimation.
To turn that dominance statement into jackknife ratio consistency, it remains to verify the row-wise square-LLN input.
The moment condition needed for that step is weaker and more local than the smoothness assumptions used for asymptotic normality and first-order bias removal in the TDNN literature.
Because the nearest-neighbor selector is a function of the covariates, the response envelope is naturally stated conditionally on the covariate value.
\begin{asm}[Response Moment Condition]\label{asm:response_moment}\mbox{}\\*
    There exist constants $\eta > 0$ and $C_Y < \infty$ such that
    \begin{equation}
        \sup_{u \in \calX}
        \Ex*{\abs*{Y}^{2+\eta} \given X = u}
        \leq C_Y.
    \end{equation}
\end{asm}
A global unconditional $(2+\eta)$-moment bound would suffice, but it is stronger than needed because the selector probes the response distribution only through neighborhoods of $x$ with nonnegligible weight.
Since $\mu$ is bounded on the compact support under \cref{asm:tdnn_variance_design}~\cref{asm:tdnn_design_compact,asm:tdnn_design_regression}, the same assumption also gives a uniform conditional $(2+\eta)$-moment bound for $\varepsilon=Y-\mu(X)$.
In the row-wise verifications below we use the fixed exponent $r_\eta \defeq 1 + \frac{1}{2}\min\{\eta,2\}$, which lies in $(1,2]$ and satisfies $2r_\eta \leq 2+\eta$.
At this exponent, \cref{lem:dnn_row_Lr} verifies the row-wise $L^r$ square-LLN condition from \cref{asm:row_Lr} for the single-scale DNN first projection, and \cref{lem:tdnn_row_Lr} proves the TDNN analogue through \cref{eq:tdnn_first_projection_target_rate} and the resulting moment bound.
Together with H\'ajek dominance, this gives the inputs for the following variance-consistency theorem.

\begin{theorem}[Ratio-Consistent Variance Estimation for the TDNN Estimator]\label{thm:JK_Consistency_TDNN}\mbox{}\\*
    Let $0 < s_1 < s_2 = o(n)$, \(1 \leq d=d_n\), \(s_2+d\leq n\) eventually, and \(s_2 d = o(n)\).
    Under the data-generating process outlined in \cref{asm:npr_dgp,asm:tdnn_variance_design,asm:response_moment}, and for kernel orders satisfying \cref{asm:kernel_order_ratio}, the associated jackknife variance estimators for the TDNN regression estimator satisfy the following ratio consistency result:
    \begin{equation}
        \frac{\hat{\sigma}_{JKD}^2\left(x, D_{[n]}; d\right)}{\sigma_{n}^{2}\left(x\right)} \pto 1.
    \end{equation}
    In particular, for $d = 1$,
    \begin{equation}
        \frac{\hat{\sigma}_{JK}^2\left(x, D_{[n]}\right)}{\sigma_{n}^{2}\left(x\right)} \pto 1.
    \end{equation}
\end{theorem}

\cref{thm:JK_Consistency_TDNN} is the general complete-case jackknife theorem specialized to the localized TDNN kernel.
Once \cref{thm:TDNN_Hajek_Dominance} supplies H\'ajek dominance at order $s_2$ and \cref{lem:dnn_row_Lr,lem:tdnn_row_Lr} verify the row-wise square-LLN input under \cref{asm:response_moment}, \cref{thm:JK_Consistency_compl} applies without further structural changes.
The same verification also pins down the variance scale: $\sigma_n^2(x) = O(s_2/n)$, so the local standard error contracts at rate $\sqrt{s_2/n}$.

We do not treat incomplete variants here because the DNN and TDNN regression estimators admit convenient closed-form representations.
The preceding variance result has an immediate inference implication: any TDNN asymptotic-normality statement on the $\sigma_n(x)$ scale can be combined with the jackknife variance estimator to obtain feasible studentization.
We state that implication explicitly.

\begin{theorem}[Studentized Inference from TDNN Asymptotic Normality]\label{thm:Inference_TDNN}\mbox{}\\*
    Let $0 < s_1 < s_2 = o(n)$, \(1 \leq d=d_n\), \(s_2+d\leq n\) eventually, and \(s_2 d = o(n)\).
    Under the data-generating process outlined in \cref{asm:npr_dgp,asm:tdnn_variance_design,asm:response_moment}, and for kernel orders satisfying \cref{asm:kernel_order_ratio}, suppose that
    \begin{equation}
        \frac{\hat{\mu}_{\mathfrak{S}}\left(x; D_{[n]}\right) - \mu\left(x\right)}{\sigma_{n}\left(x\right)} \wto \mathcal{N}\left(0,1\right).
    \end{equation}
    Then the associated jackknife variance estimators yield the following studentized limit:
    \begin{equation}
        \frac{\hat{\mu}_{\mathfrak{S}}\left(x; D_{[n]}\right) - \mu\left(x\right)}{\sqrt{\hat{\sigma}_{JKD}^2\left(x, D_{[n]}; d\right)}} \wto \mathcal{N}\left(0,1\right).
    \end{equation}
    In particular, for $d = 1$,
    \begin{equation}
        \frac{\hat{\mu}_{\mathfrak{S}}\left(x; D_{[n]}\right) - \mu\left(x\right)}{\sqrt{\hat{\sigma}_{JK}^2\left(x, D_{[n]}\right)}} \wto \mathcal{N}\left(0,1\right).
    \end{equation}
\end{theorem}
\cref{thm:Inference_TDNN} is a direct Slutsky corollary of \cref{thm:JK_Consistency_TDNN}.
When a TDNN central limit theorem is available on the variance scale $\sigma_n(x)$, jackknife ratio consistency converts it into feasible pointwise inference.
For Theorem~3 of~\cite{demirkaya_optimal_2024}, one must also verify the external smoothness inputs in \cref{asm:tdnn_clt_smoothness} and check that its residual bias term is negligible relative to $\sigma_n(x)$ after the heteroskedastic localization adjustment above.

Thus, whenever TDNN is asymptotically normal on its own variance scale, jackknife studentization yields valid inference without further modifying the variance estimator.

\begin{remark}
    Several other base learners have already been investigated by~\cite{peng_rates_2022} for whether they satisfy the asymptotic H\'ajek dominance condition.
    These include classical U-statistics such as the mean and sample variance, giving another view of jackknife consistency in well-understood cases.
    They also include classical k-nearest neighbors (kNN), k-potential nearest neighbors (kPNN), and tree-based learners such as honest and double-sampling trees.
    Thus the results here apply immediately to other generalized U-statistics once the remaining square-LLN input is verified.
    More broadly, DNN localization applied to targets other than the conditional mean --- such as conditional quantile regression, parameters defined by localized conditional estimating equations, or other localized distributional functionals at a point --- falls within the same framework, since \cref{asm:hajek_dominance} is governed by the localizer geometry rather than the specific functional being estimated.
\end{remark}

\section{Conclusion and Outlook}

This paper establishes a simple criterion for consistent nonparametric jackknife variance estimation for generalized U-statistics: asymptotic dominance of the H\'ajek projection variance.
The same first-order dominance condition already plays a central role in generalized U-statistic asymptotics, and our results show that it also governs jackknife validity.
For complete generalized U-statistics, the jackknife and its delete-$d$ variants are ratio-consistent without modifying the base estimator.
For incomplete generalized U-statistics, the same conclusion holds under an additional asymptotically sufficient sampling condition.
The TDNN application shows that the results are not merely abstract: they deliver consistent jackknife standard errors throughout the broader regime $s_{2} = o(n)$, substantially extending the range previously justified in~\cite{demirkaya_optimal_2024}, and turn any TDNN asymptotic-normality result on that variance scale into feasible pointwise inference.

The same projection-based viewpoint suggests several extensions.
In degenerate problems, the relevant dominance condition should shift from the first projection to the first nonvanishing Hoeffding component.
For plug-in and forest-type procedures, the main task is to verify that the localized or subsampling-based first-stage object still induces a controlled projection variance of the kind studied here.
A related direction is conditional Z-estimation, where DNN localizer weights enter a score-based rather than regression-based kernel; the H\'ajek-dominance machinery applies once the score-kernel first-projection variance is controlled, with leave-one-out stability playing a complementary role on the bias side.
The present paper does not resolve those broader questions.
It instead isolates projection dominance as a tractable criterion for when simple jackknife-based uncertainty quantification remains valid for modern nonparametric and machine-learning estimators, including settings of econometric interest.

\begin{appendix}
    \section{\texorpdfstring{General Results for Delete-$d$ Jackknife Consistency}{General Results for Delete-d Jackknife Consistency}}\label{sec:appendix_delete_d}
    \subsection{Notation}

We collect the overlap bookkeeping used throughout the delete-$d$ jackknife arguments.

\begin{remark}[Delete-difference notation]\mbox{}\\*
	Let \(V\) be any statistic with a corresponding deleted-sample evaluation.
	Write
	\begin{equation}
		\Delta_{\ell}[V]
		\defeq
		V(D_{[n]}) - V(D_{[n],-\ell}),
		\qquad \ell \in L_{n,d}.
	\end{equation}
	The sign convention is immaterial inside the jackknife square, but this orientation aligns with the Hoeffding-order expansions below.
	When the statistic is clear from context, we write simply \(\Delta_{\ell}\); for auxiliary statistics, we add the corresponding adornment, such as \(\overline{\Delta}_{\ell}\).
\end{remark}

\begin{remark}[Delete-$d$ overlap bookkeeping]\label{rem:delete_d_overlap}\mbox{}\\*
	Fix a delete set $\ell \in L_{n,d}$, an order $1 \leq j \leq s$, and an overlap level $0 \leq a \leq \min\{j,d\}$.
	In the delete-\(d\) proofs this bookkeeping is used along rows with \(s+d\leq n\), so the post-deletion coefficients \(\binom{n-d}{j}^{-1}\) are well-defined for \(1\leq j\leq s\).
	We write
	\begin{equation}
		\calO_{j,a}(\ell)
		\defeq
		\left\{\iota \in L_{n,j} : \abs*{\iota \cap \ell} = a\right\}.
	\end{equation}
	This partitions $L_{n,j}$ by the number of deleted indices captured by the tuple $\iota$.
	By direct counting,
	\begin{equation}
		\abs*{\calO_{j,a}(\ell)}
		=
		\binom{d}{a}\binom{n-d}{j-a}.
	\end{equation}
	For $1 \leq j \leq s-1$, define the grouped projection sums
	\begin{equation}
		G_{j,a}^{(\ell)}\left(h_{s,\omega}^{(j)}; D_{[n]}\right)
		\defeq
		\sum_{\iota \in \calO_{j,a}(\ell)} h_{s,\omega}^{(j)}(D_{\iota}),
	\end{equation}
	and for the final kernel level set
	\begin{equation}
		G_{s,a}^{(\ell)}\left(h_{s}^{(s)}; D_{[n]}\right)
		\defeq
		\sum_{\iota \in \calO_{s,a}(\ell)} h_{s}^{(s)}(D_{\iota}).
	\end{equation}
	Since the ambient kernel and sample will always be clear in the delete-$d$ variance proofs, we suppress these arguments below and write simply $G_{j,a}^{(\ell)}$ and $G_{s,a}^{(\ell)}$.
	To keep track of the overlap combinatorics and the two coefficient types that appear in the delete-$d$ difference, define
	\begin{equation}
		p_{j,a}^{(n,d)}
		\defeq
		\frac{\binom{d}{a}\binom{n-d}{j-a}}{\binom{n}{j}},
		\qquad
		q_{j}^{(n,d)}
		\defeq
		\binom{n}{j}^{-1},
		\qquad
		r_{j}^{(n,d)}
		\defeq
		\binom{n}{j}^{-1} - \binom{n-d}{j}^{-1}.
	\end{equation}
	In particular, $\sum_{a=0}^{\min\{j,d\}} p_{j,a}^{(n,d)} = 1$.
	With this notation, the delete-$d$ difference of the $j$th Hoeffding projection can be written as
	\begin{equation}
		\begin{aligned}
			\delta_{j,\ell}
			 & \defeq
			\Delta_{\ell}[H_s^j] \\
			 & =
			\binom{n}{j}^{-1}\sum_{\iota \in L_{n,j}} h_{s,\omega}^{(j)}(D_{\iota})
			-
			\binom{n-d}{j}^{-1}
			\sum_{\iota \in L_{j}\left([n]\backslash \ell\right)}
			h_{s,\omega}^{(j)}(D_{\iota}) \\
			 & =
			q_{j}^{(n,d)} \sum_{a=1}^{\min\{j,d\}} G_{j,a}^{(\ell)}
			+
			r_{j}^{(n,d)} G_{j,0}^{(\ell)}
		\end{aligned}
	\end{equation}
	for $1 \leq j \leq s-1$, and the final-order difference is
	\begin{equation}
		\delta_{s,\ell}
		\defeq
		\Delta_{\ell}[H_s^s]
		=
		q_{s}^{(n,d)} \sum_{a=1}^{\min\{s,d\}} G_{s,a}^{(\ell)}
		+
		r_{s}^{(n,d)} G_{s,0}^{(\ell)}.
	\end{equation}
\end{remark}

    \subsection{Preliminary Lemmas}
    To show the consistency of the jackknife variance estimators under consideration, we will in part rely on the following basic result from~\cite{peng_bias_2026}.

\begin{lemma}[\cite{peng_bias_2026} - Lemma C.1.]\label{lem:peng1}\mbox{}\\*
	Let \(I_n\) be a finite index set.
	Suppose that
	\[
	    \sum_{\alpha\in I_n} X_{\alpha,n}^2 \pto 1,
	    \qquad
	    \sum_{\alpha\in I_n} \Ex*{X_{\alpha,n}^2} \longrightarrow 1,
	    \qquad
	    \sum_{\alpha\in I_n} \Ex*{Y_{\alpha,n}^2} \longrightarrow 0.
	\]
	Then
	\begin{equation}
		\sum_{\alpha\in I_n}\left[X_{\alpha,n}+Y_{\alpha,n}\right]^2 \pto 1
		\quad \text{and} \quad
		\Ex*{\sum_{\alpha\in I_n}\left(X_{\alpha,n}+Y_{\alpha,n}\right)^2} \longrightarrow 1.
	\end{equation}
\end{lemma}

\begin{proof}
    Expand \(\sum_{\alpha\in I_n}(X_{\alpha,n}+Y_{\alpha,n})^2\) into the three terms
    \[
        \sum_{\alpha\in I_n} X_{\alpha,n}^2
        +
        2\sum_{\alpha\in I_n} X_{\alpha,n}Y_{\alpha,n}
        +
        \sum_{\alpha\in I_n} Y_{\alpha,n}^2.
    \]
    The first term converges in probability to $1$ by assumption.
    For the last term, \(\Ex*{\sum_{\alpha\in I_n} Y_{\alpha,n}^2} \to 0\) by assumption, so \(\sum_{\alpha\in I_n} Y_{\alpha,n}^2 \pto 0\) by Markov's inequality.
    For the cross term, Cauchy--Schwarz gives
    \[
        \abs*{2\sum_{\alpha\in I_n} X_{\alpha,n}Y_{\alpha,n}}
        \leq
        2
        \left(\sum_{\alpha\in I_n} X_{\alpha,n}^2\right)^{1/2}
        \left(\sum_{\alpha\in I_n} Y_{\alpha,n}^2\right)^{1/2}
        \pto 0.
    \]
    For expectations, the same inequality and Cauchy--Schwarz give
    \begin{equation}
        \Ex*{\abs*{2\sum_{\alpha\in I_n} X_{\alpha,n}Y_{\alpha,n}}}
        \leq
        2\left(\sum_{\alpha\in I_n}\Ex*{X_{\alpha,n}^2}\right)^{1/2}
        \left(\sum_{\alpha\in I_n}\Ex*{Y_{\alpha,n}^2}\right)^{1/2}
        \longrightarrow 0.
    \end{equation}
    Together with \(\sum_{\alpha\in I_n}\Ex*{X_{\alpha,n}^2}\to1\) and \(\sum_{\alpha\in I_n}\Ex*{Y_{\alpha,n}^2}\to0\), this proves the expectation claim.
\end{proof}

\begin{lemma}[Square-LLN for the first projection]\label{lem:first_proj_square_lln}\mbox{}\\*
    Under \cref{asm:row_Lr},
    \begin{equation}
        \frac{1}{n\zeta_{s,\omega}^{1}}
        \sum_{i=1}^{n} h_s^{(1)}(Z_i)^2
        \pto 1.
    \end{equation}
\end{lemma}
This lemma controls the diagonal term that appears when the jackknife differences are squared.
In both the complete and incomplete proofs, the leading H\'ajek contribution produces an averaged square of first-projection values, and this normalized diagonal average converges to its population target on the \(\zeta_{s,\omega}^{1}\) scale.
Once the higher-order remainder has been shown negligible, this lemma transfers the linear variance scale to the jackknife.

\begin{proof}
    This is exactly the row-wise i.i.d.\ triangular-array weak law highlighted in Section~2; see~\cite{petrov_limit_2023}.
    In the present setting, the proof is a direct application of the von Bahr--Esseen inequality~\cite{von_bahr_inequalities_1965}.
    Write
    \begin{equation}
        W_{n,i}
        \defeq
        \frac{h_s^{(1)}(Z_i)^2}{\zeta_{s,\omega}^{1}},
        \qquad
        X_{n,i}
        \defeq
        W_{n,i} - 1,
        \qquad i = 1,\dotsc,n.
    \end{equation}
    Because the first projection is centered, $\Ex{W_{n,1}} = 1$, so the variables $X_{n,1},\dotsc,X_{n,n}$ are row-wise i.i.d.\ and mean zero.
    Let $r \in (1,2]$ be as in \cref{asm:row_Lr}.
    Then the von Bahr--Esseen inequality gives
    \begin{equation}
        \E\left[
            \abs*{
                \frac{1}{n}\sum_{i=1}^{n} X_{n,i}
            }^{r}
        \right]
        \leq
        \frac{2}{n^{r}}
        \sum_{i=1}^{n} \Ex*{\abs*{X_{n,i}}^{r}}
        =
        \frac{2\Ex*{\abs*{X_{n,1}}^{r}}}{n^{r-1}}.
    \end{equation}
    Since $r \geq 1$,
    \begin{equation}
        \abs*{x-1}^{r}
        \leq
        2^{r-1}\left(\abs*{x}^{r} + 1\right),
    \end{equation}
    so
    \begin{equation}
        \frac{\Ex*{\abs*{X_{n,1}}^{r}}}{n^{r-1}}
        \lesssim
        \frac{\Ex*{W_{n,1}^{r}} + 1}{n^{r-1}}
        \longrightarrow 0
    \end{equation}
    by \cref{asm:row_Lr}.
    Therefore,
    \begin{equation}
        \frac{1}{n\zeta_{s,\omega}^{1}}
        \sum_{i=1}^{n} h_s^{(1)}(Z_i)^2
        =
        \frac{1}{n}\sum_{i=1}^{n} W_{n,i}
        \pto 1.
    \end{equation}
\end{proof}

\begin{proof}[Proof of \cref{lem:Hajek_Dominance}]
	\begin{equation}
		\begin{aligned}
			1
			 & \leq \frac{n}{s^2}\frac{\Varb*{\mathrm{U}_{n, s, \omega}\left(D_{[n]}\right)}}{\zeta_{s, \omega}^{1}}
			= \left(\frac{s^2}{n} \zeta_{s, \omega}^{1}\right)^{-1} \left(\sum_{j=1}^{s-1}\binom{s}{j}^2\binom{n}{j}^{-1} V_{s, \omega}^{j}
			+ \binom{n}{s}^{-1} V_{s}^{s}\right)                                                                                                                    \\
			 & \leq 1 + \left(\frac{s^2}{n} \zeta_{s, \omega}^{1}\right)^{-1} \frac{s^2}{n^2} \left(\sum_{j=2}^{s-1}\binom{s}{j} V_{s, \omega}^{j} + V_{s}^{s}\right) \\
			 & \leq 1 + \frac{s}{n}\left(\frac{\zeta_{s}^{s}}{s\zeta_{s, \omega}^{1}} - 1\right)
			\longrightarrow 1.
		\end{aligned}
	\end{equation}
\end{proof}

    \subsection{Complete Generalized U-Statistics}\label{sec:appendix_delete_d_complete}
    \begin{proof}[Proof of \cref{thm:JK_Consistency_compl}]\mbox{}\\*
    For each delete set \(\ell \in L_{n,d}\), write
    \begin{equation}
        \Delta_{\ell}
        \defeq
        \Delta_{\ell}[U_{n,s,\omega}].
    \end{equation}
    Then the delete-\(d\) jackknife variance estimator takes the form
    \begin{equation}
        \hat{\sigma}_{JKD}^2\left(D_{[n]}; d\right)
        =
        \frac{n-d}{d}\binom{n}{d}^{-1}\sum_{\ell \in L_{n,d}} \Delta_{\ell}^{2}.
    \end{equation}
    We first expand each delete-\(d\) difference by Hoeffding order and isolate its first-projection, or H\'ajek, component.
    We then show that the remaining higher-order terms are negligible on the \(\zeta_{s,\omega}^{1}\) scale, identify the averaged square of the leading term with the empirical average of \(h_i^2\) up to negligible off-diagonal corrections, and finally invoke \cref{lem:first_proj_square_lln,lem:peng1}.
    The last step is to compare the resulting scale back to the actual variance target via \cref{lem:Hajek_Dominance}.
    By the generalized Hoeffding decomposition,
    \begin{equation}
        \Delta_{\ell}
        =
        \sum_{j=1}^{s-1}\binom{s}{j}\delta_{j,\ell} + \delta_{s,\ell},
    \end{equation}
    where the projection-level differences \(\delta_{j,\ell}\), \(1 \leq j \leq s\), are defined in \cref{rem:delete_d_overlap}.
    To isolate the leading H\'ajek term, write
    \begin{equation}
        h_{i}
        \defeq
        h_{s,\omega}^{(1)}(Z_{i}),
        \qquad i = 1, \dotsc, n.
    \end{equation}
    Since
    \begin{equation}
        G_{1,1}^{(\ell)}
        =
        \sum_{i \in \ell} h_{i},
        \qquad
        G_{1,0}^{(\ell)}
        =
        \sum_{i \in [n]\backslash \ell} h_{i},
    \end{equation}
    and
    \begin{equation}
        q_{1}^{(n,d)} = \frac{1}{n},
        \qquad
        r_{1}^{(n,d)}
        =
        \frac{1}{n} - \frac{1}{n-d}
        =
        -\frac{d}{n(n-d)},
    \end{equation}
    the first projection contributes
    \begin{equation}
        s \delta_{1,\ell}
        =
        \frac{s}{n}\sum_{i \in \ell} h_{i}
        -
        \frac{sd}{n(n-d)}\sum_{i \in [n]\backslash \ell} h_{i}.
    \end{equation}
    We therefore write
    \begin{equation}
        \Delta_{\ell}
        =
        \frac{s\sqrt{d}}{n}
        \left[
            \frac{1}{\sqrt{d}}\sum_{i \in \ell} h_{i}
            +
            T_{\ell}
            \right],
    \end{equation}
    where
    \begin{equation}
        T_{\ell}
        \defeq
        -\frac{\sqrt{d}}{n-d}\sum_{i \in [n]\backslash \ell} h_{i}
        +
        \frac{n}{s\sqrt{d}}
        \left\{
        \sum_{j=2}^{s-1}\binom{s}{j}\delta_{j,\ell}
        +
        \delta_{s,\ell}
        \right\}.
    \end{equation}
    Consequently,
    \begin{equation}
        \hat{\sigma}_{JKD}^2\left(D_{[n]}; d\right)
        =
        \frac{(n-d)s^2}{n^2}\binom{n}{d}^{-1}
        \sum_{\ell \in L_{n,d}}
        \left[
            \frac{1}{\sqrt{d}}\sum_{i \in \ell} h_{i}
            +
            T_{\ell}
            \right]^2.
    \end{equation}
    Write
    \begin{equation}
        A_{\ell}
        \defeq
        \frac{1}{\sqrt{d}}\sum_{i \in \ell} h_{i}.
    \end{equation}
    We first show that \(T_{\ell}\) is negligible relative to the variance scale \(\zeta_{s,\omega}^{1}\).
    Since distinct Hoeffding orders are orthogonal, we have
    \begin{equation}
        \Ex*{T_{\ell}^{2}}
        =
        \frac{d}{n-d}\zeta_{s,\omega}^{1}
        +
        \frac{n^2}{s^2 d}
        \left\{
            \sum_{j=2}^{s-1}\binom{s}{j}^{2}\Ex*{\delta_{j,\ell}^{2}}
            +
            \Ex*{\delta_{s,\ell}^{2}}
        \right\}.
    \end{equation}
    The key simplification is that, once the delete-\(d\) difference has been organized by Hoeffding order and overlap strata, the second-moment calculation is mostly combinatorial.
    Distinct Hoeffding orders are orthogonal, and within a fixed canonical order the degeneracy of the projection kills the mixed overlap terms.
    What remains is to count how many tuples land in each overlap stratum and multiply by the variance attached to that order.
    The useful partition is therefore not really the pair of cases \(d<j\) and \(d \geq j\), but the simpler split between \(j\)-tuples that hit the delete set and \(j\)-tuples that avoid it.
    The exact overlap level \(a=\abs*{\iota\cap\ell}\) is only a device for counting the first class: all nonzero-overlap strata carry the same coefficient \(q_j^{(n,d)}\), while the zero-overlap stratum carries \(r_j^{(n,d)}\).
    With the standard convention that impossible binomial coefficients are zero, Vandermonde's identity gives
    \[
        \sum_{a=1}^{\min\{j,d\}}\binom{d}{a}\binom{n-d}{j-a}
        =
        \binom{n}{j}-\binom{n-d}{j},
    \]
    while \(\abs*{\calO_{j,0}(\ell)}=\binom{n-d}{j}\).
    This is why the two apparent regimes collapse to the same coefficient after squaring and taking expectations.
    For each \(2 \leq j \leq s-1\), the overlap-stratum representation and degeneracy of the canonical \(j\)th projection imply
    \begin{equation}
        \begin{aligned}
            \Ex*{\delta_{j,\ell}^{2}}
             & =
            \left[
                \left(q_{j}^{(n,d)}\right)^{2}
                \sum_{a=1}^{\min\{j,d\}} \abs*{\calO_{j,a}(\ell)}
                +
                \left(r_{j}^{(n,d)}\right)^{2}\abs*{\calO_{j,0}(\ell)}
            \right] V_{s,\omega}^{j} \\
             & =
            \left[
                \binom{n-d}{j}^{-1}
                -
                \binom{n}{j}^{-1}
            \right] V_{s,\omega}^{j},
        \end{aligned}
    \end{equation}
    and likewise
    \begin{equation}
        \Ex*{\delta_{s,\ell}^{2}}
        =
        \left[
            \binom{n-d}{s}^{-1}
            -
            \binom{n}{s}^{-1}
        \right] V_{s}^{s}.
    \end{equation}
    Substituting these identities gives
    \begin{equation}
        \begin{aligned}
            \Ex*{T_{\ell}^{2}}
             & =
            \frac{d}{n-d}\zeta_{s,\omega}^{1} \\
             & \quad +
            \frac{n}{s d}\Bigg\{
            \sum_{j=2}^{s-1}
            \frac{\binom{s-1}{j-1}}{\binom{n-1}{j-1}}
            \left(
            \binom{n}{j}\binom{n-d}{j}^{-1} - 1
            \right)\binom{s}{j}V_{s,\omega}^{j} \\
             & \qquad\qquad +
            \binom{n-1}{s-1}^{-1}
            \left(
            \binom{n}{s}\binom{n-d}{s}^{-1} - 1
            \right)V_{s}^{s}
            \Bigg\}.
        \end{aligned}
    \end{equation}
    For \(2 \leq j \leq s\), write
    \(B_{j,n,d}\) for the amount by which the \(j\)-tuple normalization changes after deleting \(d\) observations.
    The theorem's row-wise delete-domain condition \(s+d\leq n\) ensures that the denominator \(\binom{n-d}{j}\) is nonzero for every \(2\leq j\leq s\) considered here.
    Since each deleted position affects a \(j\)-tuple only at order \(j/n\), the product below formalizes the heuristic that the cumulative perturbation is \(O(jd/n)\).
    \begin{equation}
        B_{j,n,d}
        \defeq
        \binom{n}{j}\binom{n-d}{j}^{-1}
        =
        \prod_{m=0}^{d-1}\frac{n-m}{n-j-m}
        =
        \prod_{m=0}^{d-1}\left(1 + \frac{j}{n-j-m}\right).
    \end{equation}
    Since \(sd = o(n)\), we have \(jd/(n-j-d+1) \leq 1/2\) uniformly for \(2 \leq j \leq s\) and \(n\) large enough.
    A Weierstrass-product bound then yields
    \begin{equation}
        B_{j,n,d} - 1
        \lesssim
        \frac{2jd}{n-j-d+1}
        \lesssim
        \frac{3jd}{n},
        \qquad 2 \leq j \leq s.
    \end{equation}
    Using also
    \begin{equation}
        \frac{\binom{s-1}{j-1}}{\binom{n-1}{j-1}}
        \leq
        \left(\frac{e(s-1)}{n-1}\right)^{j-1},
    \end{equation}
    we obtain
    \begin{equation}
        \begin{aligned}
            \Ex*{T_{\ell}^{2}}
             & \lesssim
            \frac{d}{n-d}\zeta_{s,\omega}^{1}
            + \frac{3}{s}\Bigg\{
            \sum_{j=2}^{s-1} j
            \left(\frac{e(s-1)}{n-1}\right)^{j-1}
            \binom{s}{j}V_{s,\omega}^{j}
            +
            s\left(\frac{e(s-1)}{n-1}\right)^{s-1}V_{s}^{s}
            \Bigg\}.
        \end{aligned}
    \end{equation}
    Splitting \(j = 1 + (j-1)\), bounding the first part by \(\frac{e(s-1)}{n-1}\), and using
    \begin{equation}
        \zeta_{s}^{s}
        =
        \sum_{j=1}^{s-1}\binom{s}{j}V_{s,\omega}^{j}
        +
        V_{s}^{s},
    \end{equation}
    we further obtain
    \begin{equation}
        \begin{aligned}
            \Ex*{T_{\ell}^{2}}
             & \lesssim
            \frac{d}{n-d}\zeta_{s,\omega}^{1}
            + \frac{3e(s-1)}{s(n-1)}
            \left(\zeta_{s}^{s} - s\zeta_{s,\omega}^{1}\right) \\
             & \quad +
            \frac{3}{s}\zeta_{s}^{s}
            \sum_{j=1}^{\infty}
            j\left(\frac{e(s-1)}{n-1}\right)^{j}.
        \end{aligned}
    \end{equation}
    Evaluating the geometric derivative series gives
    \begin{equation}
        \begin{aligned}
            \Ex*{T_{\ell}^{2}}
             & \lesssim
            \Bigg[
                \frac{d}{n-d}
                +
                \frac{3e(s-1)(n-1)}{\left(n-1-e(s-1)\right)^{2}}
            \Bigg]\zeta_{s,\omega}^{1} \\
             & \quad +
            \frac{3e(s-1)}{s}
            \Bigg[
                \frac{1}{n-1}
                +
                \frac{n-1}{\left(n-1-e(s-1)\right)^{2}}
            \Bigg]
            \left(\zeta_{s}^{s} - s\zeta_{s,\omega}^{1}\right).
        \end{aligned}
    \end{equation}
    Therefore,
    \begin{equation}
        \frac{\Ex*{T_{\ell}^{2}}}{\zeta_{s,\omega}^{1}}
        \longrightarrow 0
    \end{equation}
    by \(sd = o(n)\), \(s = o(n)\), and \cref{asm:hajek_dominance}.
    Now define
    \begin{equation}
        X_{\ell}
        \defeq
        \sqrt{
            \frac{n-d}{n\binom{n}{d}\zeta_{s,\omega}^{1}}
        }\,A_{\ell},
        \qquad
        Y_{\ell}
        \defeq
        \sqrt{
            \frac{n-d}{n\binom{n}{d}\zeta_{s,\omega}^{1}}
        }\,T_{\ell}.
    \end{equation}
    Then
    \begin{equation}
        \sum_{\ell \in L_{n,d}} \left(X_{\ell} + Y_{\ell}\right)^{2}
        =
        \frac{n\hat{\sigma}_{JKD}^{2}\left(D_{[n]}; d\right)}{s^2\zeta_{s,\omega}^{1}}.
    \end{equation}
    Moreover, by exchangeability,
    \begin{equation}
        \sum_{\ell \in L_{n,d}} \Ex*{Y_{\ell}^{2}}
        =
        \frac{n-d}{n\zeta_{s,\omega}^{1}}\Ex*{T_{\ell}^{2}}
        \longrightarrow 0,
    \end{equation}
    and
    \begin{equation}
        \sum_{\ell \in L_{n,d}} \Ex*{X_{\ell}^{2}}
        =
        \frac{n-d}{n\zeta_{s,\omega}^{1}}\Ex*{A_{\ell}^{2}}
        =
        \frac{n-d}{n}
        \longrightarrow 1.
    \end{equation}
    It remains to verify \(\sum_{\ell \in L_{n,d}} X_{\ell}^{2} \pto 1\).
    This is the diagonal-square step.
    Averaging \(A_{\ell}^{2}\) over all delete sets extracts the empirical average of the squared first-projection values, while the cross-products are washed out by the combinatorial averaging over deletion patterns.
    In that sense, the jackknife behaves here as though it were applied to a linear statistic with influence values \(h_1,\dotsc,h_n\).
    Here the exact averaged-square identity is
    \begin{equation}
        \begin{aligned}
            \binom{n}{d}^{-1}\sum_{\ell \in L_{n,d}} A_{\ell}^{2}
             & =
            \frac{1}{d}\binom{n}{d}^{-1}
            \sum_{\ell \in L_{n,d}}\sum_{i \in \ell}\sum_{j \in \ell} h_{i}h_{j} \\
             & =
            \frac{1}{n}\sum_{i=1}^{n} h_{i}^{2}
            +
            \frac{d-1}{n(n-1)}
            \sum_{i \neq j} h_{i}h_{j}.
        \end{aligned}
    \end{equation}
    Note that despite $d > 1$, the averaged-square identity reduces the diagonal to $n^{-1}\sum_i h_i^2$---the same empirical average of squared first-projection values that appears in the delete-$1$ case and is controlled by \cref{asm:row_Lr}.
    The off-diagonal term is negligible because
    \begin{equation}
        \begin{aligned}
            \Ex*{\abs*{
            \frac{d-1}{n(n-1)}
            \sum_{i \neq j} h_{i}h_{j}
            }}
             & \leq
            \frac{d-1}{n(n-1)}
            \Ex*{\abs*{
            \left(\sum_{i=1}^{n} h_{i}\right)^{2}
            -
            \sum_{i=1}^{n} h_{i}^{2}
            }} \\
             & \leq
            \frac{2(d-1)}{n-1}\zeta_{s,\omega}^{1}
            =
            o\left(\zeta_{s,\omega}^{1}\right).
        \end{aligned}
    \end{equation}
    Hence
    \begin{equation}
        \sum_{\ell \in L_{n,d}} X_{\ell}^{2}
        =
        \frac{n-d}{n\zeta_{s,\omega}^{1}}
        \left[
            \frac{1}{n}\sum_{i=1}^{n} h_{i}^{2}
            +
            o_{p}\left(\zeta_{s,\omega}^{1}\right)
        \right].
    \end{equation}
    By \cref{lem:first_proj_square_lln},
    \begin{equation}
        \frac{1}{n\zeta_{s,\omega}^{1}}\sum_{i=1}^{n} h_{i}^{2} \pto 1.
    \end{equation}
    Granting this input, \(\sum_{\ell \in L_{n,d}} X_{\ell}^{2} \pto 1\), so \cref{lem:peng1} yields
    \begin{equation}
        \frac{n\hat{\sigma}_{JKD}^{2}\left(D_{[n]}; d\right)}{s^2\zeta_{s,\omega}^{1}}
        \pto
        1.
    \end{equation}
    The desired ratio consistency then follows from \cref{lem:Hajek_Dominance}.
\end{proof}

    \subsection{Incomplete Generalized U-Statistics}\label{sec:appendix_delete_d_incomplete}
    \begin{lemma}[Reciprocal binomial good-count bounds]\label{lem:reciprocal_binomial_good_count}\mbox{}\\*
    Let \(K=1+B\), where \(B\sim\mathrm{Binomial}(m-1,p)\), and let \(\mu=mp\to\infty\).
    Define \(G=\set*{\mu/2\leq K\leq 2\mu}\).
    Then there are universal constants \(C,c>0\) such that, for all large \(\mu\),
    \begin{equation}
        \Pb*{G^c}
        \leq
        C\exp(-c\mu),
        \qquad
        \Ex*{K\1*{G^c}}
        \leq
        C\mu\exp(-c\mu),
        \qquad
        \Ex*{K^2\1*{G^c}}
        \leq
        C\mu^2\exp(-c\mu),
    \end{equation}
    and
    \begin{equation}
        \Ex*{K^{-3}}
        \leq
        C\mu^{-3},
        \qquad
        \Ex*{(K-\mu)^2K^{-5}}
        \leq
        C\mu^{-4}.
    \end{equation}
\end{lemma}

\begin{proof}
    Chernoff's inequality gives \(\Pb*{G^c}\leq C\exp(-c\mu)\), after reducing \(c\) to absorb the deterministic shift from \(B\) to \(K=1+B\).
    The same Chernoff bound and the standard identity
    \[
        \Ex*{K\1*{G^c}}
        =
        \int_0^\infty
        \Pb*{K\1*{G^c}>t}
        \,\dd t
    \]
    give \(\Ex*{K\1*{G^c}}\leq C\mu\exp(-c\mu)\), again with a possibly smaller \(c\).
    Applying the same integration argument to \(K^2\1*{G^c}\) gives \(\Ex*{K^2\1*{G^c}}\leq C\mu^2\exp(-c\mu)\).
    On \(G\), \(K^{-3}\leq C\mu^{-3}\) and \(K^{-5}\leq C\mu^{-5}\).
    On \(G^c\), the trivial bound \(K^{-3}\leq 1\) and the tail bound show that the contribution is exponentially small, hence \(O(\mu^{-3})\).
    Similarly,
    \[
        \Ex*{(K-\mu)^2K^{-5}\1*{G}}
        \leq
        C\mu^{-5}\Ex*{(K-\mu)^2}
        \leq
        C\mu^{-4},
    \]
    while on \(G^c\), \((K-\mu)^2K^{-5}\leq (K+\mu)^2\) and the preceding tail estimates make the contribution exponentially small.
\end{proof}

\begin{lemma}[Zero-count centering transfer]\label{lem:incompl_zero_count_centering_transfer}\mbox{}\\*
    Let \(h_s^0=h_s-\theta_s\), where \(\theta_s=\Ex*{h_s(D_{[s]};\omega)}\).
    Consider the incomplete statistic and delete-\(d\) jackknife estimator under the zero-count conventions in \cref{def:gen_ustat,def:jk_var_estimators}, using the same Bernoulli variables and auxiliary randomness for \(h_s\) and \(h_s^0\).
    Suppose \(N_d^\circ/N\to1\), \(N/n\to\infty\), \(\theta_s\) is bounded, and
    \[
        \frac{1}{\zeta_{s,\omega}^{1}}
        =
        o\left(\frac{Ns}{n}\right).
    \]
    If the centered-kernel statistic satisfies
    \[
        \frac{n\hat\sigma_{JKD,h^0}^{2}}{s^2\zeta_{s,\omega}^{1}}
        =
        O_p(1),
        \qquad
        \frac{n\Varb*{U_{h^0}}}{s^2\zeta_{s,\omega}^{1}}
        =
        O(1),
    \]
    then
    \begin{equation}
        \frac{n}{s^2\zeta_{s,\omega}^{1}}
        \abs*{
            \hat\sigma_{JKD,h}^{2}
            -
            \hat\sigma_{JKD,h^0}^{2}
        }
        \pto
        0
    \end{equation}
    and
    \begin{equation}
        \frac{n}{s^2\zeta_{s,\omega}^{1}}
        \abs*{
            \Varb*{U_h}
            -
            \Varb*{U_{h^0}}
        }
        \longrightarrow
        0.
    \end{equation}
\end{lemma}

\begin{proof}
    Write \(A=\set*{\hat N>0}\) and \(B_\ell=\set*{\hat N_\ell^\circ>0}\).
    Since \(B_\ell\subset A\), the zero-count convention gives
    \[
        U_h=U_{h^0}+\theta_s\1*{A},
        \qquad
        U_{h,\ell}=U_{h^0,\ell}+\theta_s\1*{B_\ell},
    \]
    and hence
    \[
        \Delta_{\ell,h}
        =
        \Delta_{\ell,h^0}
        +
        C_\ell,
        \qquad
        C_\ell
        \defeq
        \theta_s\1*{A\cap B_\ell^c}.
    \]
    If \(H_1\) is the number of selected subsamples intersecting the delete set, then
    \[
        A\cap B_\ell^c
        =
        \set*{\hat N_\ell^\circ=0,\ H_1>0}.
    \]
    The avoid and hit Bernoulli layers are independent, so
    \[
        \Pb*{A\cap B_\ell^c}
        =
        (1-p)^{\binom{n-d}{s}}
        \left\{
            1-(1-p)^{\binom{n}{s}-\binom{n-d}{s}}
        \right\}
        \leq
        \exp(-N_d^\circ).
    \]
    Define
    \[
        B_n
        \defeq
        \frac{n}{s^2\zeta_{s,\omega}^{1}}
        \frac{n-d}{d}\binom{n}{d}^{-1}
        \sum_{\ell\in L_{n,d}} C_\ell^2.
    \]
    By exchangeability and the preceding bound,
    \[
        \Ex*{B_n}
        \leq
        \theta_s^2
        \frac{n(n-d)}{d\,s^2\zeta_{s,\omega}^{1}}
        \exp(-N_d^\circ).
    \]
    The assumptions imply this upper bound tends to zero: indeed
    \(1/\zeta_{s,\omega}^{1}=o(Ns/n)\), \(d\geq1\), \(s\geq1\), and \(N_d^\circ\asymp N\) reduce the polynomial factor to \(o(Nn)\), which is dominated by \(\exp(-N_d^\circ)\) because \(N/n\to\infty\).
    Hence \(B_n\pto0\).
    With
    \[
        A_n
        \defeq
        \frac{n\hat\sigma_{JKD,h^0}^{2}}{s^2\zeta_{s,\omega}^{1}},
    \]
    Cauchy--Schwarz gives
    \[
        \frac{n}{s^2\zeta_{s,\omega}^{1}}
        \abs*{\hat\sigma_{JKD,h}^{2}-\hat\sigma_{JKD,h^0}^{2}}
        \leq
        2A_n^{1/2}B_n^{1/2}+B_n
        \pto
        0.
    \]

    For the variance target, \(U_h=U_{h^0}+\theta_s\1*{A}\) and
    \(\Varb*{\theta_s\1*{A}}\leq\theta_s^2\Pb*{A^c}\leq\theta_s^2\exp(-N)\).
    The inequality
    \[
        \abs*{\Varb*{X+Y}-\Varb*{X}}
        \leq
        2\sqrt{\Varb*{X}\Varb*{Y}}+\Varb*{Y}
    \]
    with \(X=U_{h^0}\) and \(Y=\theta_s\1*{A}\), together with the centered variance tightness and the same exponential domination argument, proves the variance-transfer claim.
\end{proof}

\begin{proof}[Proof of \cref{thm:JK_Consistency_incompl}]\mbox{}\\*
    We first prove the result for the centered kernel, so \(\theta_{s}=0\) throughout the main normalization argument.
    The zero-count centering-transfer step at the end removes this temporary restriction.
    The reciprocal-count expressions below use the zero-count conventions from \cref{def:gen_ustat,def:jk_var_estimators}.
    Whenever a selected kernel term appears inside the expectation, the corresponding realized count is automatically positive, so the reciprocal manipulations are on positive-count events.
    Let
    \begin{equation}
        \hbar_{s}\left(D_{[s]}; \rho, \omega\right)
        \defeq
        \frac{\rho}{p} h_{s}\left(D_{[s]}; \omega\right),
        \qquad
        p
        \defeq
        \frac{N}{\binom{n}{s}},
    \end{equation}
    and define the Horvitz-Thompson normalized auxiliary complete generalized U-statistic
    \begin{equation}
        \overline{U}_{n,s,\rho,\omega}\left(D_{[n]}\right)
        \defeq
        \binom{n}{s}^{-1}
        \sum_{\iota \in L_{n,s}}
        \hbar_{s}\left(D_{\iota}; \rho_{\iota}, \omega\right)
        =
        \frac{1}{N}
        \sum_{\iota \in L_{n,s}}
        \rho_{\iota} h_{s}\left(D_{\iota}; \omega\right).
    \end{equation}
    This differs from \(U_{n,s,N,\omega}\) only through the denominator: the random count \(\hat N\) is replaced by its expectation \(N\).
    Working first with \(\overline U_{n,s,\rho,\omega}\) separates the random-normalization effect from the delete-\(d\) Hoeffding expansion.
    The overlap decomposition can then be carried out for a kernel with the same first projection as the original incomplete statistic, and the comparison between \(N\) and \(\hat N\) is handled afterward as a separate normalization step.

    By construction, for every \(1 \leq j \leq s-1\),
    \begin{equation}
        \hbar_{s}^{(j)}\left(D_{[j]}\right)
        =
        h_{s}^{(j)}\left(D_{[j]}\right).
    \end{equation}
    In particular,
    \begin{equation}
        \hbar_{s}^{(1)}(Z_1)
        =
        h_{s}^{(1)}(Z_1),
    \end{equation}
    so the normalized first-projection square is unchanged:
    \begin{equation}
        \frac{\hbar_{s}^{(1)}(Z_1)^2}{\zeta_{s,\omega}^{1}}
        =
        \frac{h_{s}^{(1)}(Z_1)^2}{\zeta_{s,\omega}^{1}}.
    \end{equation}
    Consequently, the Horvitz-Thompson kernel has the same first-projection scale and satisfies the same row-wise square law.
    It remains to show that Bernoulli sampling noise is a negligible remainder.

    Write \(\overline{V}_{s}^{s}\) and \(\overline{\zeta}_{s}^{s}\) for the final-order variance terms corresponding to \(\hbar_{s}\).
    Since the lower Hoeffding orders are unchanged and \(\theta_{s}=0\),
    \begin{equation}
        \overline{\zeta}_{s}^{s}
        =
        \Varb*{\hbar_{s}\left(D_{[s]}; \rho, \omega\right)}
        =
        \Varb*{\frac{\rho}{p} h_{s}\left(D_{[s]}; \omega\right)}
        =
        \frac{1}{p}\zeta_{s}^{s},
    \end{equation}
    and therefore
    \begin{equation}
        \overline{V}_{s}^{s}
        =
        V_{s}^{s} + \frac{1-p}{p}\zeta_{s}^{s}.
    \end{equation}
    Thus Horvitz-Thompson normalization leaves every lower Hoeffding order untouched and only inflates the final, purely \(s\)-fold component.

    Let \(\overline{\sigma}_{n}^{2} \defeq \Varb*{\overline{U}_{n,s,\rho,\omega}\left(D_{[n]}\right)}\).
    The variance decomposition from the proof of \cref{lem:Hajek_Dominance} yields
    \begin{equation}
        \begin{aligned}
            1
             & \leq
            \frac{n}{s^{2}}
            \frac{\overline{\sigma}_{n}^{2}}{\zeta_{s,\omega}^{1}} \\
             & =
            \left(\frac{s^{2}}{n}\zeta_{s,\omega}^{1}\right)^{-1}
            \left(
            \sum_{j=1}^{s-1}
            \binom{s}{j}^{2}\binom{n}{j}^{-1}V_{s,\omega}^{j}
            +
            \binom{n}{s}^{-1}\overline{V}_{s}^{s}
            \right) \\
             & \leq
            1
            +
            \frac{s}{n}
            \left(
            \frac{\zeta_{s}^{s}}{s\zeta_{s,\omega}^{1}} - 1
            \right)
            +
            \frac{n}{s^{2}\zeta_{s,\omega}^{1}}
            \binom{n}{s}^{-1}\frac{1-p}{p}\zeta_{s}^{s} \\
             & \leq
            1
            +
            \frac{s}{n}
            \left(
            \frac{\zeta_{s}^{s}}{s\zeta_{s,\omega}^{1}} - 1
            \right)
            +
            \frac{n\zeta_{s}^{s}}{N s^{2}\zeta_{s,\omega}^{1}}
            \longrightarrow 1,
        \end{aligned}
    \end{equation}
    where the last step uses \cref{asm:hajek_dominance,asm:AS_Sampling} and the boundedness of \(\zeta_{s}^{s}\).
    Thus,
    \begin{equation}\label{eq:var_scale_u_bar}
        \frac{n}{s^{2}}
        \frac{\overline{\sigma}_{n}^{2}}{\zeta_{s,\omega}^{1}}
        \longrightarrow 1.
    \end{equation}
    So the Horvitz-Thompson normalized statistic already lives on the same variance scale as the complete statistic.
    The asymptotically-sufficient sampling condition is used here in its most transparent role: it says that the extra top-order variance created by Bernoulli thinning is too small to compete with the first projection.

    We now analyze the delete-\(d\) jackknife built from \(\overline U_{n,s,\rho,\omega}\).
    For each delete set \(\ell \in L_{n,d}\), define
    \begin{equation}
        \overline{\Delta}_{\ell}
        \defeq
        \Delta_{\ell}[\overline{U}_{n,s,\rho,\omega}]
        =
        \overline{U}_{n,s,\rho,\omega}\left(D_{[n]}\right)
        -
        \overline{U}_{n,s,\rho,\omega}\left(D_{[n],-\ell}\right).
    \end{equation}
    By the generalized Hoeffding decomposition,
    \begin{equation}
        \overline{\Delta}_{\ell}
        =
        \sum_{j=1}^{s-1}\binom{s}{j}\delta_{j,\ell}
        +
        \overline{\delta}_{s,\ell},
    \end{equation}
    where the lower-order terms \(\delta_{j,\ell}\), \(1 \leq j \leq s-1\), are exactly the complete-case projection-level differences from \cref{rem:delete_d_overlap}, while the final-order term becomes
    \begin{equation}
        \overline{\delta}_{s,\ell}
        =
        q_{s}^{(n,d)}
        \sum_{a=1}^{\min\{s,d\}} \overline{G}_{s,a}^{(\ell)}
        +
        r_{s}^{(n,d)} \overline{G}_{s,0}^{(\ell)},
    \end{equation}
    with \(\overline{G}_{s,a}^{(\ell)}\) denoting the overlap-stratum sums built from \(\hbar_{s}^{(s)}\).
    Thus the delete-\(d\) expansion differs from the complete-case expansion only through the final-order term.
    All lower-order overlap combinatorics and leading-term algebra are unchanged, so it remains to bound this modified top-order contribution.

    Write
    \begin{equation}
        h_{i}
        \defeq
        h_{s,\omega}^{(1)}(Z_i)
        =
        \hbar_{s}^{(1)}(Z_i),
        \qquad i = 1,\dotsc,n.
    \end{equation}
    Since the first projection is unchanged, the same algebra as in the proof of \cref{thm:JK_Consistency_compl} gives
    \begin{equation}
        \overline{\Delta}_{\ell}
        =
        \frac{s\sqrt d}{n}
        \left[
            A_{\ell}
            +
            \overline{T}_{\ell}
        \right],
    \end{equation}
    where
    \begin{equation}
        A_{\ell}
        \defeq
        \frac{1}{\sqrt d}
        \sum_{i \in \ell} h_{i}
    \end{equation}
    and
    \begin{equation}
        \overline{T}_{\ell}
        \defeq
        -\frac{\sqrt d}{n-d}
        \sum_{i \in [n]\backslash \ell} h_{i}
        +
        \frac{n}{s\sqrt d}
        \left\{
            \sum_{j=2}^{s-1}\binom{s}{j}\delta_{j,\ell}
            +
            \overline{\delta}_{s,\ell}
        \right\}.
    \end{equation}

    Orthogonality of distinct Hoeffding orders yields
    \begin{equation}
        \Ex*{\overline{T}_{\ell}^{2}}
        =
        \frac{d}{n-d}\zeta_{s,\omega}^{1}
        +
        \frac{n^{2}}{s^{2}d}
        \left\{
            \sum_{j=2}^{s-1}\binom{s}{j}^{2}\Ex*{\delta_{j,\ell}^{2}}
            +
            \Ex*{\overline{\delta}_{s,\ell}^{2}}
        \right\}.
    \end{equation}
    For \(2 \leq j \leq s-1\), the lower-order terms are unchanged, so
    \begin{equation}
        \Ex*{\delta_{j,\ell}^{2}}
        =
        \left[
            \binom{n-d}{j}^{-1}
            -
            \binom{n}{j}^{-1}
        \right] V_{s,\omega}^{j},
    \end{equation}
    while the final-order term becomes
    \begin{equation}
        \Ex*{\overline{\delta}_{s,\ell}^{2}}
        =
        \left[
            \binom{n-d}{s}^{-1}
            -
            \binom{n}{s}^{-1}
        \right]\overline{V}_{s}^{s}.
    \end{equation}
    Using the same product and geometric-series bounds as in the complete-case analysis gives
    \begin{equation}
        \begin{aligned}
            \Ex*{\overline{T}_{\ell}^{2}}
             & \lesssim
            \Bigg[
                \frac{d}{n-d}
                +
                \frac{3e(s-1)(n-1)}{\left(n-1-e(s-1)\right)^{2}}
            \Bigg]\zeta_{s,\omega}^{1} \\
             & \quad +
            \frac{3e(s-1)}{s}
            \Bigg[
                \frac{1}{n-1}
                +
                \frac{n-1}{\left(n-1-e(s-1)\right)^{2}}
            \Bigg]
            \left(
                \zeta_{s}^{s} - s\zeta_{s,\omega}^{1}
            \right) \\
             & \quad +
            3\binom{n-1}{s-1}^{-1}\frac{1-p}{p}\zeta_{s}^{s}.
        \end{aligned}
    \end{equation}
    The last term is the only new Bernoulli-sampling contribution, and
    \begin{equation}
        \binom{n-1}{s-1}^{-1}\frac{1-p}{p}
        =
        \binom{n-1}{s-1}^{-1}
        \left(
            \frac{\binom{n}{s}}{N} - 1
        \right)
        =
        \frac{n}{Ns}
        -
        \binom{n-1}{s-1}^{-1}.
    \end{equation}
    Therefore,
    \begin{equation}
        \frac{1}{\zeta_{s,\omega}^{1}}
        \binom{n-1}{s-1}^{-1}\frac{1-p}{p}\zeta_{s}^{s}
        \longrightarrow 0
    \end{equation}
    by \cref{asm:AS_Sampling} and the boundedness of \(\zeta_{s}^{s}\).
    The remaining terms satisfy the corresponding complete-case bounds, so
    \begin{equation}\label{eq:overline_T_negligible}
        \frac{\Ex*{\overline{T}_{\ell}^{2}}}{\zeta_{s,\omega}^{1}}
        \longrightarrow 0.
    \end{equation}
    Thus, after subtracting off the linear H\'ajek piece, the remainder is negligible on the \(\zeta_{s,\omega}^{1}\) scale.
    The additional Bernoulli-sampling contribution is absorbed by \cref{asm:AS_Sampling}.

    Define
    \begin{equation}
        \overline{X}_{\ell}
        \defeq
        \sqrt{
            \frac{n-d}{n\binom{n}{d}\zeta_{s,\omega}^{1}}
        }\,A_{\ell},
        \qquad
        \overline{Y}_{\ell}
        \defeq
        \sqrt{
            \frac{n-d}{n\binom{n}{d}\zeta_{s,\omega}^{1}}
        }\,\overline{T}_{\ell}.
    \end{equation}
    Then
    \begin{equation}
        \sum_{\ell \in L_{n,d}}
        \left(
            \overline{X}_{\ell} + \overline{Y}_{\ell}
        \right)^{2}
        =
        \frac{n\widehat{\overline{\sigma}}_{JKD}^{2}\left(D_{[n]}; d, \rho, \omega\right)}
        {s^{2}\zeta_{s,\omega}^{1}},
    \end{equation}
    where
    \begin{equation}
        \widehat{\overline{\sigma}}_{JKD}^{2}\left(D_{[n]}; d, \rho, \omega\right)
        \defeq
        \frac{n-d}{d}\binom{n}{d}^{-1}
        \sum_{\ell \in L_{n,d}}
        \overline{\Delta}_{\ell}^{2}.
    \end{equation}
    Moreover, \cref{eq:overline_T_negligible} implies
    \begin{equation}
        \sum_{\ell \in L_{n,d}} \Ex*{\overline{Y}_{\ell}^{2}}
        =
        \frac{n-d}{n\zeta_{s,\omega}^{1}}
        \Ex*{\overline{T}_{\ell}^{2}}
        \longrightarrow 0.
    \end{equation}
    The corresponding statements for \(\overline X_{\ell}\) are exactly the same as in the complete-case proof because \(A_{\ell}\) is unchanged:
    \begin{equation}
        \sum_{\ell \in L_{n,d}} \Ex*{\overline{X}_{\ell}^{2}}
        \longrightarrow 1,
        \qquad
        \sum_{\ell \in L_{n,d}} \overline{X}_{\ell}^{2}
        \pto 1,
    \end{equation}
    where the second convergence again follows from the exact averaged-square identity and \cref{lem:first_proj_square_lln}.
    Hence \cref{lem:peng1} gives
    \begin{equation}\label{eq:scaled_ht_jkd}
        \frac{n\widehat{\overline{\sigma}}_{JKD}^{2}\left(D_{[n]}; d, \rho, \omega\right)}
        {s^{2}\zeta_{s,\omega}^{1}}
        \pto 1.
    \end{equation}
    The limit in \cref{eq:scaled_ht_jkd} is the Horvitz-Thompson target that remains to transfer to realized-count normalization.
    At this point the delete-\(d\) jackknife has already been shown to work for the Horvitz-Thompson normalized statistic.
    What remains is no longer a Hoeffding-decomposition problem.
    It is purely a question of whether replacing the deterministic counts \(N\) and \(N_{d}^{\circ}\) by their realized versions \(\hat N\) and \(\hat N_{\ell}^{\circ}\) can change the jackknife differences at first order.

    It remains to replace the Horvitz-Thompson normalization by the empirical counts.
    For each delete set \(\ell\), define
    \begin{equation}
        \hat N_{\ell}^{\circ}
        \defeq
        \sum_{\iota \in \calO_{s,0}(\ell)} \rho_{\iota},
        \qquad
        N_{d}^{\circ}
        \defeq
        \Ex*{\hat N_{\ell}^{\circ}}
        =
        p\binom{n-d}{s}
        =
        N\binom{n-d}{s}\binom{n}{s}^{-1}.
    \end{equation}
    Then
    \begin{equation}
        U_{n,s,N,\omega}\left(D_{[n]}\right)
        -
        \overline{U}_{n,s,\rho,\omega}\left(D_{[n]}\right)
        =
        \left(
            \hat N^{-1} - N^{-1}
        \right)
        \sum_{\iota \in L_{n,s}}
        \rho_{\iota} h_{s}\left(D_{\iota}; \omega\right),
    \end{equation}
    and
    \begin{equation}
        U_{n,s,N,\omega}\left(D_{[n],-\ell}\right)
        -
        \overline{U}_{n,s,\rho,\omega}\left(D_{[n],-\ell}\right)
        =
        \left(
            \left(\hat N_{\ell}^{\circ}\right)^{-1}
            -
            \left(N_{d}^{\circ}\right)^{-1}
        \right)
        \sum_{\iota \in \calO_{s,0}(\ell)}
        \rho_{\iota} h_{s}\left(D_{\iota}; \omega\right).
    \end{equation}
    The actual incomplete statistic and its Horvitz-Thompson analogue differ only through the reciprocal-count factors.
    It remains to show that count fluctuations are negligible after jackknife rescaling.
    We start with the full-sample normalization error.
    Write
    \begin{equation}
        S
        \defeq
        \sum_{\iota \in L_{n,s}}
        \rho_{\iota} h_{s}\left(D_{\iota}; \omega\right).
    \end{equation}
    By Cauchy-Schwarz,
    \begin{equation}
        S^{2}
        \leq
        \hat N
        \sum_{\iota \in L_{n,s}}
        \rho_{\iota}
        h_{s}\left(D_{\iota}; \omega\right)^{2}.
    \end{equation}
    Therefore,
    \begin{equation}
        \begin{aligned}
            \Ex*{
                \left(
                    U_{n,s,N,\omega}\left(D_{[n]}\right)
                    -
                    \overline{U}_{n,s,\rho,\omega}\left(D_{[n]}\right)
                \right)^{2}
            }
             & =
            \Ex*{
                \left(
                    \frac{N-\hat N}{N\hat N}
                \right)^{2}
                S^{2}
            } \\
            \qquad\leq
             & \frac{1}{N^{2}}
            \Ex*{
                \frac{(N-\hat N)^{2}}{\hat N}
                \sum_{\iota \in L_{n,s}}
                \rho_{\iota}
                h_{s}\left(D_{\iota}; \omega\right)^{2}
            }.
        \end{aligned}
    \end{equation}
    By exchangeability and independence of \(\rho_{\iota}\) from the kernel,
    \begin{equation}
        \begin{aligned}
             & \Ex*{
                \left(
                    U_{n,s,N,\omega}\left(D_{[n]}\right)
                    -
                    \overline{U}_{n,s,\rho,\omega}\left(D_{[n]}\right)
                \right)^{2}
            } \\
             & \quad \leq
            \frac{\binom{n}{s}\zeta_{s}^{s}}{N^{2}}
            \Ex*{
                \frac{(N-\hat N)^{2}}{\hat N}
                \rho_{[s]}
            }.
        \end{aligned}
    \end{equation}
    Conditioning on \(\rho_{[s]}=1\), we have \(\hat N = 1 + B\) with \(B \sim \mathrm{Binomial}\left(\binom{n}{s}-1,p\right)\), so
    \begin{equation}
        \begin{aligned}
             & \Ex*{
                \frac{(N-\hat N)^{2}}{\hat N}
                \given
                \rho_{[s]} = 1
            } \\
             & \quad =
            \Ex*{\hat N \given \rho_{[s]} = 1}
            - 2N
            + N^{2}\Ex*{\hat N^{-1} \given \rho_{[s]} = 1}.
        \end{aligned}
    \end{equation}
    The binomial identity
    \begin{equation}
        \Ex*{\frac{1}{1+B}}
        =
        \frac{1-(1-p)^{\binom{n}{s}}}{\binom{n}{s}p}
        \leq
        \frac{1}{N}
    \end{equation}
    gives
    \begin{equation}
        \Ex*{
            \frac{(N-\hat N)^{2}}{\hat N}
            \given
            \rho_{[s]} = 1
        }
        \leq
        1-p.
    \end{equation}
    Consequently,
    \begin{equation}\label{eq:full_norm_error_bound}
        \Ex*{
            \left(
                U_{n,s,N,\omega}\left(D_{[n]}\right)
                -
                \overline{U}_{n,s,\rho,\omega}\left(D_{[n]}\right)
            \right)^{2}
        }
        \leq
        \frac{\zeta_{s}^{s}}{N}.
    \end{equation}

    The deleted-sample normalization error is handled in exactly the same way.
    For fixed \(\ell\), write
    \begin{equation}
        S_{0}^{(\ell)}
        \defeq
        \sum_{\iota \in \calO_{s,0}(\ell)}
        \rho_{\iota} h_{s}\left(D_{\iota}; \omega\right).
    \end{equation}
    Repeating the preceding argument with \(\hat N_{\ell}^{\circ}\) in place of \(\hat N\) and \(N_{d}^{\circ}\) in place of \(N\) yields
    \begin{equation}\label{eq:deleted_norm_error_bound}
        \Ex*{
            \left(
                U_{n,s,N,\omega}\left(D_{[n],-\ell}\right)
                -
                \overline{U}_{n,s,\rho,\omega}\left(D_{[n],-\ell}\right)
            \right)^{2}
        }
        \leq
        \frac{\zeta_{s}^{s}}{N_{d}^{\circ}}.
    \end{equation}
    Moreover,
    \begin{equation}
        \frac{N_{d}^{\circ}}{N}
        =
        \frac{\binom{n-d}{s}}{\binom{n}{s}}
        =
        \prod_{m=0}^{s-1}
        \frac{n-d-m}{n-m}
        \longrightarrow 1
    \end{equation}
    because \(sd = o(n)\).
    Thus the expected retained count changes only by a \(1+o(1)\) factor, and \cref{eq:deleted_norm_error_bound} gives the same normalization bound for deleted samples.

    For the normalization transfer, the relevant cancellation comes from separating the avoided and hit subsamples.
    For fixed \(\ell\), define
    \begin{equation}
        \calO_{s,1}(\ell)
        \defeq
        \set*{\iota \in L_{n,s}: \iota \cap \ell \neq \emptyset},
    \end{equation}
    and write
    \begin{equation}
        H_{1}
        \defeq
        \sum_{\iota \in \calO_{s,1}(\ell)} \rho_{\iota},
        \qquad
        S_{1}
        \defeq
        \sum_{\iota \in \calO_{s,1}(\ell)}
        \rho_{\iota} h_{s}\left(D_{\iota}; \omega\right),
        \qquad
        N_{1}
        \defeq
        \Ex*{H_{1}}
        =
        N - N_{d}^{\circ}.
    \end{equation}
    Thus \(\hat N = \hat N_{\ell}^{\circ} + H_{1}\) and
    \begin{equation}\label{eq:hit_fraction_bound}
        \frac{N_{1}}{N}
        =
        1
        -
        \frac{\binom{n-d}{s}}{\binom{n}{s}}
        =
        1
        -
        \prod_{m=0}^{s-1}
        \left(
            1 - \frac{d}{n-m}
        \right)
        \lesssim
        \frac{sd}{n},
    \end{equation}
    where the last step uses \(sd=o(n)\) and the elementary bound
    \(1-\prod_{m}(1-a_m)\leq \sum_m a_m\) for \(a_m \in [0,1]\).
    In particular, \cref{eq:hit_fraction_bound} gives \(N_{1}=o(N)\), while \(N_{d}^{\circ}\asymp N\).

    Now let
    \begin{equation}
        \Delta_{\ell}
        \defeq
        \Delta_{\ell}[U_{n,s,N,\omega}]
        =
        U_{n,s,N,\omega}\left(D_{[n]}\right)
        -
        U_{n,s,N,\omega}\left(D_{[n],-\ell}\right),
        \qquad
        R_{\ell}
        \defeq
        \Delta_{\ell} - \overline{\Delta}_{\ell}.
    \end{equation}
    Using \(S=S_{0}^{(\ell)}+S_{1}\) and \(\hat N=\hat N_{\ell}^{\circ}+H_{1}\),
    \begin{equation}
        R_{\ell}
        =
        \left(
            \frac{1}{\hat N}
            -
            \frac{1}{N}
        \right)S_{1}
        +
        \left(
            \frac{N_{1}}{N N_{d}^{\circ}}
            -
            \frac{H_{1}}{\hat N \hat N_{\ell}^{\circ}}
        \right)S_{0}^{(\ell)}.
    \end{equation}
    We bound the two terms on the right separately.
    First, use
    \begin{equation}
        S_{1}^{2}
        \leq
        H_{1}\sum_{\iota \in \calO_{s,1}(\ell)}
        \rho_{\iota}h_{s}(D_{\iota};\omega)^{2}
        \leq
        \hat N\sum_{\iota \in \calO_{s,1}(\ell)}
        \rho_{\iota}h_{s}(D_{\iota};\omega)^{2}.
    \end{equation}
    Then
    \begin{equation}
        \left(
            \left(
                \frac{1}{\hat N}
                -
                \frac{1}{N}
            \right)S_{1}
        \right)^{2}
        \leq
        \frac{(N-\hat N)^{2}}{N^{2}\hat N}
        \sum_{\iota \in \calO_{s,1}(\ell)} \rho_{\iota}h_{s}(D_{\iota};\omega)^{2}.
    \end{equation}
    By exchangeability within \(\calO_{s,1}(\ell)\), conditioning on any fixed hit subset being selected gives the same reciprocal-count factor as in the proof of \cref{eq:full_norm_error_bound}. Hence
    \begin{equation}
        \Ex*{
            \left(
                \left(
                    \frac{1}{\hat N}
                    -
                    \frac{1}{N}
                \right)S_{1}
            \right)^{2}
        }
        \lesssim
        \frac{N_{1}}{N^{2}}\zeta_{s}^{s}
        \lesssim
        \frac{sd}{n}\frac{\zeta_{s}^{s}}{N}.
    \end{equation}
    For the second term, on the event \(\hat N_{\ell}^{\circ}>0\), write
    \begin{equation}
        g_{k}(u)
        \defeq
        \frac{u}{k(k+u)},
        \qquad k,u>0,
    \end{equation}
    so that
    \begin{equation}
        \frac{N_{1}}{N N_{d}^{\circ}}
        -
        \frac{H_{1}}{\hat N \hat N_{\ell}^{\circ}}
        =
        g_{N_{d}^{\circ}}(N_{1})
        -
        g_{\hat N_{\ell}^{\circ}}(H_{1}).
    \end{equation}
    On the complementary zero-count event, \(S_{0}^{(\ell)}=0\) by convention, so the same product is identically zero and there is nothing to bound.
    Let
    \begin{equation}
        G_\ell
        \defeq
        \set*{
            \frac12 N_d^\circ
            \leq
            \hat N_\ell^\circ
            \leq
            2N_d^\circ
        }.
    \end{equation}
    On \(G_\ell\), \(N_d^\circ\asymp\hat N_\ell^\circ\).
    Since \(N_1=o(N_d^\circ)\), the derivatives
    \[
        \partial_u g_k(u)=(k+u)^{-2},
        \qquad
        \partial_k g_k(u)=-\frac{u(2k+u)}{k^2(k+u)^2}
    \]
    are uniformly controlled along the two line segments joining
    \((N_d^\circ,N_1)\), \((\hat N_\ell^\circ,N_1)\), and \((\hat N_\ell^\circ,H_1)\).
    Therefore
    \begin{equation}
        \abs*{
            g_{N_d^\circ}(N_1)
            -
            g_{\hat N_\ell^\circ}(H_1)
        }
        \1*{G_\ell}
        \lesssim
        \left(
            \frac{\abs*{H_1-N_1}}{(N_d^\circ)^2}
            +
            \frac{N_1\abs*{\hat N_\ell^\circ-N_d^\circ}}{(N_d^\circ)^3}
        \right)
        \1*{G_\ell}.
    \end{equation}
    Using \(\left(S_{0}^{(\ell)}\right)^2\leq\hat N_\ell^\circ\sum_{\iota\in\calO_{s,0}(\ell)}\rho_\iota h_s(D_\iota;\omega)^2\) and \(\hat N_\ell^\circ\leq2N_d^\circ\) on \(G_\ell\), we obtain
    \begin{equation}
        \begin{aligned}
            & \Ex*{
                \left[
                    \left(
                        \frac{N_{1}}{N N_{d}^{\circ}}
                        -
                        \frac{H_{1}}{\hat N \hat N_{\ell}^{\circ}}
                    \right)S_{0}^{(\ell)}
                \right]^2
                \1*{G_\ell}
            } \\
            & \quad \lesssim
            \frac{1}{(N_d^\circ)^3}
            \Ex*{
                (H_1-N_1)^2
                \sum_{\iota\in\calO_{s,0}(\ell)}
                \rho_\iota h_s(D_\iota;\omega)^2
            } \\
            & \qquad +
            \frac{N_1^2}{(N_d^\circ)^5}
            \Ex*{
                (\hat N_\ell^\circ-N_d^\circ)^2
                \sum_{\iota\in\calO_{s,0}(\ell)}
                \rho_\iota h_s(D_\iota;\omega)^2
            }.
        \end{aligned}
    \end{equation}
    Now \(H_1\) depends only on the Bernoulli variables indexed by \(\calO_{s,1}(\ell)\), while \(\hat N_\ell^\circ\) and \(S_0^{(\ell)}\) depend only on those indexed by \(\calO_{s,0}(\ell)\), so the hit count is independent of the avoid-layer quantities.
    Also,
    \[
        \Ex*{
            \sum_{\iota\in\calO_{s,0}(\ell)}
            \rho_\iota h_s(D_\iota;\omega)^2
        }
        =
        N_d^\circ\zeta_s^s
    \]
    because the centered kernel has second moment \(\zeta_s^s\).
    For the second expectation, exchangeability within the avoid layer and conditioning on a fixed avoid subset being selected gives \(\hat N_\ell^\circ=1+B_0\), with \(B_0\sim\mathrm{Binomial}(\abs*{\calO_{s,0}(\ell)}-1,p)\); hence
    \[
        \Ex*{
            (\hat N_\ell^\circ-N_d^\circ)^2
            \sum_{\iota\in\calO_{s,0}(\ell)}
            \rho_\iota h_s(D_\iota;\omega)^2
        }
        \lesssim
        (N_d^\circ)^2\zeta_s^s.
    \]
    Since \(\Varb*{H_1}=N_1(1-p)\leq N_1\), this gives
    \begin{equation}
        \Ex*{
            \left[
                \left(
                    \frac{N_{1}}{N N_{d}^{\circ}}
                    -
                    \frac{H_{1}}{\hat N \hat N_{\ell}^{\circ}}
                \right)S_{0}^{(\ell)}
            \right]^2
            \1*{G_\ell}
        }
        \lesssim
        \frac{N_1}{(N_d^\circ)^2}\zeta_s^s
        +
        \frac{N_1^2}{(N_d^\circ)^3}\zeta_s^s.
    \end{equation}
    On \(G_\ell^c\), use the crude coefficient bound
    \[
        \abs*{
            \frac{N_1}{NN_d^\circ}
            -
            \frac{H_1}{\hat N\hat N_\ell^\circ}
        }
        \leq
        \frac{N_1}{NN_d^\circ}
        +
        \frac{1}{\hat N_\ell^\circ}
    \]
    on the positive-count event.
    Together with \(S_0^{(\ell)2}\leq \hat N_\ell^\circ\sum_{\calO_{s,0}(\ell)}\rho_\iota h_s(D_\iota;\omega)^2\), \(\cref{lem:reciprocal_binomial_good_count}\), and the same exchangeability argument, this yields, after reducing the exponential constant if necessary,
    \begin{equation}
        \Ex*{
            \left[
                \left(
                    \frac{N_1}{NN_d^\circ}
                    -
                    \frac{H_1}{\hat N\hat N_\ell^\circ}
                \right)S_0^{(\ell)}
            \right]^2
            \1*{G_\ell^c}
        }
        \lesssim
        \zeta_s^s\exp(-cN_d^\circ).
    \end{equation}
    Combining the good-count and complement bounds, and using \(N_d^\circ\asymp N\), we conclude
    \begin{equation}
        \Ex*{
            \left[
                \left(
                    \frac{N_{1}}{N N_{d}^{\circ}}
                    -
                    \frac{H_{1}}{\hat N \hat N_{\ell}^{\circ}}
                \right)S_{0}^{(\ell)}
            \right]^{2}
        }
        \lesssim
        \frac{N_{1}}{\left(N_{d}^{\circ}\right)^{2}}\zeta_{s}^{s}
        +
        \frac{N_{1}^{2}}{\left(N_{d}^{\circ}\right)^{3}}\zeta_{s}^{s}
        +
        \zeta_s^s\exp(-cN_d^\circ)
    \end{equation}
    and therefore, by \cref{eq:hit_fraction_bound},
    \begin{equation}
        \Ex*{
            \left[
                \left(
                    \frac{N_{1}}{N N_{d}^{\circ}}
                    -
                    \frac{H_{1}}{\hat N \hat N_{\ell}^{\circ}}
                \right)S_{0}^{(\ell)}
            \right]^{2}
        }
        \lesssim
        \frac{sd}{n}\frac{\zeta_{s}^{s}}{N}
        +
        \zeta_s^s\exp(-cN_d^\circ).
    \end{equation}
    Combining the last two displays and \(\abs*{a+b}^{2}\leq 2(a^{2}+b^{2})\), we obtain the sharpened coupled bound
    \begin{equation}\label{eq:coupled_R_bound}
        \Ex*{R_{\ell}^{2}}
        \lesssim
        \frac{sd}{n}\frac{\zeta_{s}^{s}}{N}
        +
        \zeta_s^s\exp(-cN_d^\circ).
    \end{equation}
    Define
    \begin{equation}
        \overline{Z}_{\ell}
        \defeq
        \sqrt{
            \frac{n(n-d)}{s^{2}d\binom{n}{d}\zeta_{s,\omega}^{1}}
        }\,R_{\ell}.
    \end{equation}
    Using \cref{eq:coupled_R_bound}, exchangeability gives
    \begin{equation}
        \sum_{\ell \in L_{n,d}} \Ex*{\overline{Z}_{\ell}^{2}}
        =
        \frac{n(n-d)}{s^{2}d\zeta_{s,\omega}^{1}}
        \Ex*{R_{\ell}^{2}}
        \lesssim
        \frac{n-d}{N s\zeta_{s,\omega}^{1}}
        \zeta_{s}^{s}
        +
        \frac{n(n-d)}{s^2d\zeta_{s,\omega}^{1}}
        \zeta_s^s\exp(-cN_d^\circ)
        \longrightarrow 0
    \end{equation}
    by \cref{asm:AS_Sampling}, the boundedness of \(\zeta_{s}^{s}\), \((n-d)/n \leq 1\), and the exponential domination used in \cref{lem:incompl_zero_count_centering_transfer}.

    Since
    \begin{equation}
        \frac{n\hat{\sigma}_{JKD}^{2}\left(D_{[n]}; d, \omega\right)}
        {s^{2}\zeta_{s,\omega}^{1}}
        =
        \sum_{\ell \in L_{n,d}}
        \left(
            \overline{X}_{\ell}
            +
            \overline{Y}_{\ell}
            +
            \overline{Z}_{\ell}
        \right)^{2},
    \end{equation}
    while \(\sum \Ex*{\overline{X}_{\ell}^{2}} \to 1\), \(\sum \overline{X}_{\ell}^{2} \pto 1\), and
    \begin{equation}
        \sum_{\ell \in L_{n,d}}
        \Ex*{
            \left(
                \overline{Y}_{\ell}
                +
                \overline{Z}_{\ell}
            \right)^{2}
        }
        \leq
        2
        \sum_{\ell \in L_{n,d}}
        \Ex*{\overline{Y}_{\ell}^{2}}
        +
        2
        \sum_{\ell \in L_{n,d}}
        \Ex*{\overline{Z}_{\ell}^{2}}
        \longrightarrow 0,
    \end{equation}
    \cref{lem:peng1} yields
    \begin{equation}\label{eq:scaled_incomplete_jkd}
        \frac{n\hat{\sigma}_{JKD}^{2}\left(D_{[n]}; d, \omega\right)}
        {s^{2}\zeta_{s,\omega}^{1}}
        \pto 1.
    \end{equation}
    Thus realized counts add only a second-order perturbation to the Horvitz-Thompson jackknife argument.

    Finally, \cref{eq:full_norm_error_bound} implies
    \begin{equation}
        \Varb*{
            U_{n,s,N,\omega}\left(D_{[n]}\right)
            -
            \overline{U}_{n,s,\rho,\omega}\left(D_{[n]}\right)
        }
        \leq
        \frac{\zeta_{s}^{s}}{N}
        =
        o\left(
            \frac{s^{2}}{n}\zeta_{s,\omega}^{1}
        \right),
    \end{equation}
    again by \cref{asm:AS_Sampling} and bounded \(\zeta_{s}^{s}\).
    Hence
    \begin{equation}
        \abs*{
            \sigma_{n}
            -
            \overline{\sigma}_{n}
        }
        \leq
        \sqrt{
            \Varb*{
                U_{n,s,N,\omega}\left(D_{[n]}\right)
                -
                \overline{U}_{n,s,\rho,\omega}\left(D_{[n]}\right)
            }
        }
        =
        o\left(
            \sqrt{
                \frac{s^{2}}{n}\zeta_{s,\omega}^{1}
            }
        \right),
    \end{equation}
    so \cref{eq:var_scale_u_bar} implies
    \begin{equation}
        \frac{n\sigma_{n}^{2}}{s^{2}\zeta_{s,\omega}^{1}}
        \longrightarrow 1.
    \end{equation}
    The same comparison that transfers the jackknife estimator from \(\overline U_{n,s,\rho,\omega}\) to \(U_{n,s,N,\omega}\) also transfers the variance target itself.
    So both the estimator and the target are asymptotically governed by the same first-projection scale.
    Combining this with \cref{eq:scaled_incomplete_jkd} gives, for centered kernels,
    \begin{equation}
        \frac{\hat{\sigma}_{JKD}^{2}\left(D_{[n]}; d, \omega\right)}{\sigma_{n}^{2}}
        \pto 1.
    \end{equation}
    For a general kernel with bounded \(\theta_s\), the lower Hoeffding projections and \(\zeta_{s,\omega}^1\) are unchanged after replacing \(h_s\) by \(h_s-\theta_s\).
    Moreover, \cref{asm:AS_Sampling} and the boundedness of \(\zeta_s^s\) imply \(N/n\to\infty\) and \(1/\zeta_{s,\omega}^1=o(Ns/n)\).
    Since \(N_d^\circ/N\to1\) by \(sd=o(n)\), \cref{lem:incompl_zero_count_centering_transfer} transfers both the jackknife estimator and the variance target from the centered kernel back to the original kernel.
    This proves the stated ratio consistency under the zero-count convention.
\end{proof}

    \section{DNN Kernel Tools and Single-Scale H\'ajek Dominance}\label{sec:dnn_hajek_dominance}
    \subsection{Notation}

We first isolate the single-scale DNN notation that will be used throughout \cref{sec:dnn_hajek_dominance}.

For the DNN kernel, define the first-order projection objects by
\begin{equation}
    \psi_{s}^{1}(x; d_{1})
    =
    \Ex*{h_{s}\left(x; D_{[s]}\right) \given Z_{1} = d_{1}}
\end{equation}
Let
\begin{equation}
    \theta_{s}(x) \coloneq \Ex*{h_{s}\left(x; D_{[s]}\right)}
\end{equation}
denote the finite-sample DNN mean.
Then
\begin{equation}
    h_{s}^{(1)}\left(x; d_{1}\right)
    =
    \psi_{s}^{1}(x; d_{1}) - \theta_{s}(x).
\end{equation}
Since $\theta_{s}(x)$ and $\mu(x)$ are both constants in $d_{1}$,
\[
    \Varb*{h_{s}^{(1)}(x; Z_{1})}
    =
    \Varb*{\psi_{s}^{1}(x; Z_{1}) - c}
\]
is the same under either centering.
The two conventions differ only by $\theta_{s}(x) - \mu(x) \to 0$ as $s, n \to \infty$; see \cref{lem:dem13}.
For higher-order projection kernels, write for $c = 2, \dotsc, s$,
\begin{equation}
    \psi_{s}^{c}(x; d_{[c]})
    =
    \Ex*{h_{s}\left(x; D_{[s]}\right) \given D_{[c]} = d_{[c]}},
\end{equation}
and recursively define
\begin{equation}
    h_{s}^{(c)}\left(x; d_{[c]}\right)
    =
    \psi_{s}^{c}(x; d_{[c]}) - \sum_{j = 1}^{c-1}\left(\sum_{\ell \in L_{c,j}}h_{s}^{(j)}(x; d_{\ell})\right) - \theta_{s}(x).
\end{equation}
The corresponding Hoeffding decomposition is
\begin{equation}
    \tilde{\mu}_{s}\left(x; D_{[n]}\right)
    =
    \theta_{s}(x) + \sum_{j = 1}^{s}\binom{s}{j}\binom{n}{j}^{-1} \sum_{\ell \in L_{n,j}} h^{(j)}_{s}(x; D_{\ell}).
\end{equation}
We also write, for any $1 \leq c \leq s$,
\begin{align}
    \zeta_{s}^{1}\left(x\right)
     & = \Varb*{h_{s}^{(1)}\left(x; Z_{1}\right)}                         \\
    \zeta_{s}^{s}\left(x\right)
     & = \Varb*{h_{s}\left(x; D_{[s]}\right)}                              \\
    \xi_{s}^{c}\left(x\right)
     & = \Varb*{\psi_{s}^{c}(x; D_{[c]})}                                  \\
    \Omega_{s}\left(x\right)
     & = \Ex*{h_{s}^{2}\left(x; D_{[s]}\right)}                            \\
    \Omega_{s}^{c}\left(x\right)
     & = \E\left[h_{s}\left(x; D_{[s]}\right) \cdot
        h_{s}\left(x; D_{[s]}^{\prime}\right)\right]
\end{align}
where $D_{[s]} = \{Z_1, \dotsc, Z_{s}\}$ is a vector of i.i.d.\ random variables drawn from $P$ and $D_{[s]}^{\prime} = \{Z_1, \dotsc, Z_{c}, Z_{c+1}^{\prime}, \dotsc,  Z_{s}^{\prime}\}$ where $Z_{c+1}^{\prime}, \dotsc,  Z_{s}^{\prime}$ are i.i.d.\ draws from $P$ that are independent of $D_{[s]}$.

    \subsection{DNN Kernel Lemmas}
Throughout the proofs in this section, we are concerned with the same setup.
For the sake of brevity, we introduce this setup here to avoid unnecessary repetition in the following lemmas.
Consider a sample size $n$, a subsampling scale $s$ growing with $n$, and $c$ such that $0 < c \leq s \leq n$.
Let $D_{[s]} = \left\{Z_1, Z_2, \dotsc, Z_c, Z_{c+1}, \dotsc Z_s \right\}$ be an i.i.d.\ data set drawn from $P$ as described in \cref{asm:npr_dgp}.
Let $D^{\prime}_{[s]}= \left\{Z_1, Z_2, \dotsc, Z_c, Z_{c+1}^{\prime}, \dotsc Z_s^{\prime} \right\}$ be a second data set that shares the first $c$ observations with $D_{[s]}$.
The remaining $s - c$ observations of $D^{\prime}_{[s]}$, i.e.\ $\left\{Z_{c+1}^{\prime}, \dotsc Z_s^{\prime} \right\}$, are i.i.d.\ draws from $P$ that are independent of $D_{[s]}$.
All expectations are with respect to all random elements unless a conditioning bar is displayed; we write $\Ex*{\cdot \given X}$ for conditional expectation given the sigma-field generated by the displayed variables.

\begin{lemma}[\cite{demirkaya_optimal_2024} - Lemma 12]\label{lem:dem12}\mbox{}\\*
    The indicator functions $\kappa\left(x; Z_{i}, D_{[s]}\right)$ satisfy the following properties.
    \begin{enumerate}
        \item For any $i \neq j$, we have $\kappa\left(x; Z_{i}, D_{[s]}\right) \kappa\left(x;
                  Z_{j}, D_{[s]}\right)=0$ with probability one;
        \item $\sum_{i=1}^{s} \kappa\left(x; Z_{i}, D_{[s]}\right)=1$;
        \item $\forall i \in [s]: \quad \Ex*{\kappa\left(x; Z_{i}, D_{[s]}\right)}=s^{-1}$
        \item $\Ex*{\kappa\left(x; Z_1, D_{[s]}\right) \given D_{1} = Z_{1}}
                  = \left\{1-\varphi\left(B\left(x,\norm*{X_1-x}\right)\right)\right\}^{s-1}$
    \end{enumerate}
\end{lemma}

\begin{lemma}[\cite{demirkaya_optimal_2024} - Lemma 13]\label{lem:dem13}\mbox{}\\*
    For any $L^1$ function $f$ that is continuous at $x$, it holds that
    \begin{equation}
        \lim _{s \longrightarrow \infty} \Ex*{f\left(X_1\right) s \Ex*{\kappa\left(x; Z_1, D_{[s]}\right) \given X_{1}}}
        = f(x).
    \end{equation}
\end{lemma}

We also need product analogues of \cref{lem:dem13} for the covariance calculations below.
\begin{lemma}\label{lem:expec_kernel_prod}\mbox{}\\*
    The following three statements hold.
    \begin{equation}
        \forall i \in [c]: \quad
        \Ex*{\kappa\left(x; Z_{i}, D_{[s]}\right)\kappa\left(x; Z_{i}, D^{\prime}_{[s]}\right)}
        = (2s - c)^{-1} = \omega(e^{-s})
    \end{equation}
    \begin{equation}
        \begin{aligned}
            & \forall i \in [c] \; \forall j \in \{c+1, \dotsc, s\}: \quad
             \Ex*{\kappa\left(x; Z_{i}, D_{[s]}\right)\kappa\left(x; Z_{j}^{\prime}, D^{\prime}_{[s]}\right)} \\
             & \qquad = \frac{1}{s(2s-c)}
             = \omega(e^{-s})
        \end{aligned}
    \end{equation}
    \begin{equation}
        \begin{aligned}
            & \forall i,j \in \{c+1, \dotsc, s\}: \quad
             \Ex*{\kappa\left(x; Z_{i}, D_{[s]}\right)\kappa\left(x; Z_{j}^{\prime}, D^{\prime}_{[s]}\right)} \\
             & \qquad = \frac{2}{s(2s-c)}
             = \omega(e^{-s})
        \end{aligned}
    \end{equation}
\end{lemma}

\begin{proof}[Proof of \cref{lem:expec_kernel_prod}]\mbox{}\\*
    By symmetry, it suffices to consider $i=1$ and $j=c+1$ for the first two equations.
    \begin{equation}
        \begin{aligned}
            & \Ex*{\kappa\left(x; Z_{1}, D_{[s]}\right)\kappa\left(x; Z_{1}, D^{\prime}_{[s]}\right)}\\
             & \quad = \Ex*{\kappa\left(x; Z_{1}, D_{[c]}\right)\kappa\left(x; Z_{1}, D_{(c+1):s}\right)\kappa\left(x; Z_{1}, D^{\prime}_{(c+1):s}\right)} \\
             & \quad = \Ex*{\kappa\left(x; Z_{1}, D_{[2s - c]}\right)}
            = (2s - c)^{-1}
        \end{aligned}
    \end{equation}
    For the remaining two cases we use the reciprocal-binomial identity
    \begin{equation}\label{eq:reciprocal_binomial_identity}
        \sum_{r = 0}^{a}\binom{a}{r}\binom{N}{r}^{-1}
        =
        \frac{N+1}{N+1-a},
        \qquad 0 \leq a \leq N.
    \end{equation}
    This follows by writing
    $\binom{a}{r}\binom{N}{r}^{-1} = \binom{N-r}{a-r}\binom{N}{a}^{-1}$
    and applying the hockey-stick identity.
    Considering the second case, we find the following.
    \begin{equation}
        \begin{aligned}
             & \Ex*{\kappa\left(x; Z_{1}, D_{[s]}\right)\kappa\left(x; Z_{c+1}^{\prime}, D^{\prime}_{[s]}\right)}   \\
             & = \frac{1}{(2s - c)!} \sum_{r = 0}^{s - c - 1}\binom{s - c - 1}{r} r! \left((s - 1) + (s - c - 1 - r)\right)! \\
             & = \frac{1}{(2s - c)!}\sum_{r = 0}^{s - c - 1}\binom{s - c - 1}{r} r! \left(2s - c - 2 - r\right)!             \\
             & = \frac{1}{(2s - c)(2s - c - 1)}
             \sum_{r = 0}^{s - c - 1}\binom{s - c - 1}{r}\binom{2s - c - 2}{r}^{-1} \\
             & \overset{\text{\cref{eq:reciprocal_binomial_identity}}}{=} \frac{1}{s(2s-c)}.
        \end{aligned}
    \end{equation}
    While unintuitive at first, the terms in this expression have intuitive meaning when we consider this as a combinatorial problem.
    Consider lining up the observations in order of their distance to the point of interest and counting the cases for which the expression in the expectation is equal to one.
    First, there are $(2s-c)!$ possible orderings of the observations with probability one, leading to the denominator.
    Next, notice that only those orderings where $\norm*{X_{c+1}^{\prime} - x} \leq \norm*{X_{1}-x}$ and $\norm*{X_{1} - x} \leq \norm*{X_{i} - x}$ for any $i = 2, \dotsc, c$ can possibly lead to a non-zero realization of the kernel term.
    Furthermore, out of the $(s-c-1)$ observations in $D^{\prime}_{(c+2):s}$, it is possible for $i = 0, \dotsc, s-c-1$ observations to lie at a distance to the point of interest that is smaller than $\norm*{X_{1}-x}$ but larger than $\norm*{X_{c+1}^{\prime} - x}$ in any permutation.
    The sum adjusts for those possible configurations.
    Because $c \leq s-1$, the exact value is bounded from below by $(2s^2)^{-1}$ and is therefore $\omega(e^{-s})$.

    Considering the third case, without loss of generality, we consider the case of $i = j = c+1$.
    We find the following.
    \begin{equation}
        \begin{aligned}
             & \Ex*{\kappa\left(x; Z_{c+1}, D_{[s]}\right)\kappa\left(x; Z_{c+1}^{\prime}, D^{\prime}_{[s]}\right)} \\
             & = \E\left[
                \kappa\left(x; Z_{c+1}, D_{[c]}\right)\kappa\left(x; Z_{c+1}^{\prime}, D_{[c]}\right)
                \kappa\left(x; Z_{c+1}, D_{(c+1):s}\right)\kappa\left(x; Z_{c+1}^{\prime}, D_{(c+1):s}^{\prime}\right)
            \right]                                                                                                          \\
             & = \frac{2}{(2s - c)!} \sum_{r = 0}^{s-c-1} \binom{s - c - 1}{r}(s - 1 + r)!(s-c-1-r)!                         \\
             & = \frac{2(2s - c - 2)!}{(2s-c)!} \sum_{r = 0}^{s-c-1} \binom{s - c - 1}{r}\binom{2s - c - 2}{s-1+r}^{-1}      \\
             & = \frac{2}{(2s-c)(2s-c-1)}\sum_{r = 0}^{s-c-1} \binom{s - c - 1}{r}\binom{2s - c - 2}{s-1+r}^{-1} \\
             & = \frac{2}{(2s-c)(2s-c-1)}\sum_{r = 0}^{s-c-1} \binom{s - c - 1}{r}\binom{2s - c - 2}{s-c-1-r}^{-1} \\
             & \overset{\text{\cref{eq:reciprocal_binomial_identity}}}{=} \frac{2}{s(2s-c)}.
        \end{aligned}
    \end{equation}
    The third case follows from a similar combinatorial logic as the second.
    By symmetry, take $\norm*{X_{c+1}^{\prime} - x} \leq \norm*{X_{c+1}-x}$ and multiply by two.
    Any number $i = 0, \dotsc, s-c-1$ of observations in $D^{\prime}_{(c+2):s}$ can be farther from $x$ than $X_{c+1}$ or lie between $\norm*{X_{c+1}^{\prime} - x}$ and $\norm*{X_{c+1}-x}$.
    The summation counts these configurations.
    The exact value is again bounded from below by $(s^2)^{-1}$, and hence it is $\omega(e^{-s})$.
\end{proof}

\newpage
\begin{landscape}
    \begin{lemma}\label{lem:cond_expec_kernel_prod}\mbox{}\\*
        The following statements hold.
        \begin{equation}
            \begin{aligned}
                 & \forall i \in [c]: \quad
                \Ex*{\kappa\left(x; Z_{i}, D_{[s]}\right)\kappa\left(x; Z_{i}, D^{\prime}_{[s]}\right) \given X_i}
                = \left\{1-\varphi\left(B\left(x,\norm*{X_{i}-x}\right)\right)\right\}^{2s-c-1}
            \end{aligned}
        \end{equation}
        \begin{equation}
            \begin{aligned}
                 & \forall i \in [c] \; \forall j \in \{c+1, \dotsc, s\}: \\
                 & \Ex*{\kappa\left(x; Z_{i}, D_{[s]}\right)\kappa\left(x; Z_{j}^{\prime}, D^{\prime}_{[s]}\right) \given X_i, X_{j}^{\prime}} \\
                 & \qquad = \1*{\norm*{X_{j}^{\prime} - x} \leq \norm*{X_{i} - x}}
                \left\{1-\varphi\left(B\left(x,\norm*{X_{i}-x}\right)\right)\right\}^{s-1}
                \left\{1-\varphi\left(B\left(x,\norm*{X_{j}^{\prime}-x}\right)\right)\right\}^{s-c-1}
            \end{aligned}
        \end{equation}
        \begin{equation}
            \begin{aligned}
                 & \forall i,j \in \{c+1, \dotsc, s\}: \\
                 & \Ex*{\kappa\left(x; Z_{i}, D_{[s]}\right)\kappa\left(x; Z_{j}^{\prime}, D^{\prime}_{[s]}\right) \given X_{i}, X_{j}^{\prime}} \\
                 & \qquad =
                \left\{1-\varphi\left(B\left(x, \min\left(\norm*{X_{i} - x}, \norm*{X_{j}^{\prime}-x}\right)\right)\right)\right\}^{s-c-1} \\
                 & \qquad\quad \times
                \left\{1-\varphi\left(B\left(x,\max\left(\norm*{X_{i} - x}, \norm*{X_{j}^{\prime}-x}\right)\right)\right)\right\}^{s-1}
            \end{aligned}
        \end{equation}
    \end{lemma}

    \begin{proof}[Proof of \cref{lem:cond_expec_kernel_prod}]\mbox{}\\*
        By symmetry, it suffices to take $i=1$ for the first equation.
        \begin{equation}
            \begin{aligned}
                 & \Ex*{\kappa\left(x; Z_{1}, D_{[s]}\right)\kappa\left(x; Z_{1}, D^{\prime}_{[s]}\right) \given X_1} \\
                 & = \E\left[\kappa\left(x; Z_{1}, D_{[c]}\right)
                    \kappa\left(x; Z_{1}, D_{(c+1):s}\right)
                    \kappa\left(x; Z_{1}, D^{\prime}_{(c+1):s}\right)
                \middle| X_1\right]                                                                                              \\
                 & = \Ex*{\kappa\left(x; Z_{1}, D_{[c]}\right)\given X_1}
                \Ex*{\kappa\left(x; Z_{1}, D_{(c+1):s}\right)\given X_1}
                \Ex*{\kappa\left(x; Z_{1}, D_{(c+1):s}^{\prime}\right)\given X_1}                                     \\
                 & = \left\{1-\varphi\left(B\left(x,\norm*{X_{i}-x}\right)\right)\right\}^{c-1}
                \left\{1-\varphi\left(B\left(x,\norm*{X_{i}-x}\right)\right)\right\}^{s-c}
                \left\{1-\varphi\left(B\left(x,\norm*{X_{i}-x}\right)\right)\right\}^{s-c}                                \\
                 & = \left\{1-\varphi\left(B\left(x,\norm*{X_{i}-x}\right)\right)\right\}^{2s-c-1}
            \end{aligned}
        \end{equation}

        By symmetry, it suffices to take $i=1$ and $j=c+1$ for the second equation.
        \begin{equation}
            \begin{aligned}
                 & \Ex*{\kappa\left(x; Z_{1}, D_{[s]}\right)\kappa\left(x; Z_{c+1}^{\prime}, D^{\prime}_{[s]}\right) \given X_1, X_{c+1}^{\prime}} \\
                 & = \E\Bigg[
                    \E\Bigg[
                        \kappa\left(x; Z_{1}, D_{[c]}\right)
                        \kappa\left(x; Z_{1}, D_{(c+1):s}\right)
                        \kappa\left(x; Z_{c+1}^{\prime}, D_{[c]}\right)
                        \kappa\left(x; Z_{c+1}^{\prime}, D^{\prime}_{(c+1):s}\right)
                        \Biggm| X_{[c]}, X_{c+1}^{\prime}\Bigg]
                    \Biggm| X_1, X_{c+1}^{\prime}
                \Bigg]                                                                                                                                        \\
                 & = \E\left[
                    \E\left[
                        \kappa\left(x; Z_{1}, D_{(c+1):s}\right)
                        \kappa\left(x; Z_{c+1}^{\prime}, D^{\prime}_{(c+1):s}\right)
                        \middle| X_{[c]}, X_{c+1}^{\prime}\right]
                    \kappa\left(x; Z_{1}, D_{[c]}\right)
                    \kappa\left(x; Z_{c+1}^{\prime}, D_{[c]}\right)
                \middle| X_1, X_{c+1}^{\prime}\right]                                                                                                         \\
                 & = \E\left[
                    \E\left[
                        \kappa\left(x; Z_{1}, D_{(c+1):s}\right)
                        \kappa\left(x; Z_{c+1}^{\prime}, D^{\prime}_{(c+1):s}\right)
                        \middle| X_1, X_{c+1}^{\prime}\right]
                    \kappa\left(x; Z_{1}, D_{[c]}\right)
                    \kappa\left(x; Z_{c+1}^{\prime}, D_{[c]}\right)
                \middle| X_1, X_{c+1}^{\prime}\right]                                                                                                         \\
                 & = \E\left[
                    \kappa\left(x; Z_{1}, D_{(c+1):s}\right)
                    \kappa\left(x; Z_{c+1}^{\prime}, D^{\prime}_{(c+1):s}\right)
                    \middle| X_1, X_{c+1}^{\prime}\right]
                \E\left[
                    \kappa\left(x; Z_{1}, D_{[c]}\right)
                    \kappa\left(x; Z_{c+1}^{\prime}, D_{[c]}\right)
                \middle| X_1, X_{c+1}^{\prime}\right]                                                                                                         \\
                 & = \Ex*{\kappa\left(x; Z_{1}, D_{(c+1):s}\right)\given X_1}
                \E\left[\kappa\left(x; Z_{c+1}^{\prime}, D^{\prime}_{(c+1):s}\right)
                    \middle| X_{c+1}^{\prime}\right]
                \1*{\norm*{X_{c+1}^{\prime} - x} \leq \norm*{X_{1} - x}}
                \E\left[
                    \kappa\left(x; Z_{1}, D_{[c]}\right)
                \middle| X_1\right]                                                                                                                           \\
                 & = \1*{\norm*{X_{c+1}^{\prime} - x} \leq \norm*{X_{1} - x}}
                \Ex*{\kappa\left(x; Z_{1}, D_{[s]}\right)\given X_1}
                \E\left[\kappa\left(x; Z_{c+1}^{\prime}, D^{\prime}_{(c+1):s}\right)
                \middle| X_{c+1}^{\prime}\right]                                                                                                              \\
                 & = \1*{\norm*{X_{c+1}^{\prime} - x} \leq \norm*{X_{1} - x}}
                \left\{1-\varphi\left(B\left(x,\norm*{X_1-x}\right)\right)\right\}^{s-1}
                \left\{1-\varphi\left(B\left(x,\norm*{X_{c+1}^{\prime}-x}\right)\right)\right\}^{s-c-1}
            \end{aligned}
        \end{equation}

        For the third case, without loss of generality, we consider the case of $i = j = c+1$.
        \begin{equation}
            \begin{aligned}
                 & \E\left[
                    \kappa\left(x; Z_{c+1}, D_{[s]}\right)
                \kappa\left(x; Z_{c+1}^{\prime}, D^{\prime}_{[s]}\right) \middle| X_{c+1}, X_{c+1}^{\prime}\right] \\
                 & = \E\left[
                    \E\left[
                        \kappa\left(x; Z_{c+1}, D_{[s]}\right)
                        \kappa\left(x; Z_{c+1}^{\prime}, D^{\prime}_{[s]}\right)
                        \middle| X_{[c]}, X_{c+1}, X_{c+1}^{\prime}\right]
                \middle| X_{c+1}, X_{c+1}^{\prime} \right]                                                         \\
                 & = \E\left[
                    \E\left[
                        \kappa\left(x; Z_{c+1}, D_{[c+1]}\right)
                        \kappa\left(x; Z_{c+1}^{\prime}, D^{\prime}_{[c+1]}\right)
                        \kappa\left(x; Z_{c+1}, D_{(c+1):s}\right)
                        \kappa\left(x; Z_{c+1}^{\prime}, D^{\prime}_{(c+1):s}\right)
                        \middle| X_{[c]}, X_{c+1}, X_{c+1}^{\prime}\right]
                \middle| X_{c+1}, X_{c+1}^{\prime} \right]                                                         \\
                 & = \E\left[
                    \E\left[
                        \kappa\left(x; Z_{c+1}, D_{[c+1]}\right)
                        \kappa\left(x; Z_{c+1}^{\prime}, D^{\prime}_{[c+1]}\right) \middle| X_{[c]}, X_{c+1}, X_{c+1}^{\prime}\right]
                    \kappa\left(x; Z_{c+1}, D_{(c+1):s}\right)
                    \kappa\left(x; Z_{c+1}^{\prime}, D^{\prime}_{(c+1):s}\right)
                \middle| X_{c+1}, X_{c+1}^{\prime} \right]                                                         \\
                 & = \E\left[
                    \kappa\left(x; Z_{c+1}, D_{[c+1]}\right)
                    \kappa\left(x; Z_{c+1}^{\prime}, D^{\prime}_{[c+1]}\right)
                    \middle| X_{c+1}, X_{c+1}^{\prime}\right]
                \Ex*{\kappa\left(x; Z_{c+1}, D_{(c+1):s}\right)\given X_{c+1}}
                \Ex*{\kappa\left(x; Z_{c+1}^{\prime}, D^{\prime}_{(c+1):s}\right)\given X_{c+1}^{\prime}}
            \end{aligned}
        \end{equation}
        \newpage
        Without loss of generality, consider the case that $\norm*{X_{c+1} - x} \leq \norm*{X_{c+1}^{\prime} - x}$.
        \begin{equation}
            \E\left[
                \kappa\left(x; Z_{c+1}, D_{[c+1]}\right)
                \kappa\left(x; Z_{c+1}^{\prime}, D^{\prime}_{[c+1]}\right)
                \middle| X_{c+1}, X_{c+1}^{\prime}\right]
            = \E\left[
                \kappa\left(x; Z_{c+1}^{\prime}, D^{\prime}_{[c+1]}\right)
                \middle| X_{c+1}^{\prime}\right]
        \end{equation}
        Also,
        \begin{equation}
            \Ex*{\kappa\left(x; Z_{c+1}^{\prime}, D^{\prime}_{[c+1]}\right) \given X_{c+1}^{\prime}}
            \Ex*{\kappa\left(x; Z_{c+1}^{\prime}, D^{\prime}_{(c+1):s}\right)\given X_{c+1}^{\prime}}
            = \Ex*{\kappa\left(x; Z_{c+1}^{\prime}, D^{\prime}_{[s]}\right)\given X_{c+1}^{\prime}}
        \end{equation}
        Hence
        \begin{equation}
            \begin{aligned}
                 & \Ex*{\kappa\left(x; Z_{c+1}, D_{[s]}\right)\kappa\left(x; Z_{c+1}^{\prime}, D^{\prime}_{[s]}\right) \given X_{c+1}, X_{c+1}^{\prime}} \\
                 & = \1*{\norm*{X_{c+1}^{\prime} - x} \leq \norm*{X_{c+1} - x}}
                \left\{1-\varphi\left(B\left(x,\norm*{X_{c+1} - x}\right)\right)\right\}^{s-1}
                \left\{1-\varphi\left(B\left(x,\norm*{X_{c+1}^{\prime}-x}\right)\right)\right\}^{s-c-1}                                                    \\
                 & \qquad + \1*{\norm*{X_{c+1}^{\prime} - x} > \norm*{X_{c+1} - x}}
                \left\{1-\varphi\left(B\left(x,\norm*{X_{c+1} - x}\right)\right)\right\}^{s-c-1}
                \left\{1-\varphi\left(B\left(x,\norm*{X_{c+1}^{\prime}-x}\right)\right)\right\}^{s-1}                                                      \\
                 & = \left\{1-\varphi\left(B\left(x, \min\left(\norm*{X_{c+1} - x}, \norm*{X_{c+1}^{\prime}-x}\right)\right)\right)\right\}^{s-c-1}
                \left\{1-\varphi\left(B\left(x,\max\left(\norm*{X_{c+1} - x}, \norm*{X_{c+1}^{\prime}-x}\right)\right)\right)\right\}^{s-1}
            \end{aligned}
        \end{equation}
    \end{proof}
\end{landscape}

\begin{lemma}\label{lem:kernel_prod_dirac_convergence}\mbox{}\\*
    For any $L^{2}(\calX)$ function $f$ that is continuous at $x$, it holds that
    \begin{equation}
        \lim_{s \longrightarrow \infty} \underbrace{\E\left[f^{2}(X_{1}) (2s - c)
                \Ex*{\kappa\left(x; Z_{1}, D_{[s]}\right)\kappa\left(x; Z_{1}, D^{\prime}_{[s]}\right) \given X_{1}}
                \right]}_{(A)}
        = f^{2}(x)
    \end{equation}
    \begin{equation}
        \lim_{s \longrightarrow \infty} \underbrace{\E\left[
                f(X_{1}) f(X_{c+1}^{\prime})
                \frac{\Ex*{\kappa\left(x; Z_{1}, D_{[s]}\right)\kappa\left(x; Z_{c+1}^{\prime}, D^{\prime}_{[s]}\right) \given X_{1}, X_{c+1}^{\prime}}}{\Ex*{\kappa\left(x; Z_{1}, D_{[s]}\right)\kappa\left(x; Z_{c+1}^{\prime}, D^{\prime}_{[s]}\right)}}
                \right]}_{(B)}
        = f^{2}(x)
    \end{equation}
    \begin{equation}
        \lim_{s \longrightarrow \infty} \underbrace{\E\left[
                f(X_{c+1}) f(X_{c+1}^{\prime})
                \frac{\Ex*{\kappa\left(x; Z_{c+1}, D_{[s]}\right)\kappa\left(x; Z_{c+1}^{\prime}, D^{\prime}_{[s]}\right) \given X_{c+1}, X_{c+1}^{\prime}}}{\Ex*{\kappa\left(x; Z_{c+1}, D_{[s]}\right)\kappa\left(x; Z_{c+1}^{\prime}, D^{\prime}_{[s]}\right)}}
                \right]}_{(C)}
        = f^{2}(x)
    \end{equation}
\end{lemma}

\begin{proof}[Proof of \cref{lem:kernel_prod_dirac_convergence}]\mbox{}\\*
    We will largely argue along the same lines as the original proof in~\cite{demirkaya_optimal_2024}.
    Thus, consider first the following inequalities.
    \begin{equation}
        \begin{aligned}
             & \abs*{(A) - f^{2}(x)}
            = \left|\E\left[f^{2}(X_{1}) (2s - c)
                \Ex*{\kappa\left(x; Z_{1}, D_{[s]}\right)\kappa\left(x; Z_{1}, D^{\prime}_{[s]}\right) \given X_{1}}
            \right] - f^{2}(x)\right|                                      \\
             & \leq \E\left[\abs*{f^{2}(X_{1}) - f^{2}(x)} (2s - c)
            \Ex*{\kappa\left(x; Z_{1}, D_{[s]}\right)\kappa\left(x; Z_{1}, D^{\prime}_{[s]}\right) \given X_{1}}
            \right]
        \end{aligned}
    \end{equation}
    \begin{equation}
        \begin{aligned}
             & \abs*{(B) - f^{2}(x)} \\
             & = \left|\E\left[
                f(X_{1}) f(X_{c+1}^{\prime})
                \frac{\Ex*{\kappa\left(x; Z_{1}, D_{[s]}\right)\kappa\left(x; Z_{c+1}^{\prime}, D^{\prime}_{[s]}\right) \given X_{1}, X_{c+1}^{\prime}}}{\Ex*{\kappa\left(x; Z_{1}, D_{[s]}\right)\kappa\left(x; Z_{c+1}^{\prime}, D^{\prime}_{[s]}\right)}}
            \right] - f^{2}(x)\right|      \\
             & \leq \E\left[
                \abs*{f(X_{1}) f(X_{c+1}^{\prime}) - f^{2}(x)}
                \frac{\Ex*{\kappa\left(x; Z_{1}, D_{[s]}\right)\kappa\left(x; Z_{c+1}^{\prime}, D^{\prime}_{[s]}\right) \given X_{1}, X_{c+1}^{\prime}}}{\Ex*{\kappa\left(x; Z_{1}, D_{[s]}\right)\kappa\left(x; Z_{c+1}^{\prime}, D^{\prime}_{[s]}\right)}}
                \right]
        \end{aligned}
    \end{equation}
    \begin{equation}
        \begin{aligned}
             & \abs*{(C) - f^{2}(x)} \\
             & = \left|\E\left[
                f(X_{c+1}) f(X_{c+1}^{\prime})
                \frac{\Ex*{\kappa\left(x; Z_{c+1}, D_{[s]}\right)\kappa\left(x; Z_{c+1}^{\prime}, D^{\prime}_{[s]}\right) \given X_{c+1}, X_{c+1}^{\prime}}}{\Ex*{\kappa\left(x; Z_{c+1}, D_{[s]}\right)\kappa\left(x; Z_{c+1}^{\prime}, D^{\prime}_{[s]}\right)}}
            \right] - f^{2}(x)\right|      \\
             & \leq \E\left[
                \abs*{f(X_{c+1}) f(X_{c+1}^{\prime}) - f^{2}(x)}
                \frac{\Ex*{\kappa\left(x; Z_{c+1}, D_{[s]}\right)\kappa\left(x; Z_{c+1}^{\prime}, D^{\prime}_{[s]}\right) \given X_{c+1}, X_{c+1}^{\prime}}}{\Ex*{\kappa\left(x; Z_{c+1}, D_{[s]}\right)\kappa\left(x; Z_{c+1}^{\prime}, D^{\prime}_{[s]}\right)}}
                \right]
        \end{aligned}
    \end{equation}
    Now, fix an arbitrary $\epsilon > 0$.
    By continuity of $f$ at $x$, there exists a $\delta > 0$, such that the following holds.
    \begin{equation}
        \forall X, X^{\prime} \in B(x, \delta): \quad
        \abs*{f(X)  f(X^{\prime}) - f^{2}(x)} < \epsilon
    \end{equation}
    We can consider decompositions of these terms in analogy to~\cite{demirkaya_optimal_2024}, i.e., by considering cases with observations lying within this sphere or outside of it, and observe the following.
    \begin{equation}
        \begin{aligned}
             & \E\Big[\abs*{f^{2}(X_{1}) - f^{2}(x)} (2s - c)                                                  \\
             & \quad \times \E\left[\kappa\left(x; Z_{1}, D_{[s]}\right)\kappa\left(x; Z_{1}, D^{\prime}_{[s]}\right)
                \1*{X_1 \in B(x, \delta)}  \middle| X_{1}
                \right]
            \Big]                                                                                                     \\
             & \leq \epsilon  \E\left[(2s - c)
                \E\left[\kappa\left(x; Z_{1}, D_{[s]}\right)\kappa\left(x; Z_{1}, D^{\prime}_{[s]}\right)
                    \1*{X_1 \in B(x, \delta)}  \middle| X_{1}
                    \right]
            \right]                                                                                                   \\
             & \leq \epsilon  \E\left[(2s - c)
                \E\left[\kappa\left(x; Z_{1}, D_{[s]}\right)\kappa\left(x; Z_{1}, D^{\prime}_{[s]}\right) \middle| X_{1}
                    \right]
                \right]
            = \epsilon
        \end{aligned}
    \end{equation}
    \begin{equation}
        \begin{aligned}
             & \E\Bigg[
            \abs*{f(X_{1}) f(X_{c+1}^{\prime}) - f^{2}(x)} \1*{X_1, X_{c+1}^{\prime} \in B(x, \delta)}                                                                                                                                                                              \\
             & \qquad \times \left. \frac{\Ex*{\kappa\left(x; Z_{1}, D_{[s]}\right)\kappa\left(x; Z_{c+1}^{\prime}, D^{\prime}_{[s]}\right) \given X_{1}, X_{c+1}^{\prime}}}{\Ex*{\kappa\left(x; Z_{1}, D_{[s]}\right)\kappa\left(x; Z_{c+1}^{\prime}, D^{\prime}_{[s]}\right)}}
            \right]                                                                                                                                                                                                                                                                                  \\
             & \leq \epsilon  \E\left[
                \frac{\Ex*{\kappa\left(x; Z_{1}, D_{[s]}\right)\kappa\left(x; Z_{c+1}^{\prime}, D^{\prime}_{[s]}\right) \given X_{1}, X_{c+1}^{\prime}}}{\Ex*{\kappa\left(x; Z_{1}, D_{[s]}\right)\kappa\left(x; Z_{c+1}^{\prime}, D^{\prime}_{[s]}\right)}}
                \1*{X_1, X_{c+1}^{\prime} \in B(x, \delta)}
            \right]                                                                                                                                                                                                                                                                                  \\
             & \leq \epsilon  \E\left[
                \frac{\Ex*{\kappa\left(x; Z_{1}, D_{[s]}\right)\kappa\left(x; Z_{c+1}^{\prime}, D^{\prime}_{[s]}\right) \given X_{1}, X_{c+1}^{\prime}}}{\Ex*{\kappa\left(x; Z_{1}, D_{[s]}\right)\kappa\left(x; Z_{c+1}^{\prime}, D^{\prime}_{[s]}\right)}}
                \right]
            = \epsilon
        \end{aligned}
    \end{equation}
    \begin{equation}
        \begin{aligned}
             & \E\Bigg[
                \abs*{f(X_{c+1}) f(X_{c+1}^{\prime}) - f^{2}(x)}
            \frac{\Ex*{\kappa\left(x; Z_{c+1}, D_{[s]}\right)\kappa\left(x; Z_{c+1}^{\prime}, D^{\prime}_{[s]}\right) \given X_{c+1}, X_{c+1}^{\prime}}}{\Ex*{\kappa\left(x; Z_{c+1}, D_{[s]}\right)\kappa\left(x; Z_{c+1}^{\prime}, D^{\prime}_{[s]}\right)}} \\
             & \quad \times \1*{X_{c+1}, X_{c+1}^{\prime} \in B(x, \delta)}
            \Bigg]                                                                                                                                                                                                                                                                 \\
             & \leq \epsilon  \E\left[
                \frac{\Ex*{\kappa\left(x; Z_{c+1}, D_{[s]}\right)\kappa\left(x; Z_{c+1}^{\prime}, D^{\prime}_{[s]}\right) \given X_{c+1}, X_{c+1}^{\prime}}}{\Ex*{\kappa\left(x; Z_{c+1}, D_{[s]}\right)\kappa\left(x; Z_{c+1}^{\prime}, D^{\prime}_{[s]}\right)}}
                \1*{X_{c+1}, X_{c+1}^{\prime} \in B(x, \delta)}
            \right]                                                                                                                                                                                                                                                                \\
             & \leq \epsilon  \E\left[
                \frac{\Ex*{\kappa\left(x; Z_{c+1}, D_{[s]}\right)\kappa\left(x; Z_{c+1}^{\prime}, D^{\prime}_{[s]}\right) \given X_{c+1}, X_{c+1}^{\prime}}}{\Ex*{\kappa\left(x; Z_{c+1}, D_{[s]}\right)\kappa\left(x; Z_{c+1}^{\prime}, D^{\prime}_{[s]}\right)}}
                \right]
            = \epsilon
        \end{aligned}
    \end{equation}
    For the complementary events, use the fact that if $X$ or $X^{\prime}$ does not lie in $B(x,\delta)$, then
    \begin{equation}
        B(x, \delta) \subseteq B(x, \max\left(\norm*{X - x}, \norm*{X^{\prime} - x}\right)).
    \end{equation}
    Therefore,
    \begin{equation}
        \begin{aligned}
             & \E\Bigg[\abs*{f^{2}(X_{1}) - f^{2}(x)} (2s - c)                                                  \\
             & \qquad \times \E\left[\kappa\left(x; Z_{1}, D_{[s]}\right)\kappa\left(x; Z_{1}, D^{\prime}_{[s]}\right)
                \left(1 - \1*{X_1 \in B(x, \delta)}\right) \middle| X_{1}
                \right]
            \Bigg]                                                                                                     \\
             & \leq
            \E\left[\abs*{f^{2}(X_{1}) - f^{2}(x)} (2s - c)
            \left\{1-\varphi\left(B\left(x,\delta\right)\right)\right\}^{2s-c-1}
            \left(1 - \1*{X_1 \in B(x, \delta)}\right)
            \right]                                                                                                    \\
             & \leq (2s - c) 1 \left\{1-\varphi\left(B\left(x,\delta\right)\right)\right\}^{2s-c-1}
            \Ex*{\abs*{f^{2}(X_{1}) - f^{2}(x)}}
        \end{aligned}
    \end{equation}
    In the second case, first recall the form of the conditional expectation from \cref{lem:cond_expec_kernel_prod}.
    \begin{equation}
        \begin{aligned}
             & \Ex*{\kappa\left(x; Z_{1}, D_{[s]}\right)\kappa\left(x; Z_{c+1}^{\prime}, D^{\prime}_{[s]}\right) \given X_1, X_{c+1}^{\prime}} \\
             & = \1*{\norm*{X_{c+1}^{\prime} - x} \leq \norm*{X_{1} - x}}                                                                               \\
             & \qquad \times \left\{1-\varphi\left(B\left(x,\norm*{X_1-x}\right)\right)\right\}^{s-1}
            \left\{1-\varphi\left(B\left(x,\norm*{X_{c+1}^{\prime}-x}\right)\right)\right\}^{s-c-1}
        \end{aligned}
    \end{equation}
    The indicator is nonzero only if $\max\left(\norm*{X_{1} - x}, \norm*{X_{c+1}^{\prime} - x}\right) = \norm*{X_{1} - x}$, in which case
    \begin{equation}
        B(x, \delta) \subseteq B(x, \norm*{X_{1} - x})
    \end{equation}
    Hence
    \begin{equation}
        \begin{aligned}
             & \E\Bigg[
                \abs*{f(X_{1}) f(X_{c+1}^{\prime}) - f^{2}(x)}
            \frac{\Ex*{\kappa\left(x; Z_{1}, D_{[s]}\right)\kappa\left(x; Z_{c+1}^{\prime}, D^{\prime}_{[s]}\right) \given X_{1}, X_{c+1}^{\prime}}}{\Ex*{\kappa\left(x; Z_{1}, D_{[s]}\right)\kappa\left(x; Z_{c+1}^{\prime}, D^{\prime}_{[s]}\right)}} \\
             & \quad \left(1 - \1*{X_1, X_{c+1}^{\prime} \in B(x, \delta)}\right)
            \Bigg]                                                                                                                                                                                                                                                           \\
             & \overset{(\text{\cref{lem:expec_kernel_prod}})}{\leq}
            (2s-c)^{2}  \E\Big[
            \abs*{f(X_{1}) f(X_{c+1}^{\prime}) - f^{2}(x)}                                                                                                                                                                                                            \\
             & \qquad  \Ex*{\kappa\left(x; Z_{1}, D_{[s]}\right)\kappa\left(x; Z_{c+1}^{\prime}, D^{\prime}_{[s]}\right) \given X_{1}, X_{c+1}^{\prime}}
                \left(1 - \1*{X_1, X_{c+1}^{\prime} \in B(x, \delta)}\right)
            \Big]                                                                                                                                                                                                                                                            \\
             & \overset{(\text{\cref{lem:cond_expec_kernel_prod}})}{=}
            (2s-c)^{2}  \E\Big[
            \abs*{f(X_{1}) f(X_{c+1}^{\prime}) - f^{2}(x)}
            \1*{\norm*{X_{c+1}^{\prime} - x} \leq \norm*{X_{1} - x}}                                                                                                                                                                                                       \\
             & \qquad  \times \left\{1-\varphi\left(B\left(x,\norm*{X_1-x}\right)\right)\right\}^{s-1}
            \left\{1-\varphi\left(B\left(x,\norm*{X_{c+1}^{\prime}-x}\right)\right)\right\}^{s-c-1}                                                                                                                                                                   \\
             & \qquad \times
            \left(1 - \1*{X_1, X_{c+1}^{\prime} \in B(x, \delta)}\right)
            \Big]                                                                                                                                                                                                                                                            \\
             & \leq
            (2s-c)^{2}  \left\{1-\varphi\left(B\left(x, \delta \right)\right)\right\}^{s-1}                                                                                                                                                                                  \\
             & \qquad \times \E\Big[
                \abs*{f(X_{1}) f(X_{c+1}^{\prime}) - f^{2}(x)}  \1*{\delta < \norm*{X_{c+1}^{\prime} - x} \leq \norm*{X_{1} - x}}
            \Big]                                                                                                                                                                                                                                                            \\
             & \leq (2s-c)^{2}  \left\{1-\varphi\left(B\left(x, \delta \right)\right)\right\}^{s-1}
            \Ex*{\abs*{f(X_{1}) f(X_{c+1}^{\prime}) - f^{2}(x)}}
        \end{aligned}
    \end{equation}
    \newpage
    \begin{landscape}
        Similarly, for the third case set
        \begin{equation}
            K_{\Delta}
            \defeq
            \kappa\left(x; Z_{c+1}, D_{[s]}\right)
            \kappa\left(x; Z_{c+1}^{\prime}, D^{\prime}_{[s]}\right).
        \end{equation}
        Then
        \begin{equation}
            \begin{aligned}
                 & \E\Bigg[
                    \abs*{f(X_{c+1}) f(X_{c+1}^{\prime}) - f^{2}(x)}
                    \frac{\Ex*{K_{\Delta} \given X_{c+1}, X_{c+1}^{\prime}}}{\Ex*{K_{\Delta}}}
                    \left(1 - \1*{X_{c+1}, X_{c+1}^{\prime} \in B(x, \delta)}\right)
                \Bigg]                                                                                                                                                      \\
                 & \overset{(\text{\cref{lem:expec_kernel_prod}})}{\leq}
                (2s-c)^{2} \E\Big[
                    \abs*{f(X_{c+1}) f(X_{c+1}^{\prime}) - f^{2}(x)}
                    \Ex*{K_{\Delta} \given X_{c+1}, X_{c+1}^{\prime}}
                    \left(1 - \1*{X_{c+1}, X_{c+1}^{\prime} \in B(x, \delta)}\right)
                \Big]                                                                                                                                                       \\
                 & \overset{(\text{\cref{lem:cond_expec_kernel_prod}})}{=}
                (2s-c)^{2}
                \E\Big[\abs*{f(X_{c+1}) f(X_{c+1}^{\prime}) - f^{2}(x)}
                \left\{1-\varphi\left(B\left(x, \min\left(\norm*{X_{c+1} - x}, \norm*{X_{c+1}^{\prime}-x}\right)\right)\right)\right\}^{s-c-1}                \\
                 & \qquad  \times \left\{1-\varphi\left(B\left(x,\max\left(\norm*{X_{c+1} - x}, \norm*{X_{c+1}^{\prime}-x}\right)\right)\right)\right\}^{s-1}
                \left(1 - \1*{X_{c+1}, X_{c+1}^{\prime} \in B(x, \delta)}\right)
                \Big]                                                                                                                                                       \\
                 & \leq (2s-c)^{2}
                \E\Big[\abs*{f(X_{c+1}) f(X_{c+1}^{\prime}) - f^{2}(x)}
                \left\{1-\varphi\left(B\left(x, \min\left(\norm*{X_{c+1} - x}, \norm*{X_{c+1}^{\prime}-x}\right)\right)\right)\right\}^{s-c-1}                \\
                 & \qquad  \times \left\{1-\varphi\left(B\left(x,\delta\right)\right)\right\}^{s-1} \\
                 & \qquad  \times
                \left(1 - \1*{X_{c+1}, X_{c+1}^{\prime} \in B(x, \delta)}\right)
                \Big]                                                                                                                                                       \\
                 & \leq
                (2s-c)^{2}
                \left\{1-\varphi\left(B\left(x,\delta\right)\right)\right\}^{s-1} \\
                 & \qquad \times
                \Ex*{\abs*{f(X_{c+1}) f(X_{c+1}^{\prime}) - f^{2}(x)}}
            \end{aligned}
        \end{equation}
        The remaining expectation factors are finite:
        \begin{equation}
            \begin{aligned}
                \E\Big[\abs*{f^{2}(X) - f^{2}(x)}\Big]
                 & \leq \E\Big[f^{2}(X)\Big] + f^{2}(x)
                = \norm*{f}_{L_2}^{2} + f^{2}(x)
            \end{aligned}
        \end{equation}
        \begin{equation}
            \begin{aligned}
                 & \E\Big[\abs*{f(X) f(X^{\prime}) - f^{2}(x)}\Big]
                \leq \E\Big[\abs*{f(X) f(X^{\prime})}\Big] + f^{2}(x)                 \\
                 & \leq \E\Big[\abs*{f(X)} \abs*{f(X^{\prime})}\Big] + f^{2}(x)
                = \norm*{f}_{L^{1}}^{2} + f^{2}(x)
            \end{aligned}
        \end{equation}

        Since $f\in L^{2}(\calX)$ on a bounded domain, $\norm*{f}_{L^{1}}<\infty$.
        Hence
        \begin{equation}
            \begin{aligned}
                 & \E\Bigg[\abs*{f^{2}(X_{1}) - f^{2}(x)} (2s - c)  \E\left[\kappa\left(x; Z_{1}, D_{[s]}\right)\kappa\left(x; Z_{1}, D^{\prime}_{[s]}\right)
                    \left(1 - \1*{X_1 \in B(x, \delta)}\right) \middle| X_{1}
                    \right]
                \Bigg]
                \longrightarrow 0 \quad \text{as} \quad s \longrightarrow \infty
            \end{aligned}
        \end{equation}
        \begin{equation}
            \begin{aligned}
                 & \E\Bigg[
                    \abs*{f(X_{1}) f(X_{c+1}^{\prime}) - f^{2}(x)}
                    \frac{\Ex*{\kappa\left(x; Z_{1}, D_{[s]}\right)\kappa\left(x; Z_{c+1}^{\prime}, D^{\prime}_{[s]}\right) \given X_{1}, X_{c+1}^{\prime}}}{\Ex*{\kappa\left(x; Z_{1}, D_{[s]}\right)\kappa\left(x; Z_{c+1}^{\prime}, D^{\prime}_{[s]}\right)}} \left(1 - \1*{X_1, X_{c+1}^{\prime} \in B(x, \delta)}\right)
                    \Bigg]
                \longrightarrow 0 \quad \text{as} \quad s \longrightarrow \infty
            \end{aligned}
        \end{equation}
        \begin{equation}
            \begin{aligned}
                 & \E\Bigg[
                    \abs*{f(X_{c+1}) f(X_{c+1}^{\prime}) - f^{2}(x)}
                    \frac{\Ex*{\kappa\left(x; Z_{c+1}, D_{[s]}\right)\kappa\left(x; Z_{c+1}^{\prime}, D^{\prime}_{[s]}\right) \given X_{c+1}, X_{c+1}^{\prime}}}
                    {\Ex*{\kappa\left(x; Z_{c+1}, D_{[s]}\right)\kappa\left(x; Z_{c+1}^{\prime}, D^{\prime}_{[s]}\right)}} \\
                 & \qquad \times
                    \left(1 - \1*{X_{c+1}, X_{c+1}^{\prime} \in B(x, \delta)}\right)
                    \Bigg]
                \longrightarrow 0 \quad \text{as} \quad s \longrightarrow \infty
            \end{aligned}
        \end{equation}
        Combining these bounds, for large enough $s$ each of $\norm*{(A) - f^{2}(x)}$, $\norm*{(B) - f^{2}(x)}$, and $\norm*{(C) - f^{2}(x)}$ is bounded by $2\epsilon$.
        As $\epsilon$ was arbitrary this concludes the proof.
    \end{landscape}
\end{proof}

\begin{lemma}\label{lem:expec_kernel_prod_bound}\mbox{}\\*
    The following inequalities hold.
    \begin{equation}
        \begin{aligned}
             & \forall i \in [c] \; \forall j \in \{c+1, \dotsc, s\}:                                                    \\
             & \Ex*{\kappa\left(x; Z_{i}, D_{[s]}\right)\kappa\left(x; Z_{j}^{\prime}, D^{\prime}_{[s]}\right)}
            \leq \frac{1}{s(2s-c)}
        \end{aligned}
    \end{equation}
    \begin{equation}
        \begin{aligned}
             & \forall i,j \in \{c+1, \dotsc, s\}:                                                                       \\
             & \Ex*{\kappa\left(x; Z_{i}, D_{[s]}\right)\kappa\left(x; Z_{j}^{\prime}, D^{\prime}_{[s]}\right)}
            \leq \frac{2}{s(2s-c)}
        \end{aligned}
    \end{equation}
\end{lemma}

\begin{proof}[Proof of \cref{lem:expec_kernel_prod_bound}]\mbox{}\\*
    Both inequalities are immediate from the exact identities in \cref{lem:expec_kernel_prod}.
\end{proof}

    \subsection{DNN Kernel Expectations}

We now collect the expectation calculations for the single-scale DNN kernel and its first projection under the nonparametric regression setup.

\begin{lemma}[NPR - DNN Kernel Expectation]\label{lem:DNN_k_exp}\mbox{}\\*
	Let $x$ denote a point of interest.
	Then
	\begin{equation}
		\Ex*{h_s\left(x; D_{[s]}\right)}
		= \Ex*{Y_1 s \Ex*{\kappa\left(x; Z_1, D_{[s]}\right) \given X_{1}}}
		  \longrightarrow \mu\left(x\right) \quad \text{as} \quad s \longrightarrow \infty
	\end{equation}
\end{lemma}

\begin{proof}[Proof of \cref{lem:DNN_k_exp}]
	This result follows immediately from \cref{lem:dem13} and the following observation.
    \begin{equation}
		\begin{aligned}
			& \Ex*{Y_1 s \Ex*{\kappa\left(x; Z_1, D_{[s]}\right) \given X_{1}}}
			= \Ex*{\left(\mu\left(X_1\right) + \varepsilon_1\right) s \Ex*{\kappa\left(x; Z_1, D_{[s]}\right) \given X_{1}}} \\
			& = \Ex*{\left(\mu\left(X_1\right) + \Ex*{\varepsilon_1 \given X_1}\right) s \Ex*{\kappa\left(x; Z_1, D_{[s]}\right) \given X_{1}}} \\
			& = \Ex*{\mu\left(X_1\right) s \Ex*{\kappa\left(x; Z_1, D_{[s]}\right) \given X_{1}}} \\
			& \overset{\text{(\cref{lem:dem13})}}{\longrightarrow} \mu\left(x\right)
			\quad \text{as} \quad s \longrightarrow \infty
		\end{aligned}
	\end{equation}
\end{proof}

    \begin{lemma}[NPR - DNN H\'ajek Kernel Expectation]\label{lem:psi_s_1}\mbox{}\\*
	Let $z_1 = (x_1, y_1)$ denote a specific realization of $Z$ and $x$ denote a point of interest.
	Then
	\begin{equation}
		\begin{aligned}
			\psi_{s}^{1}\left(x; z_1\right)
			 & = \mu(x_1)\Ex*{\kappa\left(x; Z_1, D_{[s]}\right)\given X_1 = x_1} \\
			 & \quad + \varepsilon_1 \Ex*{\kappa\left(x; Z_1, D_{[s]}\right)\given X_1 = x_1} \\
			 & \quad + \Ex*{\sum_{i = 2}^{s} \kappa\left(x; Z_{i}, D_{[s]}\right) \mu(X_{i})\given X_1 = x_1}.
		\end{aligned}
	\end{equation}
\end{lemma}

\begin{proof}[Proof of \cref{lem:psi_s_1}]
	\begin{equation}
		\begin{aligned}
			& \psi_{s}^{1}\left(x; z_1\right)
			  = \Ex*{h_{s}\left(x; D_{[s]}\right) \given Z_1 = z_1}
			= \Ex*{\sum_{i = 1}^{s} \kappa\left(x; Z_{i}, D_{[s]}\right) Y_{i} \given Z_1 = z_1} \\
			 & = \E\Bigg[\left(\mu(x_1) + \varepsilon_1\right)\kappa\left(x; Z_1, D_{[s]}\right) \\
			 & \qquad\qquad\quad + \sum_{i = 2}^{s} \kappa\left(x; Z_{i}, D_{[s]}\right) \mu(X_{i})\;\Bigg|\; Z_1 = z_1 \Bigg] \\
			 & = \mu(x_1)\Ex*{\kappa\left(x; Z_1, D_{[s]}\right)\given X_1 = x_1} \\
			 & \qquad + \varepsilon_1 \Ex*{\kappa\left(x; Z_1, D_{[s]}\right)\given X_1 = x_1} \\
			 & \qquad + \Ex*{\sum_{i = 2}^{s} \kappa\left(x; Z_{i}, D_{[s]}\right) \mu(X_{i})\given X_1 = x_1}
		\end{aligned}
	\end{equation}
\end{proof}

    \subsection{DNN Kernel Variances \& Covariances}

We next bound the second moments and overlap covariances needed for the single-scale DNN H\'ajek-dominance argument.

\begin{lemma}[Adapted from~\cite{demirkaya_optimal_2024}]\label{lem:omega_s}\mbox{}\\*
    Let $D_{[s]} = \{Z_1, \dotsc, Z_{s}\}$ be a vector of i.i.d.\ random variables drawn from $P$.
    Furthermore, let
    \begin{equation}
        \Omega_{s}\left(x\right)
        = \Ex*{h_{s}^{2}\left(x; D_{[s]}\right)}.
    \end{equation}
    Then,
    \begin{equation}
        \Omega_{s}\left(x\right)
        = \Ex*{\left(\mu\left(X_1\right)+ \varepsilon_1\right)^2 s \Ex*{\kappa\left(x; Z_1, D_{[s]}\right) \given X_{1}}}
        \lesssim \mu^2(x) + \overline{\sigma}_{\varepsilon}^2 + o(1)
    \end{equation}
\end{lemma}

\begin{proof}[Proof of \cref{lem:omega_s}]\mbox{}\\*
    This result follows immediately from \cref{lem:dem13} and the following observation.
    \begin{equation}
        \begin{aligned}
             & \Omega_{s}\left(x\right)
            = \Ex*{h_{s}^{2}\left(x; D_{[s]}\right)}
            = \Ex*{\left(\sum_{i = 1}^{s}\kappa\left(x; Z_{i}, D_{[s]}\right)Y_{i}\right)^2}                                                                                      \\
             & = \Ex*{\sum_{i = 1}^{s}\sum_{j = 1}^{s}\left(\kappa\left(x; Z_{i}, D_{[s]}\right)\kappa\left(x; Z_{j}, D_{[s]}\right)Y_{i}Y_{j}\right)}
            = \Ex*{s \kappa\left(x; Z_{1}, D_{[s]}\right)Y_{1}^2}                                                                                                                 \\
             & = \Ex*{Y_{1}^2 s \Ex*{\kappa\left(x; Z_{1}, D_{[s]}\right) \given X_{1}}}
            = \Ex*{\left(\mu\left(X_1\right) + \varepsilon_1\right)^2 s \Ex*{\kappa\left(x; Z_1, D_{[s]}\right) \given X_{1}}}                                         \\
             & = \Ex*{\left(\mu^2\left(X_1\right) + 2\mu\left(X_1\right)\varepsilon_1 + \varepsilon_1^2\right)s \Ex*{\kappa\left(x; Z_1, D_{[s]}\right) \given X_{1}}} \\
             & = \E\left[\left(\mu^2\left(X_1\right) + 2\mu\left(X_1\right) \Ex*{\varepsilon_1 \given X_1} + \Ex*{\varepsilon_1^2 \given X_1}\right)
            s \Ex*{\kappa\left(x; Z_1, D_{[s]}\right) \given X_{1}}\right]                                                                                                      \\
             & = \Ex*{\left(\mu^2\left(X_1\right) +\sigma_{\varepsilon}^{2}(X_1)\right) s \Ex*{\kappa\left(x; Z_1, D_{[s]}\right) \given X_{1}}}                        \\
             & \overset{\text{(\cref{lem:dem13})}}{\longrightarrow} \mu^2\left(x\right) +\sigma_{\varepsilon}^{2}(x)
            \quad \text{as} \quad s \longrightarrow \infty
        \end{aligned}
    \end{equation}
    Furthermore, we have the following inequality.
    \begin{equation}
        \mu^2(x) + \sigma_{\varepsilon}^2(x) \leq \mu^2\left(x\right) + \overline{\sigma}_{\varepsilon}^{2}
    \end{equation}
    Thus, we obtain the desired result.
\end{proof}

\begin{lemma}\label{lem:omega_sc}\mbox{}\\*
    Let $D_{[s]} = \{Z_1, \dotsc, Z_{s}\}$ be a vector of i.i.d.\ random variables drawn from $P$.
    Let $D_{[s]}^{\prime} = \{Z_1, \dotsc, Z_{c}, Z_{c+1}^{\prime}, \dotsc,  Z_{s}^{\prime}\}$ where $Z_{c+1}^{\prime}, \dotsc,  Z_{s}^{\prime}$ are i.i.d.\ draws from $P$ that are independent of $D_{[s]}$.
    Furthermore, let
    \begin{equation}
        \Omega_{s}^{c}\left(x\right)
        = \E\left[h_{s}\left(x; D_{[s]}\right)
            h_{s}\left(x; D_{[s]}^{\prime}\right)\right].
    \end{equation}
    Then,
    \begin{equation}
        \Omega_{s}^{c}\left(x\right)
        \lesssim \mu^2(x) + \overline{\sigma}_{\varepsilon}^2 + o(1)
    \end{equation}
\end{lemma}

\begin{proof}[Proof of \cref{lem:omega_sc}]
    \begin{equation}
        \begin{aligned}
             & \Omega_{s}^{c}\left(x\right)
            = \E\left[h_{s}\left(x; D_{[s]}\right)
            h_{s}\left(x; D_{[s]}^{\prime}\right)\right]                                                                                                                               \\
             & = \E\left[
                \left(\sum_{i = 1}^{s}\kappa\left(x; Z_{i}, D_{[s]}\right)Y_{i}\right)
                \left(\sum_{j = 1}^{c}\kappa\left(x; Z_{j}, D_{[s]}^{\prime}\right)Y_{j}
                + \sum_{j = c+1}^{s}\kappa\left(x; Z_{j}^{\prime}, D_{[s]}^{\prime}\right)Y_{j}^{\prime}\right)
            \right]                                                                                                                                                                    \\
            %
             & = \underbrace{\Ex*{c \kappa\left(x; Z_{1}, D_{[s]}\right)\kappa\left(x; Z_{1}, D_{[s]}^{\prime}\right)Y_{1}^{2}}}_{(A)}                                        \\
             & \quad + 2  \underbrace{\Ex*{c(s-c) \kappa\left(x; Z_{1}, D_{[s]}\right)\kappa\left(x; Z_{c+1}^{\prime}, D_{[s]}^{\prime}\right)Y_{1}Y_{c+1}^{\prime}}}_{(B)}   \\
             & \quad + \underbrace{\Ex*{(s-c)^2 \kappa\left(x; Z_{c+1}, D_{[s]}\right)\kappa\left(x; Z_{c+1}^{\prime}, D_{[s]}^{\prime}\right)Y_{c+1}Y_{c+1}^{\prime}}}_{(C)}
        \end{aligned}
    \end{equation}
    Starting from this decomposition, we will analyze the terms one by one using \cref{lem:kernel_prod_dirac_convergence}.
    \begin{equation}
        \begin{aligned}
            (A)
             & = \Ex*{c\kappa\left(x; Z_{1}, D_{[s]}\right)\kappa\left(x; Z_{1}, D_{[s]}^{\prime}\right)Y_{1}^{2}}                                                                                          \\
             & = \frac{c}{2s-c}  \Ex*{\left(\mu(X_1) + \varepsilon_1\right)^2  (2s-c)  \Ex*{\kappa\left(x; Z_{1}, D_{[s]}\right)\kappa\left(x; Z_{1}, D_{[s]}^{\prime}\right) \given X_{1}}}     \\
             & = \frac{c}{2s-c}  \Ex*{\left(\mu^{2}(X_1) + \sigma_{\varepsilon}^{2}(X_1)\right)  (2s-c)  \Ex*{\kappa\left(x; Z_{1}, D_{[s]}\right)\kappa\left(x; Z_{1}, D_{[s]}^{\prime}\right) \given X_{1}}} \\
             & \overset{\text{(\cref{lem:kernel_prod_dirac_convergence})}}{\lesssim} \frac{c}{2s-c} \left(\mu^2(x) + \sigma_{\varepsilon}^{2}(x)\right) + o(1)
        \end{aligned}
    \end{equation}
    Similarly, we can find the following.
    \begin{equation}
        \begin{aligned}
            (B)
             & = \Ex*{c(s-c) \kappa\left(x; Z_{1}, D_{[s]}\right)\kappa\left(x; Z_{c+1}^{\prime}, D_{[s]}^{\prime}\right)Y_{1}Y_{c+1}^{\prime}}                                                                                                                                             \\
             & \overset{(\text{\cref{lem:expec_kernel_prod}})}{=} \frac{c(s-c)}{s(2s-c)}
            \E\Bigg[
            \left(\mu(X_1) + \varepsilon_{1}\right)\left(\mu(X_{c+1}^{\prime}) + \varepsilon_{c+1}^{\prime}\right)                                                                                                                                                                                   \\
             & \qquad \times \left. \frac{\Ex*{\kappa\left(x; Z_{1}, D_{[s]}\right)\kappa\left(x; Z_{c+1}^{\prime}, D_{[s]}^{\prime}\right) \given X_{1}, X_{c+1}^{\prime}}}{\Ex*{\kappa\left(x; Z_{1}, D_{[s]}\right)\kappa\left(x; Z_{c+1}^{\prime}, D_{[s]}^{\prime}\right)}}
            \right]                                                                                                                                                                                                                                                                                  \\
             & = \frac{c(s-c)}{s(2s-c)} \E\left[
                \mu(X_1)  \mu(X_{c+1}^{\prime})  \frac{\Ex*{\kappa\left(x; Z_{1}, D_{[s]}\right)\kappa\left(x; Z_{c+1}^{\prime}, D_{[s]}^{\prime}\right) \given X_{1}, X_{c+1}^{\prime}}}{\Ex*{\kappa\left(x; Z_{1}, D_{[s]}\right)\kappa\left(x; Z_{c+1}^{\prime}, D_{[s]}^{\prime}\right)}}
            \right]                                                                                                                                                                                                                                                                                  \\
             & \overset{(\text{\cref{lem:kernel_prod_dirac_convergence}})}{\lesssim}  \frac{c(s-c)}{s(2s-c)}  \mu^2(x) + o(1)
        \end{aligned}
    \end{equation}
    The third term can be asymptotically bounded in the following way.
    \begin{equation}
        \begin{aligned}
            (C)
             & = \Ex*{(s-c)^2 \kappa\left(x; Z_{c+1}, D_{[s]}\right)\kappa\left(x; Z_{c+1}^{\prime}, D_{[s]}^{\prime}\right)Y_{c+1}Y_{c+1}^{\prime}} \\
             & \overset{(\text{\cref{lem:expec_kernel_prod}})}{=}
            \frac{2(s-c)^2}{s(2s-c)}
            \E\Bigg[
                \mu(X_{c+1})  \mu(X_{c+1}^{\prime}) \\
             & \qquad\qquad\qquad \times
                \frac{\Ex*{\kappa\left(x; Z_{c+1}, D_{[s]}\right)\kappa\left(x; Z_{c+1}^{\prime}, D_{[s]}^{\prime}\right) \given X_{c+1}, X_{c+1}^{\prime}}}{\Ex*{\kappa\left(x; Z_{c+1}, D_{[s]}\right)\kappa\left(x; Z_{c+1}^{\prime}, D_{[s]}^{\prime}\right)}}
            \Bigg] \\
             & \overset{(\text{\cref{lem:kernel_prod_dirac_convergence}})}{\lesssim}
            \frac{2(s-c)^2}{s(2s-c)}  \mu^2(x) + o(1)
        \end{aligned}
    \end{equation}
    The coefficients $c/(2s-c)$, $c(s-c)/(s(2s-c))$, and $2(s-c)^2/(s(2s-c))$ are uniformly bounded over $1 \leq c \leq s-1$.
    The result of \cref{lem:omega_sc} follows immediately by summing up the asymptotic bounds for the individual terms.
\end{proof}

    \subsection{Single-Scale DNN Results}

This subsection establishes the single-scale DNN results used in the TDNN analysis.

\begin{lemma}[DNN selector profile]\label{lem:dnn_selector_profile}\mbox{}\\*
    Under \cref{asm:tdnn_variance_design}~\cref{asm:tdnn_design_density}, for each scale $s$ define
    \begin{equation}
        q_s(X_1)
        \defeq
        \Ex*{
            \kappa\left(x; Z_1, D_{[s]}\right)
            \given X_1
        }.
    \end{equation}
    Let $R_1 \defeq \norm*{X_1-x}_2$ and $F_x(t) \defeq \Pb*{\norm*{X-x}_2 \leq t}$.
    Then $U_1 \defeq F_x(R_1)$ is uniformly distributed on $(0,1)$,
    \begin{equation}\label{eq:dnn_selector_profile}
        q_s(X_1)
        =
        \left(1-U_1\right)^{s-1},
    \end{equation}
    and, for any $a>0$,
    \begin{equation}\label{eq:dnn_selector_moment}
        \Ex*{q_s(X_1)^a}
        =
        \frac{1}{a(s-1)+1}.
    \end{equation}
    More generally, for any $a,b \geq 0$ with $a+b>0$,
    \begin{equation}\label{eq:dnn_selector_cross_moment}
        \Ex*{q_s(X_1)^a q_t(X_1)^b}
        =
        \frac{1}{a(s-1)+b(t-1)+1}.
    \end{equation}
\end{lemma}
The quantity $q_s(X_1)$ is the conditional probability that, once $X_1$ is fixed, none of the remaining $s-1$ observations lands closer to $x$ than $X_1$.
This geometric interpretation is what later turns selector moments into the \(s^{-1}\) first-projection rate.

\begin{proof}
    Observation $1$ is the nearest neighbor among $s$ draws exactly when the
    remaining $s-1$ observations all lie outside $B(x,R_1)$.
    Hence \cref{eq:dnn_selector_profile} follows.
    The density condition implies that $U_1=F_x(R_1)$ is uniform on $(0,1)$,
    so the displayed moment identities follow by integrating powers of
    $1-U_1$ over the unit interval.
\end{proof}

\begin{lemma}[Single-scale DNN H\'ajek dominance]\label{lem:dnn_hajek_input}\mbox{}\\*
    Consider a data-generating process as outlined in \cref{asm:npr_dgp} and \cref{asm:tdnn_variance_design}.
    Let $s = o(n)$.
    Then the DNN estimator satisfies the asymptotic H\'ajek dominance condition.
    In particular,
    \begin{equation}\label{eq:dnn_hajek_input_rates}
        \zeta_s^s(x) \lesssim 1,
        \qquad
        \zeta_s^1(x) \asymp s^{-1}.
    \end{equation}
\end{lemma}

\begin{proof}
    By \cref{lem:omega_s},
    \begin{equation}
        \zeta_s^s(x)
        \leq
        \Omega_s(x)
        \lesssim 1.
    \end{equation}
    For the lower bound on the first-projection variance, use the selector
    profile $q_s$ from \cref{lem:dnn_selector_profile}.
    The decomposition in \cref{lem:psi_s_1} gives
    $h_s^{(1)}(x; Z_1) = q_s(X_1)\varepsilon_1 + b_s(X_1)$
    for an $X_1$-measurable remainder $b_s$.
    Since $\Ex{\varepsilon_1 \given X_1} = 0$, the law of total variance gives
    \begin{equation}
        \zeta_s^1(x)
        \geq
        \Ex*{q_s(X_1)^2\sigma_\varepsilon^2(X_1)}
        \geq
        \underline{\sigma}_{\varepsilon}^{2}\,
        \Ex*{q_s(X_1)^2}.
    \end{equation}
    \cref{eq:dnn_selector_moment} with $a=2$ gives
    $\Ex*{q_s(X_1)^2} = (2(s-1)+1)^{-1}$, and hence $\zeta_s^1(x) \gtrsim s^{-1}$.
    The cross-scale identity in \cref{eq:dnn_selector_cross_moment} is used later in the TDNN non-cancellation argument.
    The reverse bound $\zeta_s^1(x) \lesssim s^{-1}$ is the first-projection
    variance rate from~\cite[Supplement, Lemma~7, eqs.~(A.131)--(A.133)]{demirkaya_optimal_2024},
    adapted to the present heteroskedastic setup by using the uniform variance
    bound in \cref{asm:tdnn_variance_design}~\cref{asm:tdnn_design_variance}.
    Consequently,
    \begin{equation}
        \zeta_s^1(x) \asymp s^{-1}.
    \end{equation}
    Therefore,
    \begin{equation}
        \frac{s}{n}
        \left(
            \frac{\zeta_s^s(x)}{s\zeta_s^1(x)} - 1
        \right)
        =
        O\left(\frac{s}{n}\right)
        \longrightarrow 0
    \end{equation}
    as $s=o(n)$, which is exactly \cref{asm:hajek_dominance} for the single-scale DNN estimator.
\end{proof}
\begin{remark}
    The two rates in \cref{eq:dnn_hajek_input_rates} depend on separate properties of the kernel.
    The bound $\zeta_s^s(x) \lesssim 1$ follows from the fact that the nearest-neighbor selector picks at most one observation per subsample, so the full-kernel second moment is controlled by the second moment of whatever function is being selected.
    The rate $\zeta_s^1(x) \asymp s^{-1}$ follows from the selection probability: a fixed observation is the nearest neighbor among $s$ draws with probability of order $s^{-1}$, independently of the function value at that observation.
    Consequently, both rates hold for any kernel of the form $\kappa(x;\,Z_i,D_{[s]})\,f(Z_i)$ in which $f$ has finite second moment --- the regression response $Y$ is one instance, but the rates are the same whenever the same nearest-neighbor selector is applied to a square-integrable function of the observation.
\end{remark}

\begin{lemma}[Single-scale DNN verifies the row-wise $L^r$ condition]\label{lem:dnn_row_Lr}\mbox{}\\*
    Consider a data-generating process as outlined in \cref{asm:npr_dgp}, \cref{asm:tdnn_variance_design}, and \cref{asm:response_moment}.
    Then the single-scale DNN first projection satisfies \cref{asm:row_Lr}.
    More precisely, with $r_\eta \defeq 1 + \frac{1}{2}\min\{\eta,2\}$, where $\eta$ is from \cref{asm:response_moment},
    \begin{equation}
        \Ex*{
            \left(
                \frac{h_s^{(1)}(x; Z_1)^2}{\zeta_s^1(x)}
            \right)^{r_\eta}
        }
        \lesssim
        s^{r_\eta-1},
    \end{equation}
    so the ratio in \cref{asm:row_Lr} vanishes whenever $s=o(n)$.
\end{lemma}

\begin{proof}
    Write $r=r_\eta$ for the exponent fixed in the statement.
    The $b_s$ bound below uses the upper rate $\zeta_s^1(x) \lesssim s^{-1}$
    from \cref{lem:dnn_hajek_input}, which rests on the first-projection variance
    calculation of~\cite[Supplement, Lemma~7]{demirkaya_optimal_2024}.
    \emph{Selector moment.}
    For $q_s$ from \cref{lem:dnn_selector_profile}, \cref{eq:dnn_selector_moment}
    with $a=2r$ gives
    \begin{equation}
        \Ex*{q_s(X_1)^{2r}}
        =
        \frac{1}{2r(s-1)+1}
        \lesssim
        s^{-1}.
    \end{equation}
    By \cref{lem:psi_s_1}, the first projection decomposes as
    $h_s^{(1)}(x; Z_1) = q_s(X_1)\,\varepsilon_1 + b_s(X_1)$,
    where
    \begin{equation}
        b_s(u)
        \coloneq
        \mu(u)\,q_s(u)
        +
        \E\!\left[
            \sum_{i=2}^{s}\kappa\!\left(x; Z_i, D_{[s]}\right)\mu(X_i)
            \,\middle|\,
            X_1 = u
        \right]
        -
        \theta_s(x).
    \end{equation}
    \emph{Bias remainder.}
    Since $\calX$ is compact and $\mu$ is continuous under
    \cref{asm:tdnn_variance_design}~\cref{asm:tdnn_design_compact,asm:tdnn_design_regression}, write
    $M_\mu \defeq \sup_{v\in\calX}\abs*{\mu(v)} < \infty$.
    For the $b_s$ remainder: since the selector weights
    $\kappa(x; Z_i, D_{[s]})$ sum to one almost surely (\cref{lem:dem12}),
    \begin{equation*}
        \mu(u)\,q_s(u)
        +
        \Ex*{\sum_{i=2}^{s}\kappa\mu(X_i) \given X_1=u}
        =
        \Ex*{\sum_{i=1}^{s}\kappa(x;Z_i,D_{[s]})\mu(X_i) \given X_1=u}.
    \end{equation*}
    The right-hand side is bounded by $M_\mu$ in absolute value.
    Together with $\abs*{\theta_s(x)} \leq M_\mu$ (since $\theta_s(x)$ is itself an expectation of the same convex combination of $\mu$ values, bounded by $M_\mu$ via \cref{lem:dem12}), we have
    $\abs*{b_s(u)} \leq 2M_\mu$ uniformly in $u$.
    Because $\Ex{\varepsilon_1 \given X_1} = 0$ kills the cross term,
    \begin{equation}
        \zeta_s^1(x)
        =
        \Varb*{h_s^{(1)}(x; Z_1)}
        =
        \Ex*{q_s(X_1)^2\,\sigma_\varepsilon^2(X_1)}
        +
        \Ex*{b_s(X_1)^2}
        \geq
        \Ex*{b_s(X_1)^2},
    \end{equation}
    so $\Ex{b_s(X_1)^2} \leq \zeta_s^1(x) \lesssim s^{-1}$ by \cref{lem:dnn_hajek_input}.
    The $L^\infty$ bound then gives
    \begin{equation}
        \Ex*{\abs*{b_s(X_1)}^{2r}}
        \leq
        (2M_\mu)^{2(r-1)}\,\Ex*{b_s(X_1)^2}
        \lesssim
        s^{-1}.
    \end{equation}
    \emph{Conditional error moment.}
    By \cref{asm:response_moment}, the choice of $r=r_\eta$, and the
    boundedness of $\mu$ on $\calX$, there is a constant
    $C_{\varepsilon,r}<\infty$ such that
    \begin{equation}\label{eq:dnn_conditional_error_moment}
        \sup_{u\in\calX}
        \Ex*{\abs*{\varepsilon}^{2r}\given X=u}
        \leq
        C_{\varepsilon,r}.
    \end{equation}
    Since $q_s(X_1)$ is $X_1$-measurable, iterated expectation and
    \cref{eq:dnn_selector_moment,eq:dnn_conditional_error_moment} give
    \begin{equation}\label{eq:dnn_selector_error_moment}
        \begin{aligned}
            \Ex*{
                q_s(X_1)^{2r}\abs*{\varepsilon_1}^{2r}
            }
             & =
            \Ex*{
                q_s(X_1)^{2r}
                \Ex*{\abs*{\varepsilon_1}^{2r}\given X_1}
            } \\
             & \leq
            C_{\varepsilon,r}
            \Ex*{q_s(X_1)^{2r}}
            \lesssim
            s^{-1}.
        \end{aligned}
    \end{equation}
    Combining \cref{eq:dnn_selector_error_moment} with the $b_s$ bound and
    $\abs*{a+b}^{2r}\leq 2^{2r-1}\left(\abs*{a}^{2r}+\abs*{b}^{2r}\right)$
    yields
    \begin{equation}\label{eq:dnn_first_projection_moment}
        \Ex*{\abs*{h_s^{(1)}(x; Z_1)}^{2r}}
        \lesssim
        s^{-1}.
    \end{equation}
    By \cref{lem:dnn_hajek_input}, $\zeta_s^1(x) \asymp s^{-1}$.
    Consequently,
    \begin{equation}
        \Ex*{
            \left(
                \frac{h_s^{(1)}(x; Z_1)^2}{\zeta_s^1(x)}
            \right)^r
        }
        =
        \frac{\Ex*{\abs*{h_s^{(1)}(x; Z_1)}^{2r}}}
        {\left(\zeta_s^1(x)\right)^r}
        \lesssim
        s^{r-1}.
    \end{equation}
    Since $r>1$ and $s=o(n)$,
    \begin{equation}
        \frac{1}{n^{r-1}}
        \Ex*{
            \left(
                \frac{h_s^{(1)}(x; Z_1)^2}{\zeta_s^1(x)}
            \right)^r
        }
        \lesssim
        \left(\frac{s}{n}\right)^{r-1}
        \longrightarrow 0.
    \end{equation}
    This proves \cref{asm:row_Lr}.
\end{proof}

    \section{Extension to the TDNN Estimator}\label{sec:tdnn_hajek_dominance_rewrite}
    \subsection{Notation}

For the TDNN extension, we only need the first TDNN projection, its associated variance terms, and the effective-localizer notation.

\begin{equation}
    \psi_{\mathfrak{S}}^{1}(x; d_{1})
    =
    \Ex*{h_{\mathfrak{S}}\left(x; D_{[s_{2}]}\right) \given Z_{1} = d_{1}}
\end{equation}
Let
\begin{equation}
    \theta_{\mathfrak{S}}(x) \coloneq \Ex*{h_{\mathfrak{S}}\left(x; D_{[s_{2}]}\right)}
\end{equation}
denote the finite-sample TDNN mean, and set
\begin{equation}
    h_{\mathfrak{S}}^{(1)}\left(x; d_{1}\right)
    =
    \psi_{\mathfrak{S}}^{1}(x; d_{1}) - \theta_{\mathfrak{S}}(x).
\end{equation}
For the TDNN kernel, write
\begin{align}
    \zeta_{\mathfrak{S}}^{1}\left(x\right)
     & = \Varb*{h_{\mathfrak{S}}^{(1)}\left(x; Z_{1}\right)}                                                        \\
    \zeta_{\mathfrak{S}}^{s_2}\left(x\right)
     & = \Varb*{h_{\mathfrak{S}}\left(x; D_{[s_2]}\right)}.
\end{align}
The two-scale kernel also admits an observation-level representation that will be useful in the H\'ajek-dominance argument.
Define the effective TDNN weights by
\begin{equation}
    \begin{aligned}
        \tilde w_i\left(x; D_{[s_2]}\right)
         & \coloneq
        w_1^* \binom{s_2}{s_1}^{-1}
        \sum_{\ell \in L_{s_2,s_1}} \1*{i \in \ell}\kappa\left(x; Z_i, D_{\ell}\right) \\
         & \qquad + w_2^* \kappa\left(x; Z_i, D_{[s_2]}\right),
        \qquad i = 1,\dotsc,s_2.
    \end{aligned}
\end{equation}
Then
\begin{equation}
    h_{\mathfrak{S}}\left(x; D_{[s_2]}\right)
    =
    \sum_{i=1}^{s_2} \tilde w_i\left(x; D_{[s_2]}\right) Y_i.
\end{equation}
The induced effective TDNN localizer is
\begin{equation}
    \tilde\tau_{\mathfrak{S}}(u)
    \coloneq
    s_2 \Ex*{\tilde w_1\left(x; D_{[s_2]}\right) \given X_1 = u}.
\end{equation}

    \subsection{Two-Scale TDNN Argument}

This subsection establishes the two-scale TDNN result from the corresponding single-scale DNN bounds.
The two-scale kernel decomposes into an embedded $s_1$-scale DNN average inside an $s_2$-sample plus the ordinary $s_2$-scale DNN kernel, so both pieces can be analyzed through the single-scale results derived above.

In what follows, we use the first-projection objects $\psi_s^1$, $\psi_{\mathfrak S}^1$, $h_s^{(1)}$, and $h_{\mathfrak S}^{(1)}$, together with the variance terms $\zeta_s^1$, $\zeta_s^s$, $\zeta_{\mathfrak S}^1$, and $\zeta_{\mathfrak S}^{s_2}$.
By \cref{lem:dnn_hajek_input}, the corresponding single-scale DNN estimator satisfies \cref{asm:hajek_dominance}, with
\begin{equation}
    \zeta_s^s(x) \lesssim 1,
    \qquad
    \zeta_s^1(x) \asymp s^{-1},
\end{equation}
for every scale $s=o(n)$.

The first structural step is to separate the signed TDNN combination from the positive averaging operator over $s_1$-subsets.
Define
\begin{equation}\label{eq:embedded_dnn_average}
    \bar h_{s_1 \mid s_2}\left(x; D_{[s_2]}\right)
    \coloneq \binom{s_2}{s_1}^{-1}
    \sum_{\ell \in L_{s_2,s_1}} h_{s_1}\left(x; D_{\ell}\right).
\end{equation}
Then the TDNN kernel can be written as
\begin{equation}\label{eq:tdnn_embedded_decomp}
    h_{\mathfrak S}\left(x; D_{[s_2]}\right)
    = w_1^* \bar h_{s_1 \mid s_2}\left(x; D_{[s_2]}\right)
    + w_2^* h_{s_2}\left(x; D_{[s_2]}\right).
\end{equation}
Under \cref{asm:kernel_order_ratio}, the coefficient $w_1^*$ is negative and $w_2^*>1$, so the final TDNN observation-level weights are signed.
Accordingly, Jensen's inequality is applied to the positive uniform average in~\cref{eq:embedded_dnn_average} inside the decomposition \cref{eq:tdnn_embedded_decomp}, not to the signed TDNN combination itself.

\paragraph{Step 1: Control the embedded $s_1$-scale second moment.}
Because $\bar h_{s_1 \mid s_2}$ is the uniform average of the $s_1$-scale kernels over all $s_1$-subsets of a fixed $s_2$-sample, Jensen's inequality gives
\begin{equation}\label{eq:embedded_jensen}
    \bar h_{s_1 \mid s_2}^2\left(x; D_{[s_2]}\right)
    \leq
    \binom{s_2}{s_1}^{-1}
    \sum_{\ell \in L_{s_2,s_1}} h_{s_1}^2\left(x; D_{\ell}\right).
\end{equation}
Taking expectations and using exchangeability,
\begin{equation}\label{eq:embedded_second_moment_bound}
    \Ex*{\bar h_{s_1 \mid s_2}^2\left(x; D_{[s_2]}\right)}
    \leq
    \Ex*{h_{s_1}^2\left(x; D_{[s_1]}\right)}
    \lesssim 1,
\end{equation}
where the last step is precisely the single-scale DNN input at scale $s_1$.
The averaging bound in \cref{eq:embedded_jensen} smooths rather than amplifies the embedded $s_1$-scale contribution: averaging over many $s_1$-subsets cannot have larger second moment than the average of their individual second moments.
Combining~\cref{eq:embedded_second_moment_bound} with the analogous DNN bound at scale $s_2$, and using that $w_1^*$ and $w_2^*$ stay bounded under \cref{asm:kernel_order_ratio}, we obtain
\begin{equation}\label{eq:tdnn_full_kernel_strategy}
    \begin{aligned}
        \zeta_{\mathfrak S}^{s_2}(x)
         & \leq \Ex*{h_{\mathfrak S}^2\left(x; D_{[s_2]}\right)} \\
         & \lesssim
        \left(w_1^*\right)^2
        \Ex*{\bar h_{s_1 \mid s_2}^2\left(x; D_{[s_2]}\right)}
        + \left(w_2^*\right)^2
        \Ex*{h_{s_2}^2\left(x; D_{[s_2]}\right)} \\
         & \lesssim 1.
    \end{aligned}
\end{equation}
Conceptually, this shows that the embedded $s_1$-piece can be controlled directly at its own scale before it is recombined with the $s_2$-piece.

\paragraph{Step 2: Identify the first projection of the embedded $s_1$-scale term.}
The key combinatorial identity is simpler at the projection level than at the raw second-moment level.
Conditionally on $Z_1 = z_1$, a uniformly drawn $s_1$-subset of an $s_2$-sample contains observation $1$ with probability $s_1/s_2$.
If the subset contains $1$, its conditional law matches the usual $s_1$-scale DNN setup with one observation fixed.
If the subset does not contain $1$, its conditional expectation is just the finite-sample DNN mean $\theta_{s_1}(x) = \Ex*{h_{s_1}\left(x; D_{[s_1]}\right)}$.
Thus, with
\begin{equation}
    \psi_{s_1 \mid s_2}^{1}(x; z_1)
    \coloneq
    \Ex*{\bar h_{s_1 \mid s_2}\left(x; D_{[s_2]}\right) \given Z_1 = z_1},
\end{equation}
we obtain the exact identity
\begin{equation}\label{eq:embedded_projection_identity}
    \psi_{s_1 \mid s_2}^{1}(x; z_1)
    = \frac{s_1}{s_2}\psi_{s_1}^{1}(x; z_1)
    + \left(1-\frac{s_1}{s_2}\right)\theta_{s_1}(x),
\end{equation}
Subtracting \(\theta_{s_1}(x)\) from \cref{eq:embedded_projection_identity} gives
\begin{equation}\label{eq:embedded_first_projection_identity}
    \bar h_{s_1 \mid s_2}^{(1)}(x; z_1)
    \coloneq \psi_{s_1 \mid s_2}^{1}(x; z_1) - \theta_{s_1}(x)
    = \frac{s_1}{s_2} h_{s_1}^{(1)}(x; z_1).
\end{equation}
This identity separates the selection step from the nearest-neighbor competition at scale $s_1$.
Observation $1$ must first be selected into the relevant $s_1$-subset and only then can it affect the ordinary DNN nearest-neighbor competition at that scale.

\paragraph{Step 3: Close the TDNN first-projection rate.}
By linearity of conditional expectation and~\cref{eq:embedded_first_projection_identity},
\begin{equation}\label{eq:tdnn_first_projection_linear_combination}
    h_{\mathfrak S}^{(1)}(x; z_1)
    = w_1^* \frac{s_1}{s_2} h_{s_1}^{(1)}(x; z_1)
    + w_2^* h_{s_2}^{(1)}(x; z_1).
\end{equation}
For each scale $s$, \cref{lem:psi_s_1} yields the exact decomposition
\begin{equation}\label{eq:dnn_first_projection_qb}
    h_s^{(1)}(x; Z_1)
    =
    q_s(X_1)\varepsilon_1
    +
    b_s(X_1),
\end{equation}
where
\begin{equation}
    \begin{aligned}
        q_s(u)
         & \coloneq
        \Ex*{
            \kappa\left(x; Z_1, D_{[s]}\right)
            \given X_1 = u
        }, \\
        b_s(u)
         & \coloneq
        \mu(u)\,q_s(u) \\
         & \quad +
        \Ex*{
            \sum_{i = 2}^{s} \kappa\left(x; Z_i, D_{[s]}\right)\mu(X_i)
            \given X_1 = u
        } \\
         & \quad - \Ex*{h_s\left(x; D_{[s]}\right)}.
    \end{aligned}
\end{equation}
Here $b_s(X_1)$ is $X_1$-measurable, so
\begin{equation}\label{eq:dnn_first_projection_variance_decomp}
    \zeta_s^1(x)
    =
    \Ex*{q_s(X_1)^2 \sigma_\varepsilon^2(X_1)}
    +
    \Ex*{b_s(X_1)^2}
\end{equation}
because $\Ex*{\varepsilon_1 \given X_1} = 0$ kills the cross term.
Combining \cref{eq:tdnn_first_projection_linear_combination} with \cref{eq:dnn_first_projection_qb}, define
\begin{equation}\label{eq:tdnn_first_projection_qb}
    h_{\mathfrak S}^{(1)}(x; Z_1)
    =
    \tilde q_{\mathfrak S}(X_1)\varepsilon_1
    +
    \tilde b_{\mathfrak S}(X_1),
\end{equation}
where
\begin{equation}
    \tilde q_{\mathfrak S}(u)
    \coloneq
    a_{\mathfrak S} q_{s_1}(u)
    +
    b_{\mathfrak S} q_{s_2}(u),
    \qquad
    \tilde b_{\mathfrak S}(u)
    \coloneq
    a_{\mathfrak S} b_{s_1}(u)
    +
    b_{\mathfrak S} b_{s_2}(u),
\end{equation}
with
\begin{equation}
    a_{\mathfrak S}
    \coloneq
    w_1^* \frac{s_1}{s_2},
    \qquad
    b_{\mathfrak S}
    \coloneq
    w_2^*.
\end{equation}
Therefore,
\begin{equation}\label{eq:tdnn_first_projection_variance_decomp}
    \zeta_{\mathfrak S}^{1}(x)
    =
    \Ex*{\tilde q_{\mathfrak S}(X_1)^2 \sigma_\varepsilon^2(X_1)}
    +
    \Ex*{\tilde b_{\mathfrak S}(X_1)^2}.
\end{equation}

This is the point where bias correction and variance part company: the coefficients are chosen to cancel leading deterministic bias terms, but they must still leave a stochastic first-order signal on the \(s_2^{-1/2}\) scale.
\begin{lemma}[Uniform non-cancellation of the TDNN selector coefficient]\label{lem:tdnn_selector_non_cancellation}\mbox{}\\*
    Suppose \cref{asm:kernel_order_ratio} holds.
    Then there exist constants $0 < c_q < C_q < \infty$ such that
    \begin{equation}
        \frac{c_q}{s_2}
        \leq
        \Ex*{\tilde q_{\mathfrak S}(X_1)^2}
        \leq
        \frac{C_q}{s_2}.
    \end{equation}
\end{lemma}

\begin{proof}
    Write $\rho \coloneq s_1 / s_2$.
    Since
    \begin{equation}
        w_1^*
        =
        \frac{1}{1-\rho^{-2/k}}
        =
        -\frac{\rho^{2/k}}{1-\rho^{2/k}},
        \qquad
        w_2^*
        =
        1-w_1^*
        =
        \frac{1}{1-\rho^{2/k}},
    \end{equation}
    we have
    \begin{equation}
        a_{\mathfrak S}
        =
        -\frac{\rho^{1+2/k}}{1-\rho^{2/k}},
        \qquad
        b_{\mathfrak S}
        =
        \frac{1}{1-\rho^{2/k}}.
    \end{equation}
    Under \cref{asm:kernel_order_ratio}, both coefficients are uniformly bounded and $b_{\mathfrak S} \geq 1$.

    The selector-profile moments from \cref{lem:dnn_selector_profile} give
    \begin{equation}
        \begin{aligned}
            \Ex*{q_{s_1}(X_1)^2}
             & =
            \frac{1}{2s_1-1}, \\
            \Ex*{q_{s_2}(X_1)^2}
             & =
            \frac{1}{2s_2-1}, \\
            \Ex*{q_{s_1}(X_1) q_{s_2}(X_1)}
             & =
            \frac{1}{s_1+s_2-1}.
        \end{aligned}
    \end{equation}
    The following matrix records the \(L^2\) geometry of the two selector profiles \(q_{s_1}\) and \(q_{s_2}\).
    Its determinant lower bound says that, under the scale-separation condition, these profiles do not become asymptotically collinear, so the signed TDNN coefficients cannot cancel the stochastic first-projection signal.
    Let
    \begin{equation}
        \Gamma_{\mathfrak S}
        \coloneq
        s_2
        \begin{pmatrix}
            \Ex*{q_{s_1}(X_1)^2}      & \Ex*{q_{s_1}(X_1) q_{s_2}(X_1)} \\
            \Ex*{q_{s_1}(X_1) q_{s_2}(X_1)} & \Ex*{q_{s_2}(X_1)^2}
        \end{pmatrix}.
    \end{equation}
    Then
    \begin{equation}
        s_2 \Ex*{\tilde q_{\mathfrak S}(X_1)^2}
        =
        \begin{pmatrix}
            a_{\mathfrak S} & b_{\mathfrak S}
        \end{pmatrix}
        \Gamma_{\mathfrak S}
        \begin{pmatrix}
            a_{\mathfrak S} \\
            b_{\mathfrak S}
        \end{pmatrix}.
    \end{equation}
    Moreover,
    \begin{equation}
        \det\left(\Gamma_{\mathfrak S}\right)
        =
        \frac{s_2^2 (s_2-s_1)^2}{(2s_1-1)(2s_2-1)(s_1+s_2-1)^2}.
    \end{equation}
    Using \cref{asm:kernel_order_ratio},
    \begin{equation}
        s_1 \geq \mathfrak c s_2,
        \qquad
        s_2-s_1 \geq \mathfrak c s_2,
        \qquad
        s_1 \leq (1-\mathfrak c)s_2,
    \end{equation}
    so
    \begin{equation}
        \det\left(\Gamma_{\mathfrak S}\right)
        \geq
        \frac{\mathfrak c^2}{16(1-\mathfrak c)}.
    \end{equation}
    Since
    \begin{equation}
        \operatorname{tr}\left(\Gamma_{\mathfrak S}\right)
        =
        \frac{s_2}{2s_1-1}
        +
        \frac{s_2}{2s_2-1}
        \leq
        \frac{1}{\mathfrak c}
        +
        1,
    \end{equation}
    the eigenvalues of $\Gamma_{\mathfrak S}$ are uniformly bounded above and away from zero.
    In particular, there exist constants $0 < \underline\lambda \leq \overline\lambda < \infty$ such that
    \begin{equation}
        \underline\lambda I_2
        \preceq
        \Gamma_{\mathfrak S}
        \preceq
        \overline\lambda I_2.
    \end{equation}
    Because $a_{\mathfrak S}$ and $b_{\mathfrak S}$ are uniformly bounded and $b_{\mathfrak S} \geq 1$, there exists $C_v < \infty$ such that
    \begin{equation}
        1
        \leq
        a_{\mathfrak S}^2 + b_{\mathfrak S}^2
        \leq
        C_v.
    \end{equation}
    Combining the last two displays gives
    \begin{equation}
        \underline\lambda
        \leq
        s_2 \Ex*{\tilde q_{\mathfrak S}(X_1)^2}
        \leq
        \overline\lambda C_v,
    \end{equation}
    which is exactly the claimed bound.
\end{proof}

By \cref{eq:dnn_first_projection_variance_decomp} and \cref{lem:dnn_hajek_input},
\begin{equation}
    \Ex*{b_s(X_1)^2}
    \leq
    \zeta_s^1(x)
    \lesssim
    s^{-1}.
\end{equation}
Since \cref{asm:kernel_order_ratio} keeps $a_{\mathfrak S}$ and $b_{\mathfrak S}$ uniformly bounded and $s_1 \asymp s_2$,
\begin{equation}
    \Ex*{\tilde b_{\mathfrak S}(X_1)^2}
    \lesssim
    \Ex*{b_{s_1}(X_1)^2}
    +
    \Ex*{b_{s_2}(X_1)^2}
    \lesssim
    s_2^{-1}.
\end{equation}
Because \cref{asm:tdnn_variance_design}~\cref{asm:tdnn_design_variance} places $\sigma_\varepsilon^2(\cdot)$ on the compact support $\calX$ as a strictly positive continuous function, there exist constants
\begin{equation}
    0 < \underline{\sigma}_\varepsilon^2
    \leq
    \sigma_\varepsilon^2(u)
    \leq
    \overline{\sigma}_\varepsilon^2
    < \infty,
    \qquad u \in \calX.
\end{equation}
Combining these bounds with \cref{eq:tdnn_first_projection_variance_decomp} and \cref{lem:tdnn_selector_non_cancellation}, we obtain
\begin{equation}\label{eq:tdnn_first_projection_target_rate}
    \zeta_{\mathfrak S}^{1}(x) \asymp s_2^{-1}.
\end{equation}

\paragraph{Step 3A: Verify the row-wise $L^r$ condition.}
With \cref{eq:tdnn_first_projection_target_rate} established, the TDNN first projection satisfies the row-wise square-LLN by the same moment calculation used for the single-scale DNN first projection.
\begin{lemma}[TDNN verification of the row-wise $L^r$ condition]\label{lem:tdnn_row_Lr}\mbox{}\\*
    Consider a data-generating process as outlined in \cref{asm:npr_dgp}, \cref{asm:tdnn_variance_design}, and \cref{asm:response_moment}, and suppose \cref{asm:kernel_order_ratio} holds.
    Then the TDNN first projection satisfies \cref{asm:row_Lr}.
    More precisely, with $r_\eta \defeq 1 + \frac{1}{2}\min\{\eta,2\}$, where $\eta$ is from \cref{asm:response_moment},
    \begin{equation}
        \Ex*{
            \left(
                \frac{h_{\mathfrak S}^{(1)}(x; Z_1)^2}{\zeta_{\mathfrak S}^{1}(x)}
            \right)^{r_\eta}
        }
        \lesssim
        s_2^{r_\eta-1},
    \end{equation}
    so the ratio in \cref{asm:row_Lr} vanishes whenever $s_2=o(n)$.
\end{lemma}

\begin{proof}
    Write $r=r_\eta$ for the exponent fixed in the statement.
    By the exact identity \cref{eq:tdnn_first_projection_linear_combination} and Minkowski's inequality in $L^{2r}$,
    \begin{equation}
        \begin{aligned}
            \left(
                \Ex*{\abs*{h_{\mathfrak S}^{(1)}(x; Z_1)}^{2r}}
            \right)^{1/(2r)}
             & \leq
            \abs*{w_1^*}\frac{s_1}{s_2}
            \left(
                \Ex*{\abs*{h_{s_1}^{(1)}(x; Z_1)}^{2r}}
            \right)^{1/(2r)} \\
             & \quad +
            \abs*{w_2^*}
            \left(
                \Ex*{\abs*{h_{s_2}^{(1)}(x; Z_1)}^{2r}}
            \right)^{1/(2r)}.
        \end{aligned}
    \end{equation}
    The bounded-ratio condition keeps $w_1^*$ and $w_2^*$ uniformly bounded, and \cref{eq:dnn_first_projection_moment} gives
    \begin{equation}
        \Ex*{\abs*{h_{s_k}^{(1)}(x; Z_1)}^{2r}}
        \lesssim
        s_k^{-1},
        \qquad k \in \{1,2\}.
    \end{equation}
    Since $s_1 \leq s_2$,
    \begin{equation}
        \left(
            \frac{s_1}{s_2}
        \right)
        s_1^{-1/(2r)}
        \leq
        s_2^{-1/(2r)}.
    \end{equation}
    Therefore,
    \begin{equation}
        \Ex*{\abs*{h_{\mathfrak S}^{(1)}(x; Z_1)}^{2r}}
        \lesssim
        s_2^{-1}.
    \end{equation}
    Dividing by \cref{eq:tdnn_first_projection_target_rate} yields
    \begin{equation}
        \Ex*{
            \left(
                \frac{h_{\mathfrak S}^{(1)}(x; Z_1)^2}{\zeta_{\mathfrak S}^{1}(x)}
            \right)^r
        }
        =
        \frac{\Ex*{\abs*{h_{\mathfrak S}^{(1)}(x; Z_1)}^{2r}}}
        {\left(\zeta_{\mathfrak S}^{1}(x)\right)^r}
        \lesssim
        s_2^{r-1}.
    \end{equation}
    Since $r>1$ and $s_2=o(n)$,
    \begin{equation}
        \frac{1}{n^{r-1}}
        \Ex*{
            \left(
                \frac{h_{\mathfrak S}^{(1)}(x; Z_1)^2}
                     {\zeta_{\mathfrak S}^{1}(x)}
            \right)^r
        }
        \lesssim
        \left(\frac{s_2}{n}\right)^{r-1}
        \longrightarrow 0.
    \end{equation}
    This proves \cref{asm:row_Lr}.
\end{proof}

\paragraph{Step 4: Conclude H\'ajek dominance.}
Combining \cref{eq:tdnn_full_kernel_strategy,eq:tdnn_first_projection_target_rate}, the TDNN kernel satisfies \cref{asm:hajek_dominance} because
\begin{equation}
    \frac{s_2}{n}
    \left(
        \frac{\zeta_{\mathfrak S}^{s_2}(x)}{s_2 \zeta_{\mathfrak S}^{1}(x)} - 1
    \right)
    \longrightarrow 0
\end{equation}
whenever $s_2=o(n)$.
This follows from the same variance scaling as in the single-scale DNN case.
The full kernel remains of constant order, while the first projection is diluted by competition among $s_2$ candidate observations.

\paragraph{Summary.}
At the level of ideas, the TDNN bound uses three ingredients beyond the single-scale DNN case.
The first is the Jensen bound for the embedded $s_1$-scale average.
The second is the projection identity~\cref{eq:embedded_first_projection_identity}, which turns the first-order TDNN analysis into a scaled combination of single-scale DNN first projections.
The third is the exact stochastic-coefficient decomposition~\cref{eq:tdnn_first_projection_qb} together with the uniform non-cancellation bound in \cref{lem:tdnn_selector_non_cancellation}.
Together, these ingredients reduce the TDNN analysis to the single-scale DNN bounds plus the selector non-cancellation step.

\begin{remark}[Multiscale extension]
The preceding argument should extend to a $K$-scale bias-corrected combination with kernel orders $1 \leq s_1 < \dotsb < s_K$.
The open ingredient is a non-cancellation lower bound for the effective localizer $\tau_{\mathbf{S},x}^{\mathrm{eff}}(u) \coloneq \sum_m a_m (s_m/s_K) \tau_{s_m,x}(u)$: one needs to show that the bias-correction coefficients $a_m$, while deliberately chosen to cancel deterministic bias terms, do not simultaneously annihilate the first-order stochastic signal.
This is the only step without a direct analogue in the two-scale case and would require a separate combinatorial argument.
The remaining ingredients extend directly: the full-kernel second moment inherits the $O(1)$ bound from the Jensen argument applied scale by scale; and the projection identity
\begin{equation}
    \bar h_{s_m \mid s_K}^{(1)}(x; z_1)
    =
    \frac{s_m}{s_K} h_{s_m}^{(1)}(x; z_1)
\end{equation}
follows from the embedded-average structure, giving each active scale a $s_K^{-1}$ first-projection variance contribution.
If the non-cancellation statement can be established, H\'ajek dominance at rate $s_K = o(n)$ would follow without further structural changes to the proof.
\end{remark}

\end{appendix}



\bibliographystyle{imsart-number_fixed} 
\bibliography{bibliography.bib}       

\end{document}